\documentclass[10pt]{amsart}
\usepackage[margin=1in]{geometry}
\usepackage{hyperref}
\hypersetup{
    colorlinks,
    citecolor=blue,
    filecolor=black,
    linkcolor=black,
    urlcolor=black
}
\usepackage{amsmath, amsthm, amssymb, url, color, graphicx, esint, subcaption}
\usepackage[foot]{amsaddr}

\usepackage[final]{showkeys}

\newcommand{\tz}{\tilde{z}}
\newcommand{\tw}{\tilde{w}}
\newcommand{\ii}{\mathbf i}
\newcommand{\eig}{\text{eig}}
\newcommand{\Id}{\operatorname{Id}}
\newcommand{\RR}{\mathbb{R}}
\newcommand{\CC}{\mathbb{C}}

\newcommand{\ZZ}{\mathbb{Z}}
\newcommand{\NN}{\mathbb{N}}
\newcommand{\Tr}{\operatorname{Tr}}

\newcommand{\EE}{\mathbb{E}}
\newcommand{\eps}{\varepsilon}

\newcommand{\bI}{\mathbf{1}}

\newcommand{\wmu}{\wtilde{\mu}}

\newcommand{\supp}{\operatorname{supp}}
\newcommand{\diag}{\text{diag}}
\newcommand{\wtilde}{\widetilde}

\newcommand{\cB}{\mathcal{B}}

\newcommand{\cH}{\mathcal{H}}
\newcommand{\gl}{\mathfrak{gl}}

\newcommand{\wlambda}{\wtilde{\lambda}}
\newcommand{\Haar}{\text{Haar}}

\newcommand{\wrho}{\wtilde{\rho}}

\newcommand{\ww}{\wtilde{w}}

\newcommand{\we}{\wtilde{e}}
\newcommand{\np}{\wtilde{p}}
\newcommand{\FF}{\mathcal{F}}
\newcommand{\Cov}{\text{Cov}}
\newcommand{\cK}{\mathcal{K}}
\newcommand{\cF}{\mathcal{F}}

\newcommand{\cN}{\mathcal{N}}
\newcommand{\fp}{\mathfrak{p}}
\newcommand{\fCov}{\mathfrak{Cov}}
\newcommand{\wz}{\wtilde{z}}
\newcommand{\wa}{\wtilde{a}}
\newcommand{\wb}{\wtilde{b}}
\newcommand{\wB}{\wtilde{B}}

\DeclareMathOperator*{\Sym}{Sym}
\newcommand{\halpha}{\hat{\alpha}}
\newcommand{\hR}{\hat{R}}

\newcommand{\OO}{\mathcal{O}}
\newcommand{\vol}{\text{vol}}
\newcommand{\GT}{\text{GT}}
\newcommand{\wt}{\wtilde{t}}
\newcommand{\ws}{\wtilde{s}}
\newcommand{\hchi}{\hat{\chi}}
\newcommand{\wPsi}{\wtilde{\Psi}}

\theoremstyle{definition}

\newtheorem{theorem}{Theorem}[section]
\newtheorem{corr}[theorem]{Corollary}
\newtheorem{lemma}[theorem]{Lemma}
\newtheorem{prop}[theorem]{Proposition}
\newtheorem{define}[theorem]{Definition}
\newtheorem{assump}[theorem]{Assumption}
\newtheorem*{remark}{Remark}

\numberwithin{equation}{section}
\numberwithin{figure}{section}

\begin{document}\sloppy
\title{Gaussian fluctuations for products of random matrices}
\date{\today}
\author{Vadim Gorin}
\address{V.G.: Department of Mathematics\\ University of Wisconsin-Madison\\ Madison, WI \\ and Institute for Information Transmission Problems of Russian Academy of Sciences\\ Moscow, Russia}
\email{vadicgor@gmail.com}

\author{Yi Sun}
\address{Y.S.: Department of Statistics\\ University of Chicago\\ Chicago, IL}
\email{yisun@statistics.uchicago.edu}

\begin{abstract}
  We study global fluctuations for singular values of $M$-fold products of several right-unitarily invariant $N \times N$ random matrix ensembles.  As $N \to \infty$, we show the fluctuations of their height functions converge to an explicit Gaussian field, which is log-correlated for $M$ fixed and has a white noise component for $M \to \infty$ jointly with $N$.  Our technique centers on the study of the multivariate Bessel generating functions of these spectral measures, for which we prove a central limit theorem for global fluctuations via certain conditions on the generating functions.  We apply our approach to a number of ensembles, including square roots of Wishart, Jacobi, and unitarily invariant positive definite matrices with fixed spectrum, using a detailed asymptotic analysis of multivariate Bessel functions to verify the necessary conditions.
\end{abstract}

\maketitle

\tableofcontents

\section{Introduction}

\subsection{Overview of the results}

Let $Y_N^1, \ldots, Y_N^M$ be i.i.d.~$N \times N$ random matrices which are right-unitarily invariant, and let $Y_N := Y_N^1 \cdots Y_N^M$ be their product.  The (squared) singular values $\mu_1^N \geq \cdots \geq \mu_N^N > 0$ of $Y_N$ occur classically in the study of ergodic theory of non-commutative random walks, and the corresponding Lyapunov exponents have a limit as $M \to \infty$ by Oseledec's multiplicative ergodic theorem (see \cite{Ose68, Rag79}).  These Lyapunov exponents for various ensembles of matrices have been the object of extensive study in the dynamical systems literature, beginning with the pioneering result of Furstenberg-Kesten in \cite{FK60} for matrices with positive entries.  In an applied context, singular values for similar models have appeared in the study of disordered systems in statistical physics as described in \cite{CPV93}, in the study of polymers as in \cite{CMR17}, and in the study of dynamical isometry for deep neural networks as the Jacobians of randomly initialized networks (see \cite{SMG14, PSG17, PSG18, CPS18, XBSSP18, TWJTN, LQ18}). 

The goal of this work is to study the global fluctuations of the squared singular values $\mu^N$. The distribution of $\mu^N$ depends only on the distribution of $X^i_N := (Y^i_N)^* Y^i_N$, and we study it for a variety of different distributions for $X^i_N$, including Wishart matrices, Jacobi matrices, and unitarily invariant positive definite matrices with fixed spectrum.  In each of these cases, we study the normalized log-spectrum
\[
\lambda^N_i := \frac{1}{M} \log \mu_i^N 
\]
via the height function
\[
\cH_N(t) := \#\{\lambda^N_i \leq t\}.
\]
We study limit shapes and fluctuations for the height function in two limit regimes, one where
$N \to \infty$ with $M$ fixed, and one where $N, M \to \infty$ simultaneously. Note that in
the main text we will use a slightly different notation for the log-spectrum and height function
for $M$ fixed which removes the normalization by $M$, but we keep this uniform notation
in the introduction for clarity of exposition.

When $M$ is fixed, results from free probability of Voiculescu and Nica-Speicher in \cite{Voi87, NS97} imply that the empirical measure $d\lambda^N$ of $\lambda^N_i$ converges to a deterministic measure $d\lambda^\infty$, which implies that $\cH_N(t)$ concentrates around a deterministic limit shape, a result we refer to as a law of large numbers.  In Theorems \ref{thm:finite-uni-product}, \ref{thm:jacobi-single}, and \ref{thm:ginibre-single} and Corollary \ref{corr:finite-uni-density}, we show that the fluctuations of the height function around its mean converge to explicit Gaussian fields.  In Corollary \ref{corr:log-density}, we show that these fields are log-correlated under a technical condition on the smoothness of the limit shape.  Namely, this means that for polynomials $f, g$, we have
\[
\lim_{N \to \infty} \Cov\left(\int (\cH_N(t) - \EE[\cH_N(t)]) f(t) dt, \int (\cH_N(s) - \EE[\cH_N(s)]) g(s) ds\right) = \int \int K(t, s) f(t) g(s) dt ds
\]
for an explicit covariance kernel $K(t, s)$ which satisfies
\[
K(t, s) = - \frac{1}{2\pi^2} \log|t - s| + O(1)
\]
for $|t - s| \to 0$.

When $M, N \to \infty$ simultaneously, the $\lambda_i^N$ are the Lyapunov exponents studied by Newman and Isopi-Newman in \cite{New86, New84, IN} and mentioned by Deift in \cite{Dei17}.  Again, their empirical measure converges to the deterministic measure 
\[
- \frac{e^{-z}}{S_{\wmu}'(S_{\wmu}^{-1}(e^{-z}))} \bI_{[-\log S_{\wmu}(-1), -\log S_{\wmu}(0)]} dz,
\]
where $S_{\wmu}$ is the limiting $S$-transform of the spectral measure of $(Y^i_N)^* Y^i_N$; we give a more detailed treatment in Sections \ref{sec:lya-fixed} and \ref{sec:lya-jg}.  This yields concentration of the height function around a deterministic limit shape.  In Theorems \ref{thm:lya}, \ref{thm:jac-lya}, and \ref{thm:gin-lya}, we show that the rescaled fluctuations of the height function around its mean
\[
M^{1/2} \Big(\cH_N(t) - \EE[\cH_N(t)]\Big)
\]
again converge to explicit Gaussian fields.  However, in these cases, we find that the field has a white noise component, meaning that for $|t - s| \to 0$ the covariance satisfies
\[
K(t, s) = \delta(t - s) + O(1),
\]
where $\delta$ denotes the Dirac delta function.  In particular, as $M \to \infty$, we see a transition from a log-correlated Gaussian field to a Gaussian field with a white noise component.  For this result, the relative growth rates of $N$ and $M$ are not important; in particular, we observe the same white noise component both for $M$ growing to infinity much faster than $N$ and much slower than $N$. This behavior is very different from that seen in recent work of Akemann-Burda-Kieburg and Liu-Wang-Wang in \cite{ABK18, LWW18}, where the local limit for the singular values of the products of Ginibre matrices depends on the ratio $M/N$.

\subsection{Comparison with fluctuations for sums of random matrices}

We may naturally compare our results to analogous results for sums of random matrices.  For $i = 1, \ldots, M$, let $A_N^i = U_i A U_i^*$ be $N \times N$ matrices with $U_i$ i.i.d.~Haar unitary and $A$ deterministic diagonal with spectrum $\mu^N$.  Define
\[
X_{N, M}^{\text{add}} := A_N^1 + \cdots + A_N^M
\]
and the height function of the spectrum $\lambda^{N, M}$ of $X_{N, M}^{\text{add}}$ by $\cH_{N, M}(t) := \#\{\lambda^{N, M}_i \leq t\}$.  If $M$ is finite and fixed and the empirical measures of $\mu^N$ converge to a measure $d\mu$, results from free probability show that $\cH_{N, M}(t)$ converges as $N \to \infty$ to a deterministic limit shape.  It is further known that the random variables
\[
\{\cH_{N, M}(t) - \EE[\cH_{N, M}(t)]\}_{t \in \RR}
\]
converge to an explicit log-correlated Gaussian field given by Collins-Mingo-\'Sniady-Speicher in \cite{CMSS} and Pastur-Vasilchuk in \cite{PV07}.

On the other hand, if we take $M \to \infty$, then we have the convergence 
\[
\frac{1}{M} X_{N, M}^{\text{add}} \to \frac{1}{N} \Big(\sum_{i = 1}^N \mu_i^N\Big) \cdot \text{Id}
\]
of the average $\frac{1}{M} X_{N, M}^{\text{add}}$ of the matrices $A_N^i$ to a deterministic matrix.  In contrast to the multiplicative case, this averaging retains only the first moment of the empirical measure of the original spectrum.  Furthermore, by the ordinary central limit theorem, as $M \to \infty$ with $N$ fixed, the random matrix
\[
\frac{1}{\sqrt{M}} \Big(X_{N, M}^{\text{add}} - \EE[X_{N, M}^{\text{add}}]\Big)
\]
converges in distribution to a Gaussian random matrix which we verify in Appendix \ref{sec:sum-fluc} to have distribution
\[
\sqrt{\frac{N}{(N - 1)(N + 1)} \cdot \left[\frac{1}{N} \sum_{i = 1}^N (\mu_i^N)^2 - \frac{1}{N^2} \Big(\sum_{i = 1}^N \mu_i^N\Big)^2\right]} \cdot \text{GUE}_{N, \Tr = 0},
\]
where $\text{GUE}_{N, \Tr = 0} \overset{d} = X - \Tr(X) \cdot \Id_N$ for $X \overset{d} = \text{GUE}_N$ being an $N \times N$ matrix drawn from the Gaussian Unitary Ensemble.  The results of Johansson in \cite{Joh98} on the GUE then suggest that as $N \to \infty$ slowly compared with $M$, fluctuations of the corresponding height function remain a log-correlated Gaussian field.

One may ask why white noise shows up for $M \to \infty$ in the multiplicative setting but not in the additive setting.  One intuitive explanation comes from comparing the decompositions
\begin{align*}
  X_{N, M}^{\text{add}} &= \EE[X_{N, M}^{\text{add}}] + \Big(X_{N, M}^{\text{add}} - \EE[X_{N, M}^{\text{add}}]\Big)\\
  \log X_{N, M}^{\text{mult}} &= \EE[\log X_{N, M}^{\text{mult}}] + \Big(\log X_{N, M}^{\text{mult}} - \EE[\log X_{N, M}^{\text{mult}}]\Big)
\end{align*}
in the additive and multiplicative settings, where $X_{N, M}^{\text{mult}} := (Y_N^1 \cdots Y_N^M)^*(Y_N^1 \cdots Y_N^M)$.

In the additive setting, note that $\EE[X_{N, M}^{\text{add}}]$ is a constant multiple of the identity, and therefore the $\lfloor t N \rfloor^\text{th}$ eigenvalue of $X_{N, M}^{\text{add}}$ admits the approximate representation
\[
\lambda_{\lfloor t N\rfloor}^{\text{add}} \sim M C_1 + \sqrt{M} C_2 \cdot \gamma_{\lfloor t N \rfloor}
\]
for constants $C_1$ and $C_2$, where $\gamma_i$ is the $i^\text{th}$ eigenvalue of $\frac{1}{\sqrt{N}} \text{GUE}_{N, \Tr = 0}$.  In particular, all fluctuations of the height function come from fluctuations of the height function for $X_{N, M}^{\text{add}} - \EE[X_{N, M}^{\text{add}}]$.

In contrast, in the multiplicative setting, $\EE[\log X_{N, M}^{\text{mult}}]$ has non-trivial spectrum, and the ordinary central limit theorem does not directly apply to $\log X_{N, M}^{\text{mult}} - \EE[\log X_{N, M}^{\text{mult}}]$.  Together, these two effects imply the analysis of the additive case does not apply in the multiplicative setting.  It would be interesting to find other geometric qualitative arguments explaining the appearance of the white noise component in the products without relying on the exact formulas that we use. We will not pursue this direction here, only mentioning that the viewpoint of Reddy in \cite{Red16} for fixed $N$ seems relevant for such a potential development.

Another heuristic explanation for the appearance of the white noise and for the difference between additive and multiplicative cases can be obtained by comparing our model to Dyson Brownian Motion (DBM) started from a deterministic initial condition\footnote{We thank Maurice Duits for communicating this point of view to us.}.  In this approach, we should treat the number of terms or factors as inverse time, $t=M^{-1}$. For addition, such an identification is possible because of the invariance of the Brownian motion with respect to the time inversion $B_{1/t} \overset{d} = t^{-1} B_t$, $t\ge 0$.  For multiplication, this is more speculative.  However, for the special case when the factors are drawn from the square Ginibre ensemble (i.e. with i.i.d.~Gaussian entries), one can compare the contour integral formulas for the correlation kernel for products given by Akemann-Kieburg-Wei and Kuijlaars-Zhang in \cite{AKW13, KZ14} with the formulas given by Br\'ezin-Hikami in \cite{BH96, BH97} for the $\beta=2$ Dyson Brownian Motion started from deterministic initial condition and notice their striking similarity upon making the identification $t=M^{-1}$.  In addition, in this case Akemann-Burda-Kieburg found in \cite{ABK18} that when $N, M \to \infty$ with $N / M$ limiting to a value in $(0, \infty)$, the local statistics for products of square Ginibre matrices coincide with those of Dyson Brownian motion started from an evenly spaced initial condition as studied by Johansson in \cite{Joh04}.

For general products of matrices, this suggests that we should start the DBM from the initial condition given by the Lyapunov exponents. The transition between white noise (small $t$, corresponding to $M\to \infty$) and log-correlated (finite $t$, corresponding to finite $M$) statistics was studied in detail for $\beta=2$ DBM by Duits--Johansson in \cite{DJ18}, for general $\beta$ DBM by Huang-Landon in \cite{HL16}, and for a finite-temperature version of the GUE ensemble by Johansson-Lambert in \cite{JL18}. In contrast, for the sums, the initial condition for the DBM is a diagonal matrix with all equal eigenvalues; hence, the global fluctuations are the same as those of the GUE, and we see log-correlated structure for the local statistics at all positive times.  It would be extremely interesting to find a conceptual explanation for this analogy between products of matrices and Dyson Brownian Motion.

\subsection{Central limit theorems for multivariate Bessel generating functions}

Our technique is based on the study of multivariate Bessel generating functions for log-spectral measures, which are continuous versions of the Schur generating function defined and studied by Bufetov-Gorin in \cite{BG15L, BG16, BG17}.  Recall that for sets of variables $a = (a_1, \ldots, a_N)$ and $b = (b_1, \ldots, b_N)$, the multivariate Bessel function is defined by
\[
\cB(a, b) := \Delta(\rho) \frac{\det(e^{a_i b_j})_{i, j = 1}^N}{\Delta(a) \Delta(b)},
\]
where $\rho = (N - 1, \ldots, 0)$ and $\Delta(a) := \prod_{1 \leq i < j \leq N} (a_i - a_j)$ denotes the Vandermonde determinant.  For an $N$-tuple $\chi_N = (\chi_{N, 1} \geq \cdots \geq \chi_{N, N})$, the multivariate Bessel generating function of a measure $d\mu_N(x)$ on $N$-tuples $x = (x_1 \geq \cdots \geq x_N)$ with respect to $\chi_N$ is defined by
\[
\phi_{\chi, N}(s) := \int \frac{\cB(s, x)}{\cB(\chi_N, x)} d\mu_N(x).
\]
We say that the measure $d\mu_N$ is $\chi$-smooth if the defining integral of $\phi_{\chi, N}(s)$ converges absolutely and uniformly on a neighborhood of $[\chi_{N, N}, \chi_{1, N}]^N$.  For a sequence of $N$-tuples $\chi_N \in \{x \in \RR^N \mid x_1 \geq \cdots \geq x_N\}$ for which $\frac{1}{N} \sum_{i = 1}^N \delta_{\chi_N / N} \to d\chi$ for a measure $d\chi$, we define a family of CLT-appropriate measures for $\chi_N$ by the following condition, which governs the behavior of their multivariate Bessel generating functions near $s = \chi_N$. For a subset $I \subset \{1, \ldots, N\}$, define
\[
\phi_{\chi, N}^I(s) := \phi_{\chi, N}(s)|_{s_j = \chi_{N, j}, j \notin I},
\]
where we note that $\phi_{\chi, N}^I(s)$ is a function of $(s_i)_{i \in I}$ alone.  For a function $f(s)$ on $\CC^N$ and some $r \in \CC^N$, we will adopt the notation that $\partial_{r_i} f(r) = \frac{\partial}{\partial s_i} f(s) \Big|_{s = r}$. 

{\renewcommand{\thetheorem}{\ref{def:lln-app}}
\begin{define}
  We say that the measures $d\mu_N$ are LLN-appropriate for $\chi_N$ if they are $\chi_N$-smooth and there exist a compact set $V_\chi \subset \RR$ and a holomorphic function $\Psi$ on an open complex neighborhood $U$ of $V_\chi$ such that $\chi_{N, i}/N \in V_{\chi}$ for all $N$ and for each fixed integer $k > 0$, uniformly in $I$ with $|I| = k$, $i \in I$, and $r^I \in U^{|I|}$, we have
  \[
  \lim_{N \to \infty} \left|\frac{1}{N}\partial_{r_i} [\log \phi_{\chi, N}^I(r N)] - \Psi(r_i)\right| = 0.
  \]
\end{define}
\addtocounter{theorem}{-1}}

{\renewcommand{\thetheorem}{\ref{def:clt-app}}
\begin{define}
We say that measures $d\mu_N$ are CLT-appropriate for $\chi_N$ if they are LLN-appropriate for $\chi_N$ with
respect to some function $\Psi$ and sets $V_\chi \subset \RR$ with neighborhood $U$ and there exists a
holomorphic function $\Lambda$ on $U^2$ such that for each fixed integer $k > 0$, uniformly in $I$ with $|I| = k$, $r^I \in U^{|I|}$,
and distinct $i, j \in I$, we have
\[
\lim_{N \to \infty} \left| \partial_{r_i} \partial_{r_j}[ \log \phi_{\chi, N}^I(rN)] - \Lambda(r_i, r_j) \right| = 0.
\]
\end{define}
\addtocounter{theorem}{-1}}

Our first technical tools are Theorems \ref{thm:lln} and \ref{thm:clt} giving a law of large numbers and central limit theorem for the moments $p_k := \int x^k d\mu$ of CLT-appropriate measures.  Define the Cauchy transform of $d\chi$ by
\[
\Xi(u) := \int \frac{1}{u - x} d\chi(x).
\]

{\renewcommand{\thetheorem}{\ref{thm:lln}}
\begin{theorem}[Law of large numbers] 
If the measures $d\mu_N$ are LLN-appropriate for $\chi_N$, then we have the following convergence in probability
\begin{equation*}
\lim_{N \to \infty} \frac{1}{N} p_k(x) = \lim_{N \to \infty} \frac{1}{N} \EE[p_k(x)] = \fp_k := \frac{1}{k + 1} \oint \Big(\Xi(u) + \Psi(u)\Big)^{k + 1} \frac{du}{2 \pi \ii},
\end{equation*}
where the $u$-contour encloses $V_\chi$ and lies within $U$.  In addition, the random measures $\frac{1}{N} \sum_{i = 1}^N \delta_{x_i}$ with $x$ distributed according to $d\mu_N$ converge in probability to a deterministic compactly supported measure $d\mu$ with $\int x^k d\mu(x) = \fp_k$. 
\end{theorem}
\addtocounter{theorem}{-1}}

{\renewcommand{\thetheorem}{\ref{thm:clt}}
\begin{theorem}
If the measures $d\mu_N$ are CLT-appropriate for $\chi_N$, then the collection of random variables
\[
\{p_k(x) - \EE[p_k(x)]\}_{k \in \NN}
\]
converges in the sense of moments to the Gaussian vector with zero mean and covariance
\begin{multline*}
\lim_{N \to \infty} \Cov\Big(p_k(x), p_l(x)\Big) = \fCov_{k, l} := \oint \oint \Big(\Xi(u) + \Psi(u)\Big)^l \Big(\Xi(w) + \Psi(w)\Big)^k\\ \left(\frac{1}{(u - w)^2} + \Lambda(w, u)\right)\frac{dw}{2\pi \ii} \frac{du}{2\pi \ii},
\end{multline*}
where the $u$ and $w$-contours enclose $V_\chi$ and lie within $U$, and the $u$-contour is contained inside the $w$-contour.
\end{theorem}
\addtocounter{theorem}{-1}}

Theorems \ref{thm:lln} and \ref{thm:clt} are parallel to results of \cite{BG16} for Schur generating functions. However, their proofs are significantly different. In \cite{BG16}, the Schur functions are normalized by their evaluation at $\chi_N = (1, \ldots, 1)$, while we allow general $N$-tuples $\chi_N$; in particular, for matrix products, we will choose $\chi_N = \rho$.  Due to this feature, adapting the proofs requires several new ideas.  In particular, our analysis hinges on the asymptotics of derivatives of $\log \phi_{\chi, N}(s)$ in at most two variables $s_i, s_j$ evaluated at the point $s = \chi_N$.  In our setting, such derivatives acquire a dependence on the choice of indices $i, j$ from the fact that $\chi_{N, i}$ may be non-constant, which did not exist in the setting of \cite{BG16}.  Handling this in the asymptotics gives rise to the integrals along contours surrounding $V_{\chi}$ in Theorem \ref{thm:clt}.

We derive a similar statement in Theorem \ref{thm:clt-many} for the case where the form of the multivariate Bessel function changes with $N$.  Let $M, N \to \infty$ simultaneously, and suppose there exists a sequence of measures $d\lambda_N'$ with multivariate Bessel generating function $\phi_{\chi, N}(s)^M$. Let $d\lambda_N$ be the pushfoward of $d\lambda_N'$ under the map $\lambda' \mapsto \frac{1}{M} \lambda'$.  We obtain a similar law of large numbers and central limit theorem for the moments of $d\lambda_N$, though we require a different scaling.

{\renewcommand{\thetheorem}{\ref{thm:lln-many}}
\begin{theorem} 
  If $d\mu_N$ is LLN-appropriate for $\chi_N$, then if $x$ is distributed according to $d\lambda_N$, in probability we have
  \[
  \lim_{N \to \infty} \frac{1}{N} p_k(x) = \lim_{N \to \infty} \frac{1}{N} \EE[p_k(x)] = \fp_k' := \oint \Xi(u) \Psi(u)^k \frac{du}{2\pi \ii},
  \]
where the $u$-contour encloses $V_\chi$ and lies within $U$.  In addition, the random measures $\frac{1}{N} \sum_{i = 1}^N \delta_{x_i}$ with $x$ distributed according to $d\lambda_N$ converge in probability to a deterministic compactly supported measure $d\lambda$ with $\int x^k d\lambda(x) = \fp_k'$. 
\end{theorem}
\addtocounter{theorem}{-1}}

{\renewcommand{\thetheorem}{\ref{thm:clt-many}}
\begin{theorem}
If the measures $d\mu_N$ are CLT-appropriate for $\chi_N$, then if $x$ is distributed according to $d\lambda_N$, the collection of random variables
\[
\{M^{1/2}(p_k(x) - \EE[p_k(x)])\}_{k \in \NN}
\]
converges in the sense of moments to the Gaussian vector with zero mean and covariance
\begin{multline*}
\lim_{N \to \infty} \Cov\Big(M^{1/2}p_k(x), M^{1/2}p_l(x)\Big)\\ = kl \oint \oint \Xi(u) \Xi(w) \Psi(u)^{k - 1} \Psi(w)^{l - 1} \Lambda(u, w) \frac{du}{2\pi\ii} \frac{dw}{2\pi\ii} + kl \oint \Xi(u) \Psi(u)^{k + l - 2} \Psi'(u) \frac{du}{2\pi\ii},
\end{multline*}
where the $u$ and $w$-contours enclose $V_\chi$ and lies within $U$, and the $u$-contour is contained inside the $w$-contour.
\end{theorem}
\addtocounter{theorem}{-1}}

\subsection{Asymptotics of multivariate Bessel generating functions for matrix products}

To apply Theorems \ref{thm:clt} and \ref{thm:clt-many} to products of random matrices, we choose $\chi_N = \rho = (N - 1, \ldots, 0)$ as the $N$-tuples for our multivariate Bessel generating functions.  For a positive definite matrix $X$, denote by $\phi_X(s)$ the multivariate Bessel generating function of its log-spectral measure for the $N$-tuple $\rho$.  Here, it is the log-spectral measure instead of the spectral measure itself which will behave nicely with respect to multivariate Bessel generating functions; moreover, considering the log-spectral measure is necessary in our later study of Lyapunov exponents.

We then use the key property shown in Lemma \ref{lem:mult} that if $X_1 = Y_1^* Y_1$ and $X_2 = Y_2^* Y_2$ are unitarily invariant random matrices, then the multivariate Bessel generating function for $X_3 = (Y_1 Y_2)^* (Y_1 Y_2)$ is given by
\[
\phi_{X_3}(s) = \phi_{X_1}(s) \phi_{X_2}(s).
\]
We conclude that the multivariate Bessel generating function of the log-spectral measure for $X_N := (Y_N^1 \cdots Y_N^M)^*(Y_N^1 \cdots Y_N^M)$ is given by
\[
\phi_N(s) = \psi_X(s)^M,
\]
where $\psi_X(s)$ is the generating function of the log-spectral measure of $X_N^1 := (Y_N^1)^*Y_N^1$. 

Our second key step is an asymptotic analysis in Theorems \ref{thm:single-asymp} and \ref{thm:multi-asymp} of the ratio $\frac{\cB(s, \lambda)}{\cB(\rho, \lambda)}$ near $s = \rho$, which allow us to check in Theorem \ref{thm:unitary-inv} that the measures we study are CLT-appropriate for $\rho$.  This analysis hinges on the following new integral representation of this ratio at $s = \mu_{a, b} := (a, N - 1, \ldots, \widehat{b}, \ldots, 0)$ motivated by the main integral formula in \cite{GP}, where $\widehat{b}$ means that $b$ is omitted from the expression.

{\renewcommand{\thetheorem}{\ref{thm:mvb-int}}
\begin{theorem}
For any $a \in \CC$ and $b \in \{N - 1, \ldots, 0\}$, we have that
\[
\frac{\cB(\mu_{a, b}, \lambda)}{\cB(\rho, \lambda)} = (-1)^{N - b + 1} \frac{\Gamma(N - b) \Gamma(b + 1)\Gamma(a - N + 1)(a - b)}{\Gamma(a + 1)} \oint_{\{e^{\lambda_i}\}} \frac{d z}{2\pi \ii} \oint_{\{0, z\}} \frac{dw}{2 \pi \ii} \cdot \frac{z^a w^{-b - 1}}{z - w}\cdot \prod_{i = 1}^N \frac{w- e^{\lambda_i}}{z - e^{\lambda_i}},
\]
where the $z$-contour encloses the poles at $e^{\lambda_i}$ and avoids the negative real axis and the $w$-contour encloses both $0$ and the $z$-contour.
\end{theorem}
\addtocounter{theorem}{-1}}

In the previous works \cite{GP, Cue18, Cue18b, Cue18c} of Gorin-Panova and Cuenca, representations of normalized symmetric polynomials as single contour integrals were used for the asymptotic analysis. We were not able to find such a representation for the normalized version of the multivariate Bessel functions that we need in this text. Nevertheless, it turns out that we can increase the dimension of the integration cycle and instead use a double contour integral, which remains well-suited for asymptotic analysis.  We also remark that our double contour integral in Theorem \ref{thm:mvb-int} has certain similarities with Okounkov's representation of the correlation kernel for the Schur measure in \cite[Theorem 2]{Oko01}.

\subsection{Multivariate Bessel generating functions and the Cholesky decomposition}

When applied to the log-spectral measure of a positive definite Hermitian unitarily invariant random matrix $X$, the multivariate Bessel generating function for $\rho$ has an interpretation in terms of the Cholesky decomposition of $X$.  Recall that the Cholesky decomposition of a positive definite Hermitian matrix $X$ is a factorization $X = R^* R$ of $X$ into the product of an upper triangular matrix $R$ and its conjugate transpose.  In Proposition \ref{prop:mvb-cholesky} we relate the multivariate Bessel generating function $\phi_X(s)$ for the log-spectral measure of $X$ to the Mellin transform of the diagonal elements of $R$ via
\begin{equation} \label{eq:mvb-r}
\phi_X(s) = \EE\left[\prod_{k = 1}^N R_{kk}^{2 (s_k - \rho_k)}\right].
\end{equation}
This is a translation of \cite[Lemma 5.3]{KK16} of Kieburg-Kosters to our language, and we view it as an analogue of the fact that the distribution of any unitarily-invariant random matrix $X$ is determined by the joint distribution of the diagonal entries $(X_{kk})_{1 \leq k \leq N}$ which was recently applied by Matsumoto-Novak in \cite{MN18} to the study of sums of random matrices, among other things, and which is implicit in the works \cite{OV96, Gor14} of Olshanski-Vershik and Gorin.

We may apply this connection to give an intuitive explanation for the multiplicative property of multivariate Bessel generating functions over products of matrices.  If $X_1 = Y_1^* Y_1$ and $X_2 = Y_2^* Y_2$ are independent unitarily invariant random matrices with Cholesky decompositions $X_1 = (R^1)^* R^1$ and $X_2 = (R^2)^* R^2$, then for $X_3 = (Y_1 Y_2)^* Y_1 Y_2$ with Cholesky decomposition $X_3 = (R^3)^* R^3$, we give in Proposition \ref{prop:r-mult} a direct geometric argument that
\[
R^3_{kk} \overset{d} = R^1_{kk} \cdot R^2_{kk}.
\]
Combined with (\ref{eq:mvb-r}), this gives another proof that $\phi_{X_3}(s) = \phi_{X_1}(s) \phi_{X_2}(s)$.

We also apply this result to give a geometric interpretation of Voiculescu's $S$-transform.   Let $d\lambda$ be a compactly supported measure on $(0, \infty)$, and let $\lambda^N$ be $N$-tuples whose empirical measures converge weakly to $d\lambda$.  Then we have the following expression for the $S$-transform of $d\lambda$ in terms of the limiting Cholesky decompositions of unitarily invariant random matrices with fixed spectra $\lambda^N$.

{\renewcommand{\thetheorem}{\ref{corr:s-interp}}
\begin{corr}
  Let $X_N$ be the $N \times N$ unitarily invariant Hermitian random matrix with spectrum $\lambda^N$.  For $t \in [0, 1]$,  the log-$S$-transform of the measure $d\lambda$ is given by
  \[
 -\log S_{d\lambda}(t - 1) = \lim_{N \to \infty} \EE\left[ 2 \log R_{\lfloor t N \rfloor, \lfloor t N \rfloor}\right],
 \]
 where $X_N = R^* R$ is the Cholesky decomposition of $X_N$. 
\end{corr}
\addtocounter{theorem}{-1}}

\subsection{Relation to the literature}

Our results relate to several different lines of work in the literature.  In many cases, we are able to prove more general results than those previously known in the same settings, though some of the papers we mention are focused on different applications.  We now summarize these connections, beginning with relations to previous results in the random matrix theory literature.
\begin{itemize}
\item Global fluctuations for sums and products of finitely many unitarily-invariant random matrices with fixed spectrum were previously studied in some cases using techniques from the Stieltjes transform, free probability, and Schwinger-Dyson equations.  For sums, a central limit for global fluctuations was proven using Stieltjes transform techniques in the works \cite{PV07, PS11} of Pastur-Vasilchuk and Pastur-Shcherbina and follows from the theory of second-order freeness in free probability (see the papers \cite{MS06, MSS07, CMSS} of Mingo-Speicher, Mingo-\'Sniady-Speicher, and Collins-Mingo-\'Sniady-Speicher) as described in \cite[Chapter 5]{MS17}.  A limiting case of our methods recovers the case of sums of random matrices; this approach was followed by Bufetov-Gorin in \cite[Section 9.4]{BG16} and shown to recover the results of \cite{PS11}.

For products, a central limit theorem was shown using Stieltjes transform techniques for the product of two unitary matrices with fixed spectrum by Vasilchuk in \cite{Vas16}.  For products of positive definite matrices, second-order freeness yields a CLT without an expression for the covariance via the results of Collins-Mingo-\'Sniady-Speicher in \cite[Theorem 7.9 and Theorem 8.2]{CMSS} and with non-explicit combinatorial formula for the covariance via the results of Arizmendi-Mingo in \cite{AM18}.  In \cite{GN15}, Guionnet-Novak show that all polynomial linear statistics are asymptotically Gaussian using the Schwinger-Dyson equations, though no explicit formula for the covariance is given and the same techniques do not extend to the log-spectrum. Finally, the recent paper \cite{COR18} of Coston-O'Rourke gives a CLT for linear statistics of eigenvalues of products of Wigner matrices. From this perspective, our Theorem \ref{thm:finite-uni-product} yields the first explicit formula for the covariance of the CLT for products of positive definite matrices.  Simultaneously, this is the first result showing that this covariance grows logarithmically on short scales.

\item Global fluctuations for some classes of matrix models amenable to our techniques have been studied extensively in the literature.  For Wigner, Wishart, and Jacobi ensembles, it was shown in the general $\beta$ setting in the works \cite{DP12, Jia13, Bor14, Bor14b, BG15, DP18} of Dumitriu-Paquette, Jiang, Borodin, and Borodin-Gorin that the multilevel eigenvalue structure of these ensembles has Gaussian free field fluctuations.  For separable covariance matrices, namely matrices of the form $AXB$ with $A, B$ deterministic and $X$ rectangular Ginibre, a CLT for global fluctuations was shown by Bai-Li-Pan in \cite{BLP16} via Stieltjes transform techniques.  Our Theorems \ref{thm:jacobi-single} and \ref{thm:ginibre-single} recover the single level versions of these results.  Finally, for a product of finitely many complex Ginibre matrices, the paper \cite{KZ14} of Kuijlaars-Zhang interprets the process of singular values as a biorthogonal ensemble with explicit recurrence relation, which satisfies a global CLT by the results of Breuer-Duits in \cite{BD17}.

\item In the setting of the $M$-fold product of $N \times N$ random matrices, the law of large numbers for Lyapunov exponents as $M \to \infty$ and then $N \to \infty$ was originally studied by Newman and Isopi-Newman in \cite{New86, New84, IN}.  It was then obtained in a free probability context by Kargin and Tucci in \cite{Kar08, Tuc}.  Our techniques recover these law of large numbers results.

In the more recent works \cite{For13, ABK} of Forrester and Akemann-Burda-Kieburg, a LLN and CLT were shown for square Ginibre ensembles with fixed $N$ and $M \to \infty$, and Forrester studied in \cite{For15} their asymptotics as $N \to \infty$.  In \cite{Red16}, Reddy showed a general LLN and CLT for Lyapunov exponents of unitarily invariant ensembles at fixed $N$.  Finally, in \cite{KK16}, Kieburg-Kosters gave a different proof for a subset of these ensembles which they call polynomial ensembles of derivative type (which corresponds to cases where the multivariate Bessel generating function is multiplicative).  However, none of these papers studies the CLT in our regime, where $M, N \to \infty$ simultaneously.

\item Local statistics of Lyapunov exponents for products of square Ginibre matrices are studied for $M, N \to \infty$ simultaneously by Akemann-Burda-Kieburg and Liu-Wang-Wang in \cite{ABK18} and \cite{LWW18}.  The local limits are shown to have picket fence statistics when $M / N \to \infty$, sine kernel statistics when $M / N \to 0$, and an interpolation between the two in the intermediate regime when $M \sim N$.  In particular, their results imply that the limits $N \to \infty$ and $M \to \infty$ do not commute on the local scale for Ginibre matrices, as earlier predicted heuristically by the same authors in \cite[Section 5]{ABK}.

Their methods rely an expression for the process of squared singular values as a determinantal point process with explicit correlation kernel, and we believe this expression may be modified to obtain similar results for rectangular Ginibre matrices.  When compared to our Theorem \ref{thm:gin-lya}, this yields a transition between the sine kernel on local scales and a Gaussian process with white noise on the global scale.

\item In the study of dynamical isometry for deep neural networks, singular values of products of many i.i.d.~copies of large random matrices appear as Jacobians of newly initialized networks in the work of many researchers in the papers \cite{SMG14, PSG17, PSG18, CPS18, XBSSP18, TWJTN, LQ18}, which study different network architectures which lead to singular values closely clustered around $1$.  We mention also the recent paper \cite{HN18} of Hanin-Nica, which proves Gaussianity for the first exponential moment of the Lyapunov exponents for a specific family of matrix products originating in the study of deep neural networks.  While these works consider orthogonally-invariant or general real matrix ensembles (as opposed to unitarily-invariant complex matrix ensembles), it is believed that the structure of fluctuations is similar in both cases, allowing us to make predictions for the behavior of fluctuations in this setting.  
\end{itemize}
Our technique via multivariate Bessel generating functions also has many relations to the literature, which we now discuss.
\begin{itemize}
\item The idea that products of unitarily-invariant random matrices are related to products of the corresponding multivariate Bessel functions dates back to at least Macdonald and Zinn-Justin in \cite[Section VII.4]{Mac} and \cite{Zin99}, where it was used to give an heuristic proof of the multiplicativity of Voiculescu's $S$-transform.  Our work may be viewed as a rigorous realization and generalization of this approach.  It rests on the analytic continuation of the key functional equation for characters of $\gl_N$ used in Lemma \ref{lem:mult}, which holds for any compact group and whose analytic continuation was studied by Berezin-Gelfand in \cite{BG56} and Helgason in \cite{Hel84}.  We mention also the recent work \cite{BGS18} of Borodin-Gorin-Strahov, where the eigenvalues of products of Jacobi ensembles were shown to form a continuous limit of a Schur process with certain specializations.

When applied to log-spectral measures of unitarily invariant random matrix ensembles, our multivariate Bessel generating functions correspond after a change of variables to the spherical transform as surveyed in \cite{Hel84}.  The spherical transform is used by Kieburg-Kosters in \cite{KK16a, KK16} to derive correlation kernels for products of random matrices drawn from so-called polynomial ensembles and to obtain a CLT for Lyapunov exponents of fixed size matrices when the polynomial ensembles are of derivative type.  In our language, this is an analogue of Theorem \ref{thm:clt-many} in the special case where $N$ is fixed, $M \to \infty$, and the multivariate Bessel generating function is multiplicative and hence $\Lambda(u, w) = 0$.  In the general $\beta$ setting, the paper \cite{GM17} of Gorin-Marcus uses the fact that such spherical transforms are scaling limits of discrete versions originating from the theory of Macdonald symmetric functions.

\item Our main technical tools Theorems \ref{thm:clt} and \ref{thm:clt-many} are in the spirit of the previous works \cite{BG15L, BG16, BG17, Hua18} of Bufetov-Gorin and Huang which together established similar results on law of large numbers and central limit theorems for measures whose Schur/Jack generating functions admit nice expansions near the point $(1, 1, \ldots, 1)$.  Our theorems may be interpreted as a partial generalization of the (continuous limit of) the Schur case of these results to expansions near arbitrary $N$-tuples $\chi$, although this introduces significant complications in the proofs.  In particular, the log-derivatives of the multivariate Bessel generating function evaluated at this point may depend on the variables in which the derivatives were taken, which was not true for evaluation of Schur functions at $(1, 1, \ldots, 1)$.  Accounting for this in the computation of moments and covariance in Theorems \ref{thm:clt} and \ref{thm:clt-many} leads to the integrals along contours around the support $V_\chi$ of $d\chi$ in their statements.

We mention also the series of papers \cite{BL17, Li18b, Li18a} by Boutillier-Li and Li which give analogues of these results for expansions of Schur generating functions near a point $(x_1, \ldots, x_n, x_1, \ldots, x_n, \ldots, x_1, \ldots, x_n)$ with periodic coordinates, though we note that those papers obtain asymptotics of Schur generating functions using different techniques.

\item The integral formula (\ref{eq:double_integral_transformed_final}) in Theorem \ref{thm:mvb-int} for evaluations of multivariate Bessel functions near $\rho$ and its Schur analogue (\ref{eq:schur-int}) were inspired by the alternate proof of \cite[Theorem 3.6]{GP} discussed in the remark immediately after its statement.  The Schur version of this formula is also used by Cuenca-Gorin in \cite[Theorem 3.4]{CG18} to study the $q$-analogue of the Gelfand-Tsetlin branching graph.
\end{itemize}

\subsection{Organization of the paper}

The remainder of this paper is organized as follows.  In Section \ref{sec:mvb-clt}, we introduce the multivariate Bessel generating functions and prove our main technical results, Theorems \ref{thm:clt} and \ref{thm:clt-many}, which give central limit theorems for global fluctuations of measures with nice multivariate Bessel generating functions.  In Section \ref{sec:mvb-asymp}, we give a new contour integral formula for certain multivariate Bessel functions in Theorem \ref{thm:mvb-int} and use it to analyze their asymptotics in Theorems \ref{thm:single-asymp} and \ref{thm:multi-asymp}.  In Section \ref{sec:rmt-prod}, we apply the results of the previous two sections to prove our main results, Theorems \ref{thm:finite-uni-product}, \ref{thm:lya}, \ref{thm:jac-lya}, and \ref{thm:gin-lya} on the convergence of height functions to Gaussian random fields.  Finally, in Section \ref{sec:2d-prod}, we present an extension of our technique to obtain a 2-D Gaussian field from products of different numbers of random matrices in Theorem \ref{thm:lya-2d}.  For the convenience of the reader, all definitions and notations will be given again in the following sections.

\subsection*{Acknowledgments}

We would like to thank Andrew Ahn, Alexei Borodin, Maurice Duits, Octavio Arizmendi Echegaray, Alice Guionnet, Victor Kleptsyn, James Mingo, Roland Speicher, and Alex Zhai for useful discussions.  Both authors were supported by the NSF grants DMS-1664619 and DMS-1664650. V.~G.~was partially supported by the NEC Corporation Fund for Research in Computers and Communications, by the Sloan Research Fellowship, by NSF grant DMS-1949820, and by the Office of the Vice Chancellor for Research and Graduate Education at the University of Wisconsin--Madison with funding from the Wisconsin Alumni Research Foundation.  Y.~S.~was supported by a Junior Fellow award from the Simons Foundation and NSF Grant DMS-1701654/2039183.  The authors also thank the organizers of the Park City Mathematics Institute research program on Random Matrix Theory and the MATRIX workshop ``Non-Equilibrium Systems and Special Functions,'' where part of this work was completed.

\section{Central limit theorem via multivariate Bessel generating functions} \label{sec:mvb-clt}

\subsection{Multivariate Bessel generating functions}

For a set of variables $a = (a_1, \ldots, a_N)$, we denote the Vandermonde determinant by $\Delta(a) := \prod_{1 \leq i < j \leq N} (a_i - a_j)$.  For sets of complex variables $a = (a_1, \ldots, a_N)$ and $b = (b_1, \ldots, b_N)$ with all $\{a_i\}$ and $\{b_i\}$ distinct, the multivariate Bessel function is given by
\begin{equation} \label{eq:mvb-def}
\cB(a, b) := \Delta(\rho) \frac{\det(e^{a_i b_j})_{i, j = 1}^N}{\Delta(a) \Delta(b)},
\end{equation}
where $\rho = (N - 1, \ldots, 0)$, and it admits analytic continuation to all values of $a$ and $b$. For a probability measure $d\mu_N$ on $W^N := \{(x_1, \ldots, x_N) \in \RR^N \mid x_1 \geq \cdots \geq x_N\}$ and an $N$-tuple $\chi \in W^N$, define its multivariate Bessel generating function $\phi_{\chi, N}(s) := \phi_{\chi, N}(s_1, \ldots, s_N)$ by
\begin{equation} \label{eq:mvb-gf}
\phi_{\chi, N}(s) := \int \frac{\cB(s, x)}{\cB(\chi, x)} d\mu_N(x).
\end{equation}

\begin{define} \label{def:smooth}
We say that a measure $d\mu_N$ is $\chi$-smooth if the defining integral of $\phi_{\chi, N}(s)$ converges
absolutely and uniformly on a complex neighborhood of $[\chi_{N}, \chi_{1}]^N$.
\end{define}

\begin{remark}
In our applications in later sections, we will take $\chi = \rho$.  For convenience, in those applications we will omit the $\chi$'s in the notation and write simply $\phi_N(s)$ and smooth instead of $\phi_{\rho, N}(s)$ and $\rho$-smooth.
\end{remark}

Define the operators
\begin{equation} \label{eq:op-def}
D_k := \Delta(s)^{-1} \circ \sum_{i = 1}^N \partial_{s_i}^k \circ \Delta(s)
\end{equation}
and the moments $p_k(x) := x_1^k + \cdots + x_N^k$.  We show now that repeated applications of the operators $D_k$ yield mixed moments of a smooth measure.

\begin{lemma} \label{lem:mom-smooth}
  For any $k_1, \ldots, k_l \geq 1$, if $d\mu_N$ is $\chi$-smooth and $x$ is $d\mu_N$-distributed, then $\EE[p_{k_1}(x) \cdots p_{k_l}(x)]$ is finite and equals
  \[
  \EE[p_{k_1}(x) \cdots p_{k_l}(x)] = D_{k_1} \cdots D_{k_l} \phi_{\chi, N}(\chi).
  \]
\end{lemma}
\begin{proof}
  Choose $\delta > 0$ small so that for $s$ in a neighborhood of $[\chi_N, \chi_1]^N$, the defining integral of $\phi_{\chi, N}$ is convergent at $s + \delta \eps$ for all $2^N$ choices of vector $\eps$ with entries in $\{\pm 1\}$.  We therefore see that on this neighborhood we have
  \begin{align*}
    \sum_{\eps_i \in \{\pm 1\}} \Delta(s + \delta \eps) \phi_{\chi, N}(s + \delta \eps) &= \int \frac{\Delta(\chi)}{\det(e^{\chi_i x_j})_{i, j = 1}^N} \sum_{\eps_i \in \{\pm 1\}} \sum_{\sigma \in S_N} (-1)^\sigma \prod_{i = 1}^N e^{s_i x_{\sigma(i)} + \delta \eps_i x_{\sigma(i)}} d\mu_N(x) \\
    &= \int \frac{\Delta(\chi)}{\det(e^{\chi_i x_j})_{i, j = 1}^N} \prod_{i = 1}^N [2\cosh(\delta x_i)] \sum_{\sigma \in S_N} (-1)^\sigma \prod_{i = 1}^N e^{s_i x_{\sigma(i)}} d\mu_N(x) \\
    &= \Delta(s) \int \prod_{i = 1}^N  [2\cosh(\delta x_i)] \frac{\cB(s, x)}{\cB(\chi, x)} d\mu_N(x).
  \end{align*}
  Notice now that
  \[
  \sum_{\eps_i \in \{\pm 1\}} \Delta(s + \delta \eps) \phi_{\chi, N}(s + \delta \eps)
  \]
  vanishes at $s_i = s_j$, meaning that 
  \begin{equation} \label{eq:cosh-conv}
      \sum_{\eps_i \in \{\pm 1\}} \frac{\Delta(s + \delta \eps)}{\Delta(s)} \phi_{\chi, N}(s + \delta \eps) = \int \prod_{i = 1}^N [2\cosh(\delta x_i)] \frac{\cB(s, x)}{\cB(\chi, x)} d\mu_N(x)
  \end{equation}
  admits an analytic continuation to and converges absolutely on a neighborhood of $[\chi_N, \chi_1]^N$ for some $\delta > 0$.  Notice now that $|p_{k_1}(x) \cdots p_{k_l}(x)| < \prod_{i = 1}^N 2 \cosh(\delta x_i)$ outside of a compact set, which implies that
  \[
  \left|p_{k_1}(x) \cdots p_{k_l}(x) \frac{\cB(s, x)}{\cB(\chi, x)}\right| < \left| \prod_{i = 1}^N [2 \cosh(\delta x_i)] \frac{\cB(s, x)}{\cB(\chi, x)}\right|
  \]
  outside of a compact set.  Therefore, the absolute convergence of (\ref{eq:cosh-conv}) together with the eigenrelation
  \[
  D_k \cB(s, x) = p_k(x) \cB(s, x)
  \]
  implies by dominated convergence that for $s$ in some neighborhood of $[\chi_N, \chi_1]^N$, we have
  \[
  D_{k_1} \cdots D_{k_l} \phi_{\chi, N}(s) = \int p_{k_1}(x) \cdots p_{k_l}(x) \frac{\cB(s, x)}{\cB(\chi, x)} d\mu_N(x),
  \]
  Plugging in $s = \chi$ then yields
  \[
  \EE[p_{k_1}(x) \cdots p_{k_l}(x)] = \int p_{k_1}(x) \cdots p_{k_l}(x) d\mu_N = D_{k_1} \cdots D_{k_l} \phi_{\chi, N}(\chi). \qedhere
  \]
\end{proof}

\subsection{Statement of the results}

Let $\chi_N \in W^N$ be a sequence of $N$-tuples such that we have the weak convergence of measures
\[
\frac{1}{N} \sum_{i = 1}^N \delta_{\chi_{N, i} / N} \to d \chi
\]
for some compactly supported measure $d\chi$ on $\RR$. Let $\phi_{\chi, N}(s) := \phi_{\chi, N}(s_1, \ldots, s_N)$ be the multivariate Bessel generating function of a sequence of smooth measures $d\mu_N$.  We now define conditions on $\phi_{\chi, N}(s)$ which imply a limit shape and Gaussian fluctuations for these measures.  For a subset $I \subset \{1, \ldots, N\}$, define $s^I := (s_i)_{i \in I}$ and
\[
\phi_{\chi, N}^I(s) := \phi_{\chi, N}(s)|_{s_j = \chi_j, j \notin I},
\]
where we note that $\phi_{\chi, N}^I(s)$ is a function of $s^I$ alone.  For a function $f(s)$ on $\CC^N$ and some $r \in \CC^N$, we will adopt the notation that $\partial_{r_i} f(r) = \frac{\partial}{\partial s_i} f(s) \Big|_{s = r}$. 

\begin{define} \label{def:lln-app}
  We say that the measures $d\mu_N$ are LLN-appropriate for $\chi_N$ if they are $\chi_N$-smooth and there exist a compact set $V_\chi \subset \RR$ and a holomorphic function $\Psi$ on an open complex neighborhood $U$ of $V_\chi$ such that $\chi_{N, i}/N \in V_{\chi}$ for all $N$ and for each fixed integer $k > 0$, uniformly in $I$ with $|I| = k$, $i \in I$, and $r^I \in U^{|I|}$, we have
  \[
  \lim_{N \to \infty} \left|\frac{1}{N}\partial_{r_i} [\log \phi_{\chi, N}^I(r N)] - \Psi(r_i)\right| = 0.
  \]
\end{define}

\begin{define} \label{def:clt-app}
We say that measures $d\mu_N$ are CLT-appropriate for $\chi_N$ if they are LLN-appropriate for $\chi_N$ with
respect to some function $\Psi$ and sets $V_\chi \subset \RR$ with neighborhood $U$ and there exists a
holomorphic function $\Lambda$ on $U^2$ such that for each fixed integer $k > 0$, uniformly in $I$ with $|I| = k$, $r^I \in U^{|I|}$,
and distinct $i, j \in I$, we have
\[
\lim_{N \to \infty} \left| \partial_{r_i} \partial_{r_j}[ \log \phi_{\chi, N}^I(rN)] - \Lambda(r_i, r_j) \right| = 0.
\]
\end{define}

\medskip

Define the Cauchy transform of the measure $d\chi$ by 
\begin{equation} \label{eq:chi-def}
  \Xi(u) := \int \frac{1}{u - x} d\chi(x)
\end{equation}
as a holomorphic function on $\CC \setminus V_\chi$.

\begin{theorem}[Law of large numbers] \label{thm:lln}
If the measures $d\mu_N$ are LLN-appropriate for $\chi_N$, then we have the following convergence in probability
\begin{equation} \label{eq:mom}
\lim_{N \to \infty} \frac{1}{N} p_k(x) = \lim_{N \to \infty} \frac{1}{N} \EE[p_k(x)] = \fp_k := \frac{1}{k + 1} \oint \Big(\Xi(u) + \Psi(u)\Big)^{k + 1} \frac{du}{2 \pi \ii},
\end{equation}
where the $u$-contour encloses $V_\chi$ and lies within $U$.  In addition, the random measures $\frac{1}{N} \sum_{i = 1}^N \delta_{x_i}$ with $x$ distributed according to $d\mu_N$ converge in probability to a deterministic compactly supported measure $d\mu$ with $\int x^k d\mu(x) = \fp_k$. 
\end{theorem}

\begin{theorem}[Central limit theorem] \label{thm:clt}
If the measures $d\mu_N$ are CLT-appropriate for $\chi_N$, then the collection of random variables
\[
\{p_k(x) - \EE[p_k(x)]\}_{k \in \NN}
\]
converges in the sense of moments to the Gaussian vector with zero mean and covariance
\begin{multline} \label{eq:cov}
\lim_{N \to \infty} \Cov\Big(p_k(x), p_l(x)\Big) = \fCov_{k, l} := \oint \oint \Big(\Xi(u) + \Psi(u)\Big)^l \Big(\Xi(w) + \Psi(w)\Big)^k\\ \left(\frac{1}{(u - w)^2} + \Lambda(w, u)\right)\frac{dw}{2\pi \ii} \frac{du}{2\pi \ii},
\end{multline}
where the $u$ and $w$-contours enclose $V_\chi$ and lie within $U$, and the $u$-contour is contained inside the $w$-contour.
\end{theorem}

\begin{remark}
In Theorems \ref{thm:lln} and \ref{thm:clt}, the coordinates of $\chi_N$ are allowed to repeat.  In particular, we may choose $\chi_N = (0, \ldots, 0)$, in which case they are a continuous version of the results of \cite{BG15L, BG16} and are closely related to the results of \cite{MN18}.
\end{remark}

In our applications, we will use the following computation of $\Xi(u)$ for $d\chi(x) = \bI_{[0, 1]} dx$.

\begin{lemma} \label{lem:chi-comp}
If $d\chi(x) = \bI_{[0, 1]} dx$, we have that $\Xi(u) = \log(u/ (u - 1))$ on $\CC \setminus [0, 1]$. 
\end{lemma}
\begin{proof}
By direct computation of $\int_0^1 \frac{1}{u - x} dx$. 
\end{proof}

\subsection{The symmetrization procedure on analytic functions} \label{sec:sym-proc}

In what follows, we will make extensive use of the following properties of symmetrizations of analytic functions.

\begin{define} \label{def:sym-func}
  Let $f(s^I)$ be a function in variables $(s_i)_{i \in I}$.  For $K = \{k_1, \ldots, k_l\} \subset I$, we define its symmetrization over $(s_k)_{k \in K}$ by 
  \[
  \Sym_K f(s^I) := \frac{1}{l!} \sum_{\sigma \in S_{l}} f\Big((s^\sigma_i)_{i \in I}\Big),
  \]
  where for $\sigma \in S_l$, we define
  \[
  s^\sigma_i = \begin{cases} s_i & i \notin K \\ s_{k_{\sigma(m)}} & i = k_m \end{cases}.
  \]
\end{define}

\begin{lemma} \label{lem:sym-func}
  If $f(s^I)$ is an analytic function of $s^I$ on $U^{|I|}$ for a complex domain $U \subset \CC$, for any $A \subset I$ and $i \in I \setminus A$, the symmetrization over $K := \{i\} \cup A$ given by
  \[
  \Sym_K \left[\prod_{a \in A} \frac{1}{s_i - s_a} f(s^I)\right]
  \]
  is an analytic function of $s^I$ on $U^{|I|}$.  Further, if $f_t(s^I)$ for $t = 1, 2, \ldots$ are a sequence of analytic functions uniformly converging to $0$ on $U^{|I|}$, then the sequence of functions
  \[
  \Sym_K \left[\prod_{a \in A} \frac{1}{s_i - s_a} f_t(s^I)\right]
  \]
converges to $0$ uniformly on compact subsets of $U^{|I|}$.
\end{lemma}
\begin{proof}
  Denote by $g(s^I)$ the symmetrized function, which is meromorphic in $s^I$ on $U^{|I|}$.  Notice that the poles of $g(s^I)$ are contained within the hyperplanes $\{s^I \in U^{|I|} \mid s_{k_1} = s_{k_2}\}$ for $k_1, k_2 \in K$.  The complement of the set
  \[
  V := \{s^I \in U^{|I|} \mid s_{k_1} = s_{k_2}, s_k \text{ distinct and not equal to $s_{k_1}$ for $k \notin \{k_1, k_2\}$} \}
  \]
  in the union of these hyperplanes has codimension $2$.  Because the set of poles is either empty or has codimension $1$ by Riemann's second extension theorem (see \cite[Theorem 7.1.2]{GR84}), it suffices for us to check that the poles of $g(s^I)$ avoid $V$.  For any $r^I \in V$ with $r_{k_1} = r_{k_2} = z$, we see that $g(s^I)$ is evidently holomorphic near $r^I$ in the variables $(s_i)_{i \in I \setminus \{k_1, k_2\}}$.  Notice now that we may rewrite $g(s^I)$ as
  \begin{multline} \label{eq:sym-expression}
    g(s^I) = \Sym_K \left[\frac{1}{s_{k_1} - s_{k_2}} \frac{1}{s_{k_1} - s_i} \prod_{k \in K \setminus \{i, k_1, k_2\}} \frac{1}{s_{k_1} - s_k} f(\tau(s^I)) \right] \\
    = \frac{1}{2} \Sym_K \left[\frac{1}{s_{k_1} - s_{k_2}} \left(\prod_{k \in K \setminus \{k_1, k_2\}} \frac{1}{s_{k_1} - s_k} f(\tau(s^I)) - \prod_{k \in K \setminus \{k_1, k_2\}} \frac{1}{s_{k_2} - s_k} f(\tau'(s^I)) \right) \right],
  \end{multline}
  where $\tau$ exchanges $s_i$ with $s_{k_1}$ and $\tau'$ exchanges $s_i$ with $s_{k_2}$.  Notice that 
  \[
  \prod_{k \in K \setminus \{k_1, k_2\}} \frac{1}{s_{k_1} - s_k} f(\tau(s^I)) \qquad \text{ and } \qquad \prod_{k \in K \setminus \{k_1, k_2\}} \frac{1}{s_{k_2} - s_k} f(\tau'(s^I))
  \]
  are both analytic near $s^I = r^I$ and are given by exchanging $s_{k_1}$ with $s_{k_2}$, which implies that the argument of the symmetrization in (\ref{eq:sym-expression}) and hence $g(s^I)$ are analytic in $s_{k_1}$ and $s_{k_2}$ near $r^I$.  Thus $g(s^I)$ is analytic in all variables near $r^I$, completing the proof of the first claim.

  For the second claim, for any compact subset $V' \subset U^{|I|}$, choose $|K|$ nested contours $(\gamma_k)_{k \in K}$ contained in $U$ and containing the coordinate projections of $V'$.  Notice that for any partition $K = \{i\} \sqcup A$, the function $\prod_{a \in A} \frac{1}{u_i - u_a}$ is uniformly bounded for $(u_k)_{k \in K}$ with $u_k \in \gamma_k$.  Denoting by $g_t(s^I)$ the symmetrized function corresponding to $f_t(s^I)$, this implies that $g_t \to 0$ uniformly on $\{s \in U \mid s_k \in \gamma_k\}$.  Applying the Cauchy integral formula $|K|$ times with the $|K|$ nested contours $(\gamma_k)_{k \in K}$ then implies that $g_t$ converges to $0$ uniformly on $V'$, as desired.
\end{proof}

\subsection{A graphical calculus} \label{sec:graph}

The purpose of this section is to produce an expansion for $m \leq n$ of 
\[
\frac{D_{k_m} \cdots D_{k_n} \phi_{\chi, N}(s)}{\phi_{\chi, N}(s)}
\]
into the sum of terms.  Each term will be associated with a \emph{forest} tagged with some
combinatorial data, where many terms can have the same index.  The proofs of Theorems \ref{thm:lln}
and Theorem \ref{thm:clt} will identify terms associated to specific forests as leading order,
allowing us to only compute the values of those terms.  In what follows, we specify this
expansion precisely via an inductive construction in $m$.

In what follows, for each $m \in \{1, \ldots, n\}$, we will define certain \emph{type $m$ terms}
appearing in our expansion by downwards induction on $m$. Such type $m$ terms will be associated
with a rooted forest $F$ on the vertex set $\{m, \ldots, n\}$
with either single or bold edges along with an assignment to each tree $T$ in $F$ of a set of
indices $K_T \subseteq \{1, \ldots, N\}$, a LLN weight $W_{T, L} \geq 0$ and a CLT weight $W_{T, C} \geq 0$.
Let $\cF_{m, n}$ be the set of such forests with associated data.  This association will have the
property that if $F = T_1 \sqcup \cdots \sqcup T_l \in \cF_{m, n}$ is a division of the forest into
trees, then the term takes the form
\[
X_1 \cdots X_l,
\]
where each $X_i$ is associated with $T_i$.  We now proceed to define the notion of a type $m$ term.

Call sets of indices $A_1, \ldots, A_n \subseteq \{1, \ldots, N\}$, indices $j_1, \ldots, j_n \in \{1, \ldots, N\}$,
and (possibly-empty) vectors of non-negative integers $B_1, \ldots, B_n$ and $C = (c_1, \ldots, c_{n - 1})$ a \emph{good label}
if $j_i \notin A_i$, $B_i = (b^i_1, \ldots, b^i_{l_i})$ with $b^i_j > 0$, $c_n = 0$, and
$k_i = |A_i| + \sum_j b^i_j + c_i$. We define a type $n$ term to be an expression of the form 
\[
\Sym_{\{j_n\} \cup A_n} \left[\prod_{a \in A_n} \frac{1}{s_{j_n} - s_a}\right] [\partial_{j_n}^{b_1^n} \log \phi_{\chi, N}(s)] \cdots [\partial_{j_n}^{b_{l_n}^n} \log \phi_{\chi, N}(s)]
\]
where $A_n, j_n, B_n$ come from a good label.  We associate to this type $n$ term the forest in $\cF_{n, n}$
consisting of the single vertex tree $T$ rooted at $n$ with index set $K_T := \{j_n\} \cup A_n$ and LLN and CLT
weights $W_L = W_C = \sum_i (b^n_i - 1) + |A_n|$.  It consists of a single multiplicand associated to $T$.

For $m < n$, we define type $m$ terms by downwards induction on $m$. Let $X_1 \cdots X_l$ be a type $m + 1$ term corresponding
to some good label and associated with a forest $F \in \cF_{m + 1, n}$ consisting of trees $T_1, \ldots, T_l$
associated to $X_1, \ldots, X_l$, respectively.  For some $c_m^1, \ldots, c_m^l \geq 0$ such that $c_m^1 + \cdots + c_m^l = c_m$ 
and where $A_i, j_i, B_i, c_i$ come from the same good label as the type $m + 1$ term, define
\begin{equation} \label{eq:new-indices}
K_T' := \{j_m\} \cup A_m \cup \bigcup_{i : c_m^i > 0} K_{T_i}.
\end{equation}
We define a type $m$ term to be an expression of the form
\begin{equation} \label{eq:type-m-term}
\left[\prod_{i: c^i_m = 0} X_i \right] \cdot \Sym_{K_T'} \left[\left[\prod_{a \in A_m} \frac{1}{s_{j_m} - s_a}\right]\left[ \prod_{i: c_m^i > 0} \partial_{j_m}^{c_m^i}[X_i]\right] [\partial_{j_m}^{b_1^m} \log \phi_{\chi, N}(s)] \cdots [\partial_{j_m}^{b_{l_m}^m} \log \phi_{\chi, N}(s)]\right].
\end{equation}
We associate this term to the forest $F' \in \cF_{m, n}$ consisting of $\{T_i \mid c^i_m = 0\}$ and a new
tree $T$ with root at $m$ with edges to the root of each tree in $\{T_i \mid c^i_m > 0\}$.  Let the edge between
$m$ and the root of $T_i$ with $c^i_m > 0$ be single if $j_m \notin K_{T_i}$ and bold otherwise.  Assign
to $T$ the index set $K_T'$ and LLN and CLT weight given by 
\begin{align*}
W'_L &:= |A_m| + \sum_i (b^m_i - 1) + c_m + \sum_{\text{$\partial_{j_m}$ applied to $X_i$}} W_{T_i, L}\\
W'_C &:= |A_m| + \sum_i (b^m_i - 1) + c_m + \sum_{\text{$\partial_{j_m}$ applied to $X_i$}} (W_{T_i, C} + \bI_{j_m \notin K_{T_i}}).
\end{align*}
Finally, associate $T_i$ to $X_i$ and $T$ to the multiplicand 
\[
\Sym_{K_T'} \left[\left[\prod_{a \in A_m} \frac{1}{s_{j_m} - s_a}\right]\left[ \prod_{i: c_m^i > 0} \partial_{j_m}^{c_m^i}[X_i]\right] [\partial_{j_m}^{b_1^m} \log \phi_{\chi, N}(s)] \cdots [\partial_{j_m}^{b_{l_m}^m} \log \phi_{\chi, N}(s)]\right].
\]
In Proposition \ref{prop:expansion}, we show that these terms are sufficiently rich to expand the action
of our differential operators in.  Lemma \ref{lem:tree-prop} and Corollary \ref{corr:tree-prop-corr} then
prove properties of this expansion.

\begin{prop} \label{prop:expansion}
The quantity
\[
\frac{D_{k_m} \cdots D_{k_n} \phi_{\chi, N}(s)}{\phi_{\chi, N}(s)}
\]
is a linear combination with coefficients independent of $N$ of type-$m$ terms, where the number of
terms associated to any $F \in \cF_{m, n}$ is independent of $N$.
\end{prop}

\begin{remark}
Before giving a proof, for the reader's convenience we illustrate Proposition \ref{prop:expansion}
explicitly for $n = 2$.  Direct computation shows that 
\begin{multline} \label{eq:2exp}
\frac{D_{k_1}D_{k_2} \phi_{\chi, N}(s)}{\phi_{\chi, N}(s)} = \frac{1}{\phi_{\chi, N}(s)} \sum_{j_1, j_2 = 1}^N \sum_{\substack{A_1 \subset \{1, \ldots, N\} \\ j_1 \notin A_1}} \sum_{\substack{A_2 \subset \{1, \ldots, N\} \\ j_2 \notin A_2}} \binom{k_1}{|A_1|} \binom{k_2}{|A_2|} \left[\prod_{a \in A_1} \frac{1}{s_{j_1} - s_a}\right] \\
\partial_{s_{j_1}}^{k_1 - |A_1|} \left[\phi_{\chi, N}(s) \frac{1}{\phi_{\chi, N}(s)}\left[\prod_{a \in A_2} \frac{1}{s_{j_2} - s_a}\right] \partial_{s_{j_2}}^{k_2 - |A_2|} \phi_{\chi, N}(s)\right].
\end{multline}
Proposition \ref{prop:expansion} associates the symmetrized summands of this expression to forests on $\{1, 2\}$.
Terms where $\partial_{s_{j_1}}^{k_1 - |A_1|}$ is applied only to the first factor $\phi_{\chi, N}(s)$ on the second line
of (\ref{eq:2exp}) correspond to the forest consisting
of two $1$-vertex trees $T_1 = \{1\}$ and $T_2 = \{2\}$ with $K_{T_1} = \{j_1\} \cup A_1$ and $K_{T_2} = \{j_2\} \cup A_2$.
The remaining terms correspond to the forest consisting of the single tree $T = \{1, 2\}$ rooted at $1$, where
the edge from $1$ to $2$ is single if $j_1 \notin \{j_2\} \cup A_2$ and bold if $j_1 \in \{j_2\} \cup A_2$. In
this case, $K_T = \{j_1\} \cup \{j_2\} \cup A_1 \cup A_2$.

If $k_1 = k_2 = 1$, and $F \in \cF_{1, 2}$ is the forest $\{\{1\}, \{2\}\}$ with given $K_{T_1}, K_{T_2}$, there are
$|K_{T_1}| \cdot |K_{T_2}|$ possibilities for $j_1, j_2, A_1, A_2$.  Neither the number of terms in the expansion
(\ref{eq:2exp}) associated to this data nor the coefficients of those terms depend on $N$.  However, we remark
that $K_{T_1}$ and $K_{T_2}$ can range over subsets of $\{1, \ldots, N\}$ of size at most $1$, meaning the
number of possible choices for $F$ does depend on $N$.
\end{remark}

\begin{proof}[Proof of Proposition \ref{prop:expansion}]
We proceed by downwards induction on $m$.  For $m = n$, by the identity of differential operators
\begin{equation} \label{eq:diff-vander}
\Delta(s)^{-1} \partial_{s_i}^{k_n} \Delta(s) = \sum_{h = 0}^{k_n} \binom{k_n}{h} \sum_{\substack{A_n \subset \{1, \ldots, N\} \\ |A_n| = h, i \notin A_n}} \left[\prod_{a \in A_n} \frac{1}{s_i - s_a}\right] \partial_{s_i}^{k_n - h},
\end{equation}
we conclude that 
\begin{align*}
D_{k_n} \phi_{\chi, N}(s) &= \Delta(s)^{-1} \sum_{i = 1}^N \partial_{s_i}^{k_n} \Delta(s) \phi_{\chi, N}(s)\\
&= \sum_{i = 1}^N  \sum_{\substack{A_n \subset \{1, \ldots, N\} \\ i \notin A_n}}  \binom{k_n}{|A_n|}\left[\prod_{a \in A_n} \frac{1}{s_i - s_a}\right] \partial_{s_i}^{k_n - |A_n|} \phi_{\chi, N}(s).
\end{align*}
Because $D_{k_n} \phi_{\chi, N}(s)$ is symmetric in $s$, we may replace the summand of this double summation
by its symmetrization over $\{s_j \mid j \in \{i\} \cup A_n\}$ multiplied by a constant independent of $N$.
Consequently, by the Leibnitz rule and the identity
\[
\partial_{s_j} \phi_{\chi, N}(s) = [\partial_{s_j} \log \phi_{\chi, N}(s)] \phi_{\chi, N}(s),
\]
this symmetrized summand is exactly a linear combination of terms of type $n$ with coefficients
independent of $N$ .  Since $k_n = |A_n| + \sum_j b^n_j$ by construction, the number of such terms
associated with a fixed choice of index set $K_T = \{j_n\} \cup A_n$ and LLN and CLT weights
$W_L = W_C$ is a function of $k_n$ and therefore also independent of $N$.

Suppose now that the claim is true for some $m + 1 \leq n$.  It suffices to check that for any type $m + 1$
term $X_1 \cdots X_l$, the expression
\[
\frac{D_{k_m}[X_1 \cdots X_l \phi_{\chi, N}(s)]}{\phi_{\chi, N}(s)}
\]
is a linear combination of type $m$ terms.  By (\ref{eq:diff-vander}), we find that
\begin{align*}
D_{k_m}[X_1 \cdots X_l \phi_{\chi, N}(s)] &= \sum_{i = 1}^N \sum_{\substack{A_m \subset \{1, \ldots, N\} \\ i \notin A_m}} \binom{k_m}{|A_m|}\left[\prod_{a \in A_m} \frac{1}{s_i - s_a}\right] \partial_{s_i}^{k_m - |A_m|}[X_1 \cdots X_l \phi_{\chi, N}(s)].
\end{align*}
Applying the Leibnitz rule to distribute the derivative and symmetrizing over $\{s_i \mid i \in K_T'\}$ with
$K_T'$ from (\ref{eq:new-indices}) the portion of the resulting summand which excludes terms $X_i$ whose $s_i$
derivative are not taken, we obtain a decomposition into a linear combination of type $m$ terms. By the inductive
hypothesis, the number of type $m + 1$ terms $X_1 \cdots X_l$ associated to $F \in \cF_{m + 1, n}$ is independent of $N$
and their coefficients are independent of $N$.  In addition, if a type $m$ term is associated to $F' \in \cF_{m, n}$,
it comes from a type $m + 1$ term $F' \in \cF_{m + 1, n}$ via the procedure just described, and the number of possible
choices of $F'$ is independent of $N$.  Since the coefficients and number of type $m$ terms associated to $F$ by the
procedure above are independent of $N$, we conclude that the resulting overall decomposition into type $m$ terms
again has coefficients independent of $N$ and number of terms associated to $F \in \cF_{m, n}$ independent of $N$, completing
the induction.
\end{proof}

\begin{lemma} \label{lem:tree-prop}
  If a term $X_1\cdots X_l$ corresponds to a forest with trees $\{T_1, \ldots, T_l\}$ and index sets $K_{T_1}, \ldots, K_{T_l}$, then letting
  \[
  X_t^I(s) := X_t(s)|_{s_j = \chi_j, j \notin I}, \qquad t \in \{1, \ldots, l\},
  \]
  for
  \[
  I = K_{T_1} \cup \cdots \cup K_{T_l},
  \]
  each $X_t^I(s)$ with $t \in \{1, \ldots, l\}$ satisfies
\begin{itemize}
\item[(a)] $X_t^I(s)$ is symmetric in $\{s_i\}_{i \in K_{T_t}}$;
\item[(b)] $X_t^I(rN)$ is analytic in $r^I$;

\item[(c)] if $d\mu_N$ is LLN-appropriate for $\chi_N$, $N^{W_{T_t, L}} X_t^I(rN)$ converges as $N \to \infty$ to an analytic function of $r^I$ on $U^{|I|}$;
  
\item[(d)] if $d\mu_N$ is CLT-appropriate for $\chi_N$, $N^{W_{T_t, C}} X_t^I(rN)$ converges as $N \to \infty$ to an analytic function of $r^I$ on $U^{|I|}$;
  
\item[(e)] if $d\mu_N$ is CLT-appropriate for $\chi_N$ and $i \in I \setminus K_{T_t}$, $N^{W_{T_t, C} + 1} \partial_{r_i} X_t^I(rN)$ converges as $N \to \infty$ to an analytic function of $r^I$ on $U^{|I|}$;

\item[(f)] the number of appearances of $\log \phi_{\chi, N}^I(s)$ in $X_t^I(s)$ is given by $\sum_{i \in T_t} k_i - W_{T_t, L}$.
\end{itemize}
\end{lemma}
\begin{proof}
  We consider the expansion of
  \[
  \frac{D_{k_m} \cdots D_{k_n} \phi_{\chi, N}(s)}{\phi_{\chi, N}(s)}
  \]
  into terms indexed by forests on $\{m, \ldots, n\}$ and apply downward induction on $m$ from $n$ to $1$.  For $m = n$, we consider terms of the form
  \[
  X_t^I(s) = \Sym_{\{j_n\} \cup A_n} \left[\prod_{a \in A_n} \frac{1}{s_{j_n} - s_a}\right] [\partial_{j_n}^{b_1^n} \log \phi_{\chi, N}^I(s)] \cdots [\partial_{j_n}^{b_{l_n}^n} \log \phi_{\chi, N}^I(s)].
  \]
  Claims (a) and (b) follow from the definition and an application of Lemma \ref{lem:sym-func}.  For (c) and (d), we see that
  \begin{multline*}
N^{W_{T_t, C}} X_t^I(rN) =  N^{W_{T_t, L}} X_t^I(rN)\\ = \Sym_{\{j_n\} \cup A_n} \left[\prod_{a \in A_n} \frac{1}{r_{j_n} - r_a}\right] [N^{-1} \partial_{r_{j_n}}^{b_1^n} \log \phi_{\chi, N}^I(rN)] \cdots [N^{-1} \partial_{r_{j_n}}^{b_{l_n}^n} \log \phi_{\chi, N}^I(rN)].
  \end{multline*}
  By Definition \ref{def:lln-app}, we see that $\frac{1}{N} \partial_{r_i} \log \phi_{\chi, N}^I(r N)$ converges to an analytic function of $r^I$, so (c) and (d) both follow from Lemma \ref{lem:sym-func}.  For (e), because $i \notin K_{T_t}$, we see that $\partial_{r_i}$ and $\Sym_{K_{T_t}}$ commute, meaning that
  \[
  N^{W_{T_t, C} + 1} \partial_{r_i} X_t^I(rN) = \Sym_{\{j_n\} \cup A_n} \left[\prod_{a \in A_n} \frac{1}{r_{j_n} - r_a}\right] N \partial_{r_i} \Big([N^{-1} \partial_{r_{j_n}}^{b_1^n} \log \phi_{\chi, N}^I(rN)] \cdots [N^{-1} \partial_{r_{j_n}}^{b_{l_n}^n} \log \phi_{\chi, N}^I(rN)]\Big),
  \]
  where the final term converges to an analytic function of $r^I$ by Definition \ref{def:clt-app}.  By the same reasoning as for (c) and (d), we obtain (e).  Finally, Claim (f) follows from the fact that
  \[
  \#\{\text{appearances of $\log \phi_{\chi, N}^I(s)$}\} = l_n = k_n - |A_n| - \sum_i (b^n_i - 1) = k_n - W_{T, L}.
  \]
  For the inductive step, Claims (a) and (b) follow for the same reasons.  For (c), we see that $N^{W_{T_t, L}} X_t^I(rN)$ takes the form
  \begin{multline*}
  N^{W_{T_t, L}} X_t^I(rN) = \Sym_{K_{T_t}} \left[\prod_{a \in A_m} \frac{1}{r_{j_m} - r_a} \left[[N^{W_{i_1, L}}\partial_{r_{j_m}}^{c_m^1} Y_{i_1}^I(r N)] \cdots [N^{W_{i_k, L}}\partial_{r_{j_m}}^{c_m^k} Y_{i_k}^I(rN)] \right]\right.\\ \left.[\partial_{r_{j_m}}^{b_1^m} \log \phi_{\chi, N}^I(rN)] \cdots [\partial_{r_{j_m}}^{b_{l_m}^m} \log \phi_{\chi, N}^I(rN)]\right],
  \end{multline*}
  where $Y_i$ are terms with weights $W_{i, L}$ appearing in the expansion of $\frac{D_{k_{m + 1}} \cdots D_{k_n} \phi_{\chi, N}(s)}{\phi_{\chi, N}(s)}$ and $c_m^i > 0$ are such that $c_m^1 + \cdots + c_m^k = c_m$.  Claim (c) then follows by the inductive hypothesis, Definition \ref{def:lln-app}, and the same argument as in the base case.  For (d), we notice that
  \begin{multline*}
    N^{W_{T_t, C}} X_t^I(rN) = \Sym_{K_{T_t}} \left[\prod_{a \in A_m} \frac{1}{r_{j_m} - r_a} \left[[N^{W_{i_1, C} + \bI_{j_m \notin K_{i_1}}} \partial_{r_{j_m}}^{c_m^1} Y_{i_1}^I(r N)] \cdots [N^{W_{i_k, C} + \bI_{j_m \notin K_{i_n}}}\partial_{r_{j_m}}^{c_m^k} Y_{i_k}^I(rN)] \right]\right. \\ \left. [\partial_{r_{j_m}}^{b_1^m} \log \phi_{\chi, N}^I(rN)] \cdots [\partial_{r_{j_m}}^{b_{l_m}^m} \log \phi_{\chi, N}^I(rN)]\right].
  \end{multline*}
  By the inductive hypothesis and the fact that CLT-appropriate measures are LLN-appropriate in Definition \ref{def:clt-app}, all terms after $\Sym_{K_{T_t}} \left[\prod_{a \in A_m} \frac{1}{r_{j_m} - r_a}\right]$ converge to an analytic function in $r^I$, hence so does the final term upon application of Lemma \ref{lem:sym-func}.  For (e), if $i \notin K_{T_t}$, we notice that $\partial_{r_i}$ and $\Sym_{K_{T_t}}$ commute and we obtain
    \begin{multline*}
    N^{W_{T_t, C} + 1} \partial_{r_i} X_t^I(rN) = \Sym_{K_{T_t}} \left[\prod_{a \in A_m} \frac{1}{r_{j_m} - r_a} N\partial_{r_i}\left([N^{W_{i_1, C} + \bI_{j_m \notin K_{i_1}}} \partial_{r_{j_m}}^{c_m^1} Y_{i_1}^I(r N)] \cdots \right.\right. \\
    \left. \left.[N^{W_{i_k, C} + \bI_{j_m \notin K_{i_n}}}\partial_{r_{j_m}}^{c_m^k} Y_{i_k}^I(rN)]  [\partial_{r_{j_m}}^{b_1^m} \log \phi_{\chi, N}^I(rN)] \cdots [\partial_{r_{j_m}}^{b_{l_m}^m} \log \phi_{\chi, N}^I(rN)]\right)\right].
    \end{multline*}
    By a combination of the inductive hypothesis and Definition \ref{def:clt-app}, we see that all terms after $\Sym_{K_{T_t}} \left[\prod_{a \in A_m} \frac{1}{r_{j_m} - r_a}\right]$ converge to an analytic function in $r^I$, so Claim (e) follows from Lemma \ref{lem:sym-func}.  Finally, Claim (f) now follows from our recursive definition of the LLN weight and the fact that $k_m = |A_m| + \sum_{i = 1}^{l_m} b^m_i + c_m$. 
\end{proof}

\begin{corr} \label{corr:tree-prop-corr}
 If a term $X_1\cdots X_l$ corresponds to a forest with trees $\{T_1, \ldots, T_l\}$ and index sets $K_{T_1}, \ldots, K_{T_l}$, then for $t \in \{1, \ldots, l\}$, $X_t^I(s)$ satisfies
\begin{itemize}
\item[(a)] if $d\mu_N$ is LLN-appropriate for $\chi_N$, uniformly in $s^I / N$ in a neighborhood of $V_\chi^{|I|}$, $m_1, \ldots, m_k \geq 0$, and $i_1, \ldots, i_k \in K_{T_t}$, we have that
\[
\partial_{s_{i_1}}^{m_1} \cdots \partial_{s_{i_k}}^{m_k} X_t^I(s) = O\Big(N^{-W_{T_t, L} - \sum_j m_j}\Big);
\]
\item[(b)] if $d\mu_N$ is CLT-appropriate for $\chi_N$, uniformly in $s^I / N$ in a neighborhood of $V_\chi^{|I|}$, $m_1, \ldots, m_k \geq 0$, and $i_1, \ldots, i_k \in K_{T_t}$, we have that
\[
\partial_{s_{i_1}}^{m_1} \cdots \partial_{s_{i_k}}^{m_k} X_t^I(s) = O\Big(N^{-W_{T_t, C} - \sum_j m_j - \bI_{\{\text{$m_j > 0$ for some $i_j \notin K_{T_t}$}\}}}\Big).
\]
\end{itemize}
\end{corr}
\begin{proof}
Claim (a) follows from Lemma \ref{lem:tree-prop}(c), and Claim (b) follows from Lemma \ref{lem:tree-prop}(d) and (e).
\end{proof}

\subsection{Proof of Theorem \ref{thm:lln}}

Define the expressions
\[
\we_{a, i}(s) := a! e_{a}\Big(\frac{1}{s_i - s_\star}\Big)_{\star \neq i} \qquad \text{ and } \qquad \we_{a}(u, s) := a! e_a\Big(\frac{1}{u - s_\star}\Big),
\]
where $e_a$ denotes the $a^{\text{th}}$ elementary symmetric polynomial, and $(s_1, \ldots, s_N)$ is either a collection of variables or numbers with $s_i \neq s_j$ for any $j \neq i$.

\begin{lemma} \label{lem:power-asymp}
For $u \in \CC \setminus V_\chi$, we have that
\[
N^{-a}\we_{a}(u, \chi_N/N) = \Xi(u)^a + O(N^{-1}).
\]
\end{lemma}
\begin{proof}
Define the normalized power sums by
\[
\np_k(u, s) := \frac{1}{N} \sum_{i = 1}^N \frac{1}{(u - s_i)^k}.
\]
Viewing $\np_k(u, \chi_N/N)$ as a Riemann sum, we find that
\begin{equation} \label{eq:p-comp}
\np_k(u, \chi_N/N) = \int_{-\infty}^\infty \frac{1}{(u - x)^k} d\chi(x) + O(N^{-1}) = \begin{cases} \Xi(u) & k = 1\\ O(1) & k > 1 \end{cases} + O(N^{-1}).
\end{equation}
Now, by \cite[Example I.2.8]{Mac}, we have the determinant formula expressing elementary symmetric polynomials in terms of power sums
\[
e_a = \frac{1}{a!} \left|\begin{matrix} p_1 & 1 & 0 & \cdots & 0 \\ p_2 & p_1 & 2 & \cdots & 0 \\ \vdots & \vdots & \ddots &  & \vdots \\ p_{a - 1} & p_{a - 2} & \cdots & & a - 1 \\ p_a & p_{a - 1} & \cdots & & p_1 \end{matrix}\right|.
\]
If we apply this formula to compute $\we_a(u, \chi_N/N)$, each factor of $p_k$ has order $N$ by (\ref{eq:p-comp}).  Therefore, the leading order contribution must come from the product of the diagonal entries, implying that 
\[
N^{-a} \we_a(u, \chi_N/N) = \np_1(u, \chi_N/N)^a(1 + O(N^{-1})) = \Xi(u)^a + O(N^{-1}). \qedhere
\]
\end{proof}

We are now ready to prove the law of large numbers.  By Lemma \ref{lem:mom-smooth}, we have that 
\[
\frac{1}{N} \EE[p_k(x)] = \frac{1}{N} \frac{D_k \phi_{\chi, N}(\chi_N)}{\phi_{\chi, N}(\chi_N)}.
\]
Applying the graphical calculus of the previous section and noting the additional prefactor of $N^{-1}$, this is a linear combination of the evaluation at $s = \chi_N$ of terms of the form
\[
N^{-|A_1| - 1} \Sym_{\{j_1\} \cup A_1} \left[\prod_{a \in A_1} \frac{1}{s_{j_1} - s_a}\right] [\partial_{j_1}^{b^1_1} \log \phi_{\chi, N}(s)] \cdots [\partial_{j_n}^{b^1_l} \log \phi_{\chi, N}(s)],
\]
where $j_1 \notin A_1$, $b^1_h > 0$, and $|A_1| + \sum_h b^1_h = k$.  Such a term corresponds to $1$-vertex tree with index set $K = \{j_1\} \cup A_1$ and LLN weight $|A_1| + \sum_h (b^1_h - 1)$, meaning by Corollary \ref{lem:tree-prop}(c) that it has order $O(N^{-|A_1| - \sum_h (b^1_h - 1) - 1})$; there are $O(N^{|A_1| + 1})$ such terms, meaning that their sum has order $O(N^{-\sum_h (b^1_h - 1)})$.  As a result, the terms with leading order contribution are those with $b^1_h = 1$, implying that
\[
\frac{1}{N} \EE[p_k(x)] = \frac{1}{N} \sum_{b = 0}^k \binom{k}{b} N^{b - k} \sum_{j = 1}^N \we_{k - b, j}(r) [N^{-1} \partial_{r_j} \log \phi_{\chi, N}(r N)]^b \Big|_{r = \chi_N / N} + O(N^{-1}).
\]
Notice now that
\[
\sum_{j = 1}^N \we_{k - b, j}(r) [N^{-1} \partial_{r_j} \log \phi_{\chi, N}(rN)]^b \Big|_{r = \chi_N / N}
\]
is a linear combination with coefficients independent of $N$ of $O(N^{k - b + 1})$ terms of the form
\[
\Sym_{\{j\} \cup A} \left[\prod_{a \in A} \frac{1}{r_j - r_a}\right][N^{-1} \partial_{r_j} \log \phi_{\chi, N}^I(r N)]^b \Big|_{r = \chi_N / N}
\]
for $j\notin A$, $|A| = k - b$, and $I = \{j\} \cup A$.  By LLN-appropriateness we see that
\[
[N^{-1} \partial_{r_j}\log \phi_{\chi, N}^I(rN)]^b  = \Psi(r_j)^b + o(1),
\]
uniformly on $U^{|I|}$ for an open neighborhood $U \supset V_\chi$.  By Lemma \ref{lem:sym-func}, this implies that
\[
\Sym_{\{j\} \cup A} \left[\prod_{a \in A} \frac{1}{r_j - r_a}\right]\Big([N^{-1} \partial_{r_j} \log \phi_{\chi, N}^I(r N)]^b - \Psi(r_j)^b\Big)\Big|_{r = \chi_N / N} = o(1)
\]
uniformly on compact subsets of $U^{|I|}$.  This implies that
\begin{multline*}
\frac{1}{N} \sum_{b = 0}^k \binom{k}{b} N^{b - k} \sum_{j = 1}^N \we_{k - b, j}(r) \Big([N^{-1} \partial_{r_j} \log \phi_{\chi, N}(r N)]^b - \Psi(r_j)^b\Big) \Big|_{r = \chi_N / N}\\ = \frac{1}{N} \sum_{b = 0}^k \binom{k}{b} N^{b - k} \cdot o(N^{k - b + 1}) = o(1).
\end{multline*}
We conclude that 
\begin{align*}
\frac{1}{N} \EE[p_k(x)] &= \frac{1}{N} \sum_{b = 0}^k  N^{b - k} \binom{k}{b} \sum_{j = 1}^N  \we_{k - b, j}(r) \Big(\Psi(r)^b + o(1)\Big) \Big|_{r = \chi_N / N}\\
&= \frac{1}{N} \sum_{b = 0}^k \binom{k}{b} \oint \frac{du}{2\pi \ii} N^{b - k} \frac{\we_{k - b + 1}(u, \chi_N/N)}{k - b + 1} \Psi(u)^b + o(1),
\end{align*}
where the $u$-contour encloses $V_\chi$ and lies within $U$.  By Lemma \ref{lem:power-asymp}, we find that
\begin{align*}
\frac{1}{N} \EE[p_k(x)] &= \sum_{b = 0}^k \binom{k}{b} \oint \frac{du}{2\pi\ii} \frac{\Xi(u)^{k - b + 1}}{k - b + 1} \Psi(u)^b + o(1)\\
&= \frac{1}{k + 1} \sum_{b = 0}^k \binom{k + 1}{b} \oint \frac{du}{2\pi \ii} \Xi(u)^{k - b + 1} \Psi(u)^b + o(1)\\
&= \frac{1}{k + 1} \oint \frac{du}{2\pi \ii} \Big(\Xi(u) + \Psi(u)\Big)^{k + 1} + o(1).
\end{align*}

It remains to check that $\lim_{N \to \infty} \frac{1}{N^2} \Cov(p_k(x), p_k(x)) = 0$.  For this, define
\begin{align} \label{eq:c-def}
C_{N, k, l}(s) &:= \frac{D_k D_l \phi_{\chi, N}(s)}{\phi_{\chi, N}(s)} - \frac{D_l\phi_{\chi, N}(s)}{\phi_{\chi, N}(s)} \frac{D_l \phi_{\chi, N}(s)}{\phi_{\chi, N}(s)}\\ \nonumber
&= \sum_{a = 0}^k \sum_{l = 0}^b \binom{k}{a} \binom{l}{b} \sum_{c = 1}^a \sum_{i = 1}^N \binom{a}{c} \we_{k - a, i}(s) \partial_{s_i}^c \left[\sum_{j = 1}^N \we_{l - b, j}(s) \frac{\partial_{s_j}^b \phi_{\chi, N}(s)}{\phi_{\chi, N}(s)}\right] \frac{\partial_{s_i}^{a - c} \phi_{\chi, N}(s)}{\phi_{\chi, N}(s)}
\end{align}
and notice by Lemma \ref{lem:mom-smooth} that
\[
\frac{1}{N^2} \Cov(p_k(x), p_k(x)) = \frac{1}{N^2} C_{N, k, k}(\chi_N).
\]
Choose indices $i, j$ and sets $A_1, A_2$ with $i \notin A_1$ and $j \notin A_2$.  Defining $K_1 := \{i\} \cup \{j\} \cup A_1 \cup A_2$ and $K_2 := \{j\} \cup A_2$, we notice that $C_{N, k, k}(\chi_N)$ is a linear combination of the evaluation at $\chi_N$ of terms of the form
\begin{multline} \label{eq:term-form}
N^{-|A_1| - |A_2|}\Sym_{K_1} \left[\prod_{*\in A_1} \frac{1}{s_i/N - s_*/N}\right] \prod_{h = 1}^{t'}[\partial_{s_i}^{b_h'} \log \phi_{\chi, N}^{K_1}(s)]\\ \partial_{s_i}^c\left[\Sym_{K_2} \left[\prod_{*\in A_2} \frac{1}{s_j/N - s_*/N}\right] \prod_{h = 1}^t [\partial_{s_j}^{b_h} \log \phi_{\chi, N}^{K_1}(s)]\right] 
\end{multline}
for different choices of $i, j, A_1, A_2$ and $c, b_h, b_h' > 0$ with $|A_1| + c + \sum_h b_h' = k$, $|A_2| + \sum_h b_h = l$, and $c \geq 1$.  Such terms are again linear combinations of the terms associated to the forest on $\{1, 2\}$ consisting of a $2$-vertex tree with index set $K_1$ and LLN weight
\[
|A_1| + |A_2| + \sum_h (b_h' - 1) + \sum_h (b_h - 1) + c,
\]
so by Lemma \ref{lem:tree-prop}(c), they have order
\[
O\Big(N^{- |A_1| - |A_2| - \sum_h (b_h' - 1) - \sum_h (b_h - 1) - c}\Big).
\]
There are $O(N^{|A_1| + |A_2| + 1 + \bI_{i \notin K_2}})$ such terms, meaning that their sum has order
\[
O\Big(N^{- \sum_h (b_h' - 1) - \sum_h (b_h - 1) + (1 - c) + \bI_{i \notin K_2}}\Big),
\]
hence the terms with leading order contribution to $C_{N, k, k}(\chi_N)$ are those with $b_h' = 1$, $b_h = 1$, $c = 1$, and $i \notin K_2$.  We conclude that $C_{N, k, k}(\chi_N) = O(N)$, hence
\[
\frac{1}{N^2} \Cov(p_k(x), p_k(x)) = \frac{1}{N^2} C_{N, k, k}(\chi_N) = O(N^{-1}),
\]
which concludes the proof that $\lim_{N \to \infty} \frac{1}{N} p_k(x) = \lim_{N \to \infty} \EE[p_k(x)]$ in probability.  This immediately implies the corresponding convergence of empirical measures.

\subsection{Proof of Theorem \ref{thm:clt}}

We prove Theorem \ref{thm:clt} in two steps.  We first prove that $\Cov(p_k(x), p_l(x))$ has the claimed value and then show that all higher cumulants vanish.

\subsubsection{Computing the covariance}

Notice that $\Cov(p_k(x), p_l(x)) = C_{N, k, l}(\chi_N)$, where we recall the definition of $C_{N, k, l}(s)$
from (\ref{eq:c-def}).  Applying the expansion from the previous section into terms of the form (\ref{eq:term-form}),
we see that $C_{N, k, l}(\chi_N)$ is a linear combination of terms associated to the forest on $\{1, 2\}$
consisting of a $2$-vertex tree with index sets $K_1 = \{i\} \cup A_1 \cup \{j\} \cup A_2$ and $K_2 = \{j\} \cup A_2$ and CLT weight
\[
|A_1| + |A_2| + \sum_h (b_h' - 1) + \sum_h (b_h - 1) + c + \bI_{i \notin K_2},
\]
so by Lemma \ref{lem:tree-prop}(d), they have order
\[
O\Big(N^{- |A_1| - |A_2| - \sum_h (b_h' - 1) - \sum_h (b_h - 1) - c - \bI_{i \notin K_2}}\Big).
\]
There are $O(N^{|A_1| + |A_2| + 1 + \bI_{i \notin K_2}})$ such terms, meaning that their sum has order
\[
O(N^{- \sum_h (b_h' - 1) - \sum_h (b_h - 1) + (1 - c)}).
\]
Thus, the leading order terms in $C_{N, k, l}(\chi_N)$ are those with $b_h' = 1$, $b_h = 1$, and $c = 1$.  To compute the covariance, we analyze these terms more precisely.  Setting $I = K_1$, define
\[
\Delta^{I, j}_N(r) := \frac{1}{N} \partial_{r_j} \log \phi_{\chi, N}^I(r N) - \Psi(r_j)
\]
so that from Definition \ref{def:clt-app} we have
\begin{align} \label{eq:d-asymp}
\partial_{r_j}^m \Delta^{I, j}_N(r) &= o(1), \qquad m \geq 0 \\ \notag
\partial_{r_i} \Delta^{I, j}_N(r) &= \frac{1}{N} F^{(1, 1)}(r_i, r_j) + o(N^{-1}), \qquad i \neq j.
\end{align}
With this notation, the terms with $b_h' = 1$, $b_h = 1$, and $c = 1$ take the form
\begin{multline*}
    N^{-|A_1| - |A_2| - 1} \Sym_{K_1} \left[\prod_{a \in A_1} \frac{1}{r_i - r_a} \partial_{r_i}\left[\Sym_{K_2} \prod_{a \in A_2} \frac{1}{r_j - r_a} [N^{-1}\partial_{r_j} \log \phi_{\chi, N}^I(rN)]^{l - |A_2|}\right] [N^{-1}\partial_{r_i} \log \phi_{\chi, N}^I(rN)]^{k - |A_1| - 1}\right] \\
    = N^{-|A_1| - |A_2| - 1} \Sym_{K_1} \left[\prod_{a \in A_1} \frac{1}{r_i - r_a} \partial_{r_i}\left[\Sym_{K_2} \prod_{a \in A_2} \frac{1}{r_j - r_a} \Big(\Psi(r_j) + \Delta^{I, j}_N(r)\Big)^{l - |A_2|}\right] \Big(\Psi(r_i) + \Delta^{I, i}_N(r)\Big)^{k - |A_1| - 1}\right]\\
    = N^{-|A_1| - |A_2| - 1}\sum_{m = 0}^{l - |A_2|} \Sym_{K_1} \left[\prod_{a \in A_1} \frac{1}{r_i - r_a} \partial_{r_i}\left[ f_{j, A_2}^m(r)\right] \Big(\Psi(r_i) + \Delta^{I, i}_N(r)\Big)^{k - |A_1| - 1}\right]
\end{multline*}
for some $i \notin A_1$ and $j \notin A_2$ and $f_{j, A_2}^m(r)$ defined for $0 \leq m \leq l - |A_2|$ by
\[
f_{j, A_2}^m(r) := \Sym_{K_2} \prod_{a \in A_2} \frac{1}{r_j - r_a} \binom{l - |A_2|}{m} \Psi(r_j)^{l - |A_2| - m} \Delta^{I, j}_N(r)^{m}.
\]
We study the asymptotics of each summand separately.  We see that each $f_{j, A_2}^m(r)$ converges as $N \to \infty$ to an analytic function of $r^I$, which is $0$ unless $m = 0$.  Therefore, if $i \in K_2$ and $m > 0$ the summand
\[
N^{-|A_1| - |A_2| - 1}\Sym_{K_1} \left[\prod_{a \in A_1} \frac{1}{r_i - r_a} \partial_{r_i}\left[ f_{j, A_2}^m(r)\right] \Big(\Psi(r_i) + \Delta^{I, i}_N(r)\Big)^{k - |A_1| - 1}\right]
\]
has analytic limit of order $o(N^{-|A_1| - |A_2| - 1})$ by Lemma \ref{lem:sym-func}.  Similarly if $m = 0$, for $g > 0$, the terms 
\begin{equation} \label{eq:cov-term1}
N^{-|A_1| - |A_2| - 1}\Sym_{K_1} \left[\prod_{a \in A_1} \frac{1}{r_i - r_a} \partial_{r_i}\left[ f_{j, A_2}^0(r)\right] \binom{k - |A_1| - 1}{g} \Psi(r_i)^{k - |A_1| - 1 - g} \Delta^{I, i}_N(r)^{g}\right]
\end{equation}
resulting from the binomial expansion of $\Big(\Psi(r_i) + \Delta^{I, i}_N(r)\Big)^{k - |A_1| - 1}$ have analytic limit of order $o(N^{-|A_1| - |A_2| - 1})$ by Lemma \ref{lem:sym-func}.  There are $O(N^{|K_2|}) = O(N^{|A_1| + |A_2| + 1})$ of each of these type of terms, so the resulting terms do not contribute to the limit unless $m = 0$ and $g = 0$.

Now, if $i \notin K_2$, for $m \geq 1$ we have that
\[
\partial_{r_i} f_{j, A_2}^m(r) = \frac{1}{N} \Sym_{K_2} \prod_{a \in A_2} \frac{1}{r_j - r_a} \binom{l - |A_2|}{m} \Psi(r_j)^{l - |A_2| - m} m \Delta^{I, j}_N(r)^{m - 1} F^{(1, 1)}(r_i, r_j) +o(N^{-1}),
\]
meaning that $N \partial_{r_i} f_{j, A_2}^m(r)$ converges as $N \to \infty$ to an analytic function of $r^I$ which is $0$ unless $m = 1$.  Therefore, if $i \notin K_1$, for $m > 1$ the summand
\[
N^{-|A_1| - |A_2| - 1}\Sym_{K_1} \left[\prod_{a \in A_1} \frac{1}{r_i - r_a} \partial_{r_i}\left[ f_{j, A_2}^m(r)\right] \Big(\Psi(r_i) + \Delta^{I, i}_N(r)\Big)^{k - |A_1| - 1}\right]
\]
has analytic limit of order $o(N^{-|A_1| - |A_2| - 2})$ by Lemma \ref{lem:sym-func}.  Similarly, if $m = 1$, for $g > 0$ the terms 
\begin{equation} \label{eq:cov-term2}
N^{-|A_1| - |A_2| - 1}\Sym_{K_1} \left[\prod_{a \in A_1} \frac{1}{r_i - r_a} \partial_{r_i}\left[ f_{j, A_2}^m(r)\right] \binom{k - |A_1| - 1}{g} \Psi(r_i)^{k - |A_1| - 1 - g} \Delta^{I, i}_N(r)^g \right]
\end{equation}
have analytic limit of order $o(N^{-|A_1| - |A_2| - 2})$ by Lemma \ref{lem:sym-func}.  Since there are $O(N^{|K_2|}) = O(N^{|A_1| + |A_2| + 2})$ such terms, their total contribution is $o(1)$ unless $m = 1$ and $g = 0$.

We conclude that the leading order terms in $C_{N, k, l}(\chi_N)$ correspond to those with $m = 0$, $g = 0$, and $i \in K_2$ in (\ref{eq:cov-term1}) or $m = 1$, $g = 0$, and $i \notin K_2$ in (\ref{eq:cov-term2}), with their coefficients given by their coefficients in the direct expansions $X_1(\chi_N / N)$ and $X_2(\chi_N / N)$ for 
\begin{align*}
X_1(r) &= \sum_{a = 0}^k \sum_{b = 0}^l \binom{k}{a} \binom{l}{b} \sum_{i = 1}^N a N^{a - k + b - l - 1} \we_{k - a, i}(r) \partial_{r_i}\left[\sum_{j = 1}^N \we_{l - b, j}(r) \Psi(r_j)^b\right]\Psi(r_i)^{a - 1}\\
X_2(r) &= \sum_{a = 0}^k \sum_{b = 0}^l \binom{k}{a} \binom{l}{b} \sum_{i = 1}^N a N^{a - k + b - l - 2} \we_{k - a, i}(r)\sum_{j \neq i} \sum_{\substack{A_2 : i, j \notin A_2\\ |A_2| = l - b}} (l - b)! \prod_{a \in A_2} \frac{1}{r_j - r_a}  b \Psi(r_j)^{b - 1} F^{(1, 1)}(r_i, r_j)\Psi(r_i)^{a - 1},
\end{align*}
where $X_1(\chi_N / N)$ corresponds to the terms with $m = 0$, and $X_2(\chi_N / N)$ corresponds to the terms with $m = 1$.  Defining $\hchi_N := \chi_N / N$, we see that
\begin{align*}
X_1(\hchi_N) &= \sum_{a = 0}^k \sum_{b = 0}^l \binom{k}{a} \binom{l}{b} a N^{a - k + b - l - 1} \sum_{i = 1}^N \we_{k - a, i}(r) \partial_{r_i} \left.\left[ \oint \frac{du}{2\pi \ii} \frac{\we_{l - b + 1}(u, r)}{l - b + 1} \Psi(u)^b\right]\Psi(r_i)^{a - 1} \right|_{r = \hchi_N} \\
&= \sum_{a = 0}^k \sum_{b = 0}^l \binom{k}{a} \binom{l}{b} a N^{a - k + b - l - 1} \left.\sum_{i = 1}^N \we_{k - a, i}(r) \oint \frac{du}{2\pi \ii} \sum_{p = 0}^{l - b} \frac{(-1)^p (l - b)_p}{(u - r_i)^{p + 2}} \we_{l - b - p}(u, r) \Psi(u)^b \Psi(r_i)^{a - 1}\right|_{r = \hchi_N} \\
&= \sum_{a = 0}^k \sum_{b = 0}^l \binom{k}{a} \binom{l}{b} a N^{a - k + b - l - 1} \oint \oint \frac{dw}{2\pi \ii} \frac{du}{2\pi \ii} \frac{\we_{k - a + 1}(w, \hchi_N)}{k - a + 1} \sum_{p = 0}^{l - b} \frac{(-1)^p (l - b)_p}{(u - w)^{p + 2}} \we_{l - b - p}(u, \hchi_N) \Psi(u)^b \Psi(w)^{a - 1},
\end{align*}
where both the $u$ and $w$-contours enclose $V_\chi$ and lie in $U$, and the $u$-contour is contained inside the $w$-contour.  Applying Lemma \ref{lem:power-asymp} yields
\begin{align*}
  X_1(\hchi_N) &= \sum_{a = 0}^k \sum_{b = 0}^l \binom{k}{a} \binom{l}{b} a N^{-p} \oint \oint \frac{dw}{2\pi \ii} \frac{du}{2\pi \ii} \frac{\Xi(w)^{k - a + 1}}{k - a + 1} \sum_{p = 0}^{l - b} \frac{(-1)^p (l - b)_p}{(u - w)^{p + 2}} \Xi(u)^{l - b - p}\Psi(u)^b \Psi(w)^{a - 1} + O(N^{-1})\\
&= \sum_{a = 0}^k \sum_{b = 0}^l \binom{k}{a} \binom{l}{b} a \oint \oint \frac{dw}{2\pi \ii} \frac{du}{2\pi \ii} \frac{\Xi(w)^{k - a + 1}}{k - a + 1} \Xi(u)^{l - b} \frac{\Psi(u)^b \Psi(w)^{a - 1}}{(u - w)^{2}} + O(N^{-1}).
\end{align*}
Finally, adding a (vanishing) term with $a = k + 1$ in the second from last equality, we obtain
\begin{align*}
X_1(\hchi_N) &= \sum_{a = 1}^{k + 1} \sum_{b = 0}^l \binom{k}{a - 1} \binom{l}{b} \oint \oint \frac{dw}{2\pi \ii} \frac{du}{2\pi \ii} \Xi(w)^{k - a + 1} \Xi(u)^{l - b} \frac{\Psi(u)^b \Psi(w)^{a - 1}}{(u - w)^{2}}  + O(N^{-1})\\
&= \oint \oint \frac{dw}{2\pi \ii} \frac{du}{2\pi \ii} \Big(\Xi(w) + \Psi(w)\Big)^k \Big(\Xi(u) + \Psi(u)\Big)^l \frac{1}{(u - w)^{2}} + O(N^{-1}),
\end{align*}
where both the $u$ and $w$-contours enclose $V_\chi$ and lie in $U$, and the $u$-contour is contained inside the $w$-contour. Notice now that
\begin{align*}
  X_2(\hchi_N) &= \sum_{a = 0}^k \sum_{b = 0}^l \binom{k}{a} \binom{l}{b} \sum_{i = 1}^N ab N^{a - k + b - l - 2} \we_{k - a, i}(r)\\
  &\phantom{==} \left.\oint \frac{du}{2\pi\ii} (l - b)! e_{l - b + 1}\Big(\frac{1}{u - r_*}\Big)_{*\neq i} \Psi(u)^{b - 1} F^{(1, 1)}(r_i, u) \Psi(r_i)^{a - 1} \right|_{r = \hchi_N} \\
  &= \sum_{a = 0}^k \sum_{b = 0}^l \binom{k}{a} \binom{l}{b} \sum_{i = 1}^N ab N^{a - k + b - l - 2} \we_{k - a, i}(r)\\
  &\phantom{==} \left.\oint \frac{du}{2\pi\ii} \sum_{p = 0}^{l - b + 1} (-1)^p \frac{(l - b + 1)_p}{(l - b + 1)(u - r_i)^p} \we_{l - b + 1 - p}(u, r)\Psi(u)^{b - 1} F^{(1, 1)}(r_i, u) \Psi(r_i)^{a - 1}\right|_{r = \hchi_N/N} \\
  &= \sum_{a = 0}^k \sum_{b = 0}^l \binom{k}{a} \binom{l}{b} ab N^{a - k + b - l - 2} \sum_{p = 0}^{l - b + 1} (-1)^p \\
  &\phantom{==} \oint \oint \frac{du}{2\pi\ii} \frac{dw}{2\pi\ii} \frac{\we_{k - a + 1}(w, \hchi_N)}{k - a + 1} \frac{(l - b + 1)_{p}}{(l - b + 1)(u - w)^p} \we_{l - b + 1 - p}(u, \hchi_N) \Psi(u)^{b - 1} F^{(1, 1)}(w, u) \Psi(w)^{a - 1},
\end{align*}
where both the $u$ and $w$-contours enclose $V_\chi$ and lie in $U$, and the $u$-contour is contained inside the $w$-contour.  Applying Lemma \ref{lem:power-asymp}, we find that
\begin{align*}
  X_2(\hchi_N) &= \sum_{a = 0}^k \sum_{b = 0}^l \binom{k}{a} \binom{l}{b} ab N^{-p} \sum_{p = 0}^{l - b + 1} (-1)^p  \oint \oint \frac{du}{2\pi\ii} \frac{dw}{2\pi\ii} \frac{\Xi(w)^{k - a + 1}}{k - a + 1}\\
  &\phantom{==}  \frac{(l - b + 1)_{p}}{(l - b + 1)(u - w)^p} \Xi(u)^{l - b + 1 - p} \Psi(u)^{b - 1} F^{(1, 1)}(w, u) \Psi(w)^{a - 1} + O(N^{-1})\\
  &= \sum_{a = 0}^k \sum_{b = 0}^l \binom{k}{a} \binom{l}{b} ab \oint \oint \frac{du}{2\pi\ii} \frac{dw}{2\pi\ii} \frac{\Xi(w)^{k - a + 1}}{k - a + 1} \frac{\Xi(u)^{l - b + 1}}{l - b + 1} \Psi(u)^{b - 1} F^{(1, 1)}(w, u) \Psi(w)^{a - 1} + O(N^{-1}),
\end{align*}
where both the $u$ and $w$-contours enclose $V_\chi$ and lie in $U$, and the $u$-contour is contained inside the $w$-contour. Adding (identically zero) terms with $a = k + 1$ and $b = l + 1$ in the last equality, we find that
\begin{align*}
X_2(\hchi_N) &= \sum_{a = 1}^{k + 1} \sum_{b = 1}^{l + 1} \binom{k}{a - 1} \binom{l}{b - 1} \oint \oint \frac{du}{2\pi\ii} \frac{dw}{2\pi\ii} \Xi(w)^{k - a + 1} \Xi(u)^{l - b + 1} \Psi(u)^{b - 1} F^{(1, 1)}(w, u) \Psi(w)^{a - 1} + O(N^{-1})\\
  &= \oint \oint \Big(\Xi(w) + \Psi(w)\Big)^k \Big(\Xi(u) + \Psi(u)\Big)^l F^{(1, 1)}(w, u) \frac{du}{2\pi \ii} \frac{dw}{2\pi \ii} + O(N^{-1}),
\end{align*}
where both the $u$ and $w$-contours enclose $V_\chi$ and lie in $U$, and the $u$-contour is contained inside the $w$-contour.  Putting these computations together, we find the covariance to be
\begin{multline*}
  \Cov(p_k(x), p_l(x)) = X_1(\hchi_N) + X_2(\hchi_N) + O(N^{-1})\\
  = \oint \oint \frac{dw}{2\pi \ii} \frac{du}{2\pi \ii} \Big(\Xi(u) + \Psi(u)\Big)^l \Big(\Xi(w) + \Psi(w)\Big)^k \Big(\frac{1}{(u - w)^2} + F^{(1, 1)}(w, u)\Big)+ O(N^{-1}),
\end{multline*}
where both the $u$ and $w$-contours enclose $V_\chi$ and lie in $U$ with the $u$-contour contained inside the $w$-contour.

\subsubsection{Establishing Gaussianity}

For any $k_1, \ldots, k_n$, define the cumulant by 
\[
\cK_{k_1, \ldots, k_n} := \sum_{U_1 \sqcup \cdots \sqcup U_t = \{1, \ldots, N\}} (-1)^{t - 1} (t - 1)! \prod_{i = 1}^t \EE\left[\prod_{j \in U_i} p_{k_j}(x)\right].
\]
We now check that $\cK_{k_1, \ldots, k_n} = o(1)$ for $n \geq 3$.  By an analogue of \cite[Lemma 3.10]{BG17}, the only terms in $\cK_{k_1, \ldots, k_n}$ with non-zero contribution correspond to a forest composed of a single tree and hence $t = 1$ and $U_1 = \{1, \ldots, N\}$.  We now check that these terms are $o(1)$.  By construction, if a forest has a single tree, its index set is
\[
K = \{j_1\} \cup \cdots \cup \{j_n\} \cup A_1 \cup \cdots \cup A_n
\]
and it has CLT weight
\[
\sum_{m  = 1}^n |A_m| + \sum_{i, m} (b_i^m - 1) + \sum_m c_m + \#\{\text{single edges in $T$}\},
\]
where we notice that
\[
\sum_m c_m \geq \#\{\text{edges in $T$}\}.
\]
On the other hand, the size of the index set is at most
\[
1 + \sum_{m = 1}^n |A_m| + \#\{\text{single edges in $T$}\}.
\]
By Lemma \ref{lem:tree-prop}(d), the total contribution of these $O(N^{1 + \sum_{m = 1}^n |A_m| + \#\{\text{single edges in $T$}\}})$ terms has order
\[
O\Big(N^{\#\{1 - \sum_{i, m} (b_i^m - 1) - \sum_m c_m\}}\Big),
\]
where we notice that
\[
1 - \sum_{i, m} (b_i^m - 1) - \sum_m c_m \leq 1 - \#\{\text{edges in $T$}\} \leq 2 - n.
\]
This implies that the total contribution to the cumulant is $O(N^{2 - n})$, which vanishes as desired for $n \geq 3$.

\subsection{Central limit theorems for growing measures}

In this section, we adapt our theorems to a case where the measures grow in a specific manner.  Choose $M = \omega(1)$, where we recall that $f = \omega(g)$ if $\lim_{N \to \infty} f / g = \infty$, so that $M, N \to \infty$ simultaneously.  As before, let $d\mu_N$ be a sequence of measures with multivariate Bessel generating function $\phi_{\chi, N}(s)$.  Suppose that there exist a sequence of smooth measures $d\lambda_N'$ with multivariate Bessel generating function $\phi_{\chi, N}(s)^M$, and let $d\lambda_N$ be the pushforward of $d\lambda_N'$ under the map $\lambda' \mapsto \frac{1}{M} \lambda'$.  In this setting, we have the following extensions of Theorem \ref{thm:lln} and Theorem \ref{thm:clt}.

\begin{theorem} \label{thm:lln-many}
  If $d\mu_N$ is LLN-appropriate for $\chi_N$, then if $x$ is distributed according to $d\lambda_N$, we have the following convergence in probability
  \[
  \lim_{N \to \infty} \frac{1}{N} p_k(x) = \lim_{N \to \infty} \frac{1}{N} \EE[p_k(x)] = \fp_k' := \oint \Xi(u) \Psi(u)^k \frac{du}{2\pi \ii},
  \]
where the $u$-contour encloses $V_\chi$ and lies within $U$.  In addition, the random measures $\frac{1}{N} \sum_{i = 1}^N \delta_{x_i}$ with $x$ distributed according to $d\lambda_N$ converge in probability to a deterministic compactly supported measure $d\lambda$ with $\int x^k d\lambda(x) = \fp_k'$. 
\end{theorem}
\begin{proof}
Define $\Phi_{\chi, N}(s) := \phi_{\chi, N}(s)^M$.  As in our previous analysis, we have that 
\[
\EE[p_{k_1}(x) \cdots p_{k_n}(x)] = M^{-k_1 - \cdots - k_n} \frac{D_{k_1} \cdots D_{k_n} \Phi_{\chi, N}(\chi_N)}{\Phi_{\chi, N}(\chi_N)},
\]
which may be written as a linear combination of terms encoded by the graphical calculus of Section \ref{sec:graph}.  The coefficients of this linear combination are independent of $N$ and are monomials in $M$.  By Lemma \ref{lem:tree-prop}(f), the exponent in $M$ is bounded above by $-W_L$, where $W_L$ is the LLN weight of the term. Since $M = \omega(1)$, the leading order terms in our situation are the subset of the leading order terms of the case $M = 1$ corresponding to the terms whose coefficient has the maximal power of $M$. In particular, this implies by the proof of Theorem \ref{thm:lln} that
\begin{align*}
  \frac{1}{N} \EE[p_k(x)] &= \frac{M^{-k}}{k + 1} \oint \Big(\Xi(u) + M \Psi(u)\Big)^{k + 1} \frac{du}{2\pi\ii} + O(N^{-1})\\
  &= \oint \Xi(u) \Psi(u)^k \frac{du}{2\pi \ii} + O(M^{-1} + N^{-1}),
\end{align*}
where the contour encloses $V_\chi$ and lies within a neighborhood $U$ of $V_\chi$, and we notice that
\[
\oint \Psi(u)^{k+1} \frac{du}{2\pi\ii}
\]
vanishes because $\Psi$ is holomorphic on $U$.  Similarly, we have from the proof of Theorem \ref{thm:lln} that $\frac{1}{N^2}\Cov\Big(p_k(x), p_k(x)\Big) = O(N^{-1})$, which yields the desired.
\end{proof}

\begin{theorem} \label{thm:clt-many}
If the measures $d\mu_N$ are CLT-appropriate for $\chi_N$, then if $x$ is distributed according to $d\lambda_N$, the collection of random variables
\[
\{M^{1/2}(p_k(x) - \EE[p_k(x)])\}_{k \in \NN}
\]
converges in the sense of moments to the Gaussian vector with zero mean and covariance
\begin{multline*}
  \lim_{N \to \infty} \Cov\Big(M^{1/2}p_k(x), M^{1/2}p_l(x)\Big)\\ = kl \oint \oint \Xi(u) \Xi(w) \Psi(u)^{k - 1} \Psi(w)^{l - 1} F^{(1, 1)}(u, w) \frac{du}{2\pi\ii} \frac{dw}{2\pi\ii} + kl \oint \Xi(u) \Psi(u)^{k + l - 2} \Psi'(u) \frac{du}{2\pi\ii},
\end{multline*}
where the $u$ and $w$-contours enclose $V_\chi$ and lie within $U$, and the $u$-contour is contained inside the $w$-contour.
\end{theorem}
\begin{proof}
  As in the proof of Theorem \ref{thm:clt}, we analyze for $n \geq 2$ the terms which appear in the expansion of the cumulant $\cK_{k_1, \ldots, k_n}$ of $p_{k_1}(x), \ldots, p_{k_n}(x)$ as encoded in the graphical calculus of Section \ref{sec:graph}.  In particular, each term with non-zero contribution corresponds to a single tree $T$ on $n$ vertices; such terms have LLN weight
  \begin{equation} \label{eq:lln-wt-calc}
W_L = \sum_{m = 1}^n |A_m| + \sum_{m = 1}^n \sum_i (b_i^m - 1) + \sum_{m = 1}^n c_m \geq \#\{\text{edges in $T$}\} = n - 1.
  \end{equation}
  By Lemma \ref{lem:tree-prop}(f), this implies that the exponent of $M$ in each term with non-zero contribution to an $n^\text{th}$ mixed cumulant is at most $1 - n$.

  We now analyze the covariance and the higher cumulants separately. For the higher cumulants, by the proof of Theorem \ref{thm:clt} and what we just showed, any $n^\text{th}$ cumulant of $\{p_k(x)\}$ has order at most $O(N^{2 - n} M^{1 - n})$, meaning that any $n^\text{th}$ cumulant of $\{M^{1/2}(p_k(x) - \EE[p_k(x)])\}$ has order at most $O(N^{2 - n} M^{1 - n/2}) = o(1)$ for $n \geq 3$.  

  For the covariance, we again analyze terms in the expansion as in the proof of Theorem \ref{thm:clt}.  Notice that the coefficients of each term are simply multiplied by a power of $M$ equal to the number of appearances of a (derivative) of $\log \phi_{\chi, N}^I(s)$ which appears in it.  As we just argued, this power is non-positive, which means that any leading order term in this expansion comes from a leading order term in the case $M = 1$ and that all other terms have order at most $O(N^{-1})$.  This implies that
\begin{multline} \label{eq:covar-lya-calc}
  \Cov\Big(M^{1/2} p_k(x), M^{1/2} p_l(x)\Big) = M^{-k - l + 1} \oint \oint \Big(\Xi(u) + M\Psi(u)\Big)^l\\
  \left(\Xi(w) + M\Psi(w)\Big)^k \Big(\frac{1}{(u - w)^2} + M F^{(1, 1)}(w, u)\right) \frac{dw}{2\pi\ii} \frac{du}{2\pi\ii} + O(N^{-1}),
\end{multline}
where the $u$ and $w$-contours enclose $V_\chi$ and lie in $U$, and the $u$-contour is contained inside the $w$-contour.  Observe now that for $1 \leq a \leq l$ and $1 \leq b \leq k$ we have
\begin{align*}
  \oint \oint \Psi(u)^l \Psi(w)^k F^{(1, 1)}(w, u) \frac{dw}{2\pi \ii} \frac{du}{2\pi \ii} &= 0 \\
  \oint \oint \Xi(u)^a \Psi(u)^{l - a} \Psi(w)^k F^{(1, 1)}(w, u) \frac{dw}{2\pi \ii} \frac{du}{2\pi \ii} &= 0 \\
  \oint \oint \Xi(w)^b \Psi(u)^{l} \Psi(w)^{k - b} F^{(1, 1)}(w, u) \frac{dw}{2\pi \ii} \frac{du}{2\pi \ii} &= 0 \\
  \oint \oint \Psi(u)^{l} \Psi(w)^k \frac{1}{(u - w)^2} \frac{dw}{2\pi \ii} \frac{du}{2\pi \ii} &= 0\\
  \oint \oint \Xi(w) \Psi(u)^l \Psi(w)^{k - 1} \frac{1}{(u - w)^2} \frac{du}{2\pi\ii} \frac{dw}{2\pi\ii} &= 0,
\end{align*}
where the first, third, fourth, and fifth integrals vanish because the integrand is holomorphic inside the $u$-contour, and the second integral vanishes because the integrand is holomorphic inside the $w$-contour.  We conclude that
\begin{align*}
  \Cov\Big(M^{1/2} p_k(x), M^{1/2} p_l(x)\Big) &= kl \oint \oint \Xi(u) \Xi(w) \Psi(u)^{l - 1} \Psi(w)^{k - 1} F^{(1, 1)}(w, u)\frac{du}{2\pi\ii} \frac{dw}{2\pi\ii} \\
  &\phantom{=} + l \oint \oint \Xi(u) \Psi(u)^{l - 1} \Psi(w)^k \frac{1}{(u - w)^2} \frac{du}{2\pi\ii} \frac{dw}{2\pi\ii} + O(N^{-1}).
\end{align*}
Computing the $w$-integral as the residue at $w = u$ in the second integral yields the claim.
\end{proof}

\section{Asymptotics of multivariate Bessel functions} \label{sec:mvb-asymp}

\subsection{Statement of the results}

Recall that $\rho = (N - 1, \ldots, 0)$.  In this section, we study the asymptotics of the normalized multivariate Bessel function
\[
\frac{\cB(\mu, \lambda)}{\cB(\rho, \lambda)}
\]
when $\mu$ deviates from $\rho$ in only $k$ coordinates.  In this case, we may parametrize $\mu$ as  
\begin{equation} \label{eq:mu-def}
\mu := \Big(a_1, a_2, \ldots, a_k, N - 1, \ldots, \widehat{b_1}, \ldots, \widehat{b_2}, \ldots, \widehat{b_k}, \ldots, 0\Big).
\end{equation}
for some $a_1, \ldots, a_k$ and $b_1 > \cdots > b_k \in \{0, \ldots, N - 1\}$, where $\widehat{b_i}$ denotes that the coordinate $b_i$ is missing.  We obtain asymptotics for these $\mu$ from asymptotics for 
\[
\mu_{a, b} := \Big(a, N - 1, \ldots, \hat{b}, \ldots, 0\Big)
\]
differing from $\rho$ in a single coordinate.  Our first two results of this section relate these two cases and give a new contour integral expression (\ref{eq:double_integral_transformed_final}) for the normalized multivariate Bessel function when $\mu = \mu_{a, b}$ which will be crucial to our asymptotic analysis.

\begin{prop} \label{prop:mvb-det}
For two $k$-tuples $a_1, \ldots, a_k \in \CC$ and $b_1 > \cdots > b_k \in \{0, \ldots, N - 1\}$, we have that
\[
\frac{\cB(\mu, \lambda)}{\cB(\rho, \lambda)} = (-1)^{\frac{k(k - 1)}{2}} \frac{\prod_{m \neq l} (a_m - b_l)}{\Delta(a_m)\Delta(b_m)} \det\left(\frac{a_i - b_i}{a_i - b_j}\frac{\cB(\mu_{a_i, b_j}, \lambda)}{\cB(\rho, \lambda)}\right)_{i, j = 1}^k. 
\]
\end{prop}

\begin{theorem} \label{thm:mvb-int}
For any $a \in \CC \setminus \{N - 1, \ldots, 0\}$ and $b \in \{N - 1, \ldots, 0\}$, we have that
\begin{multline} \label{eq:double_integral_transformed_final}
\frac{\cB(\mu_{a, b}, \lambda)}{\cB(\rho, \lambda)} = (-1)^{N - b + 1} \frac{\Gamma(N - b) \Gamma(b + 1)\Gamma(a - N + 1)(a - b)}{\Gamma(a + 1)}\\ \oint_{\{e^{\lambda_i}\}} \frac{dz}{2\pi \ii} \oint_{\{0, z\}} \frac{dw}{2 \pi \ii} \cdot \frac{z^a w^{-b - 1}}{z - w}\cdot \prod_{i = 1}^N \frac{w - e^{\lambda_i}}{z - e^{\lambda_i}},
\end{multline}
where the $z$-contour encloses the poles at $e^{\lambda_i}$ and avoids the negative real axis and the $w$-contour encloses both $0$ and the $z$-contour.
\end{theorem}

\begin{remark}
Proposition \ref{prop:mvb-det} and Theorem \ref{thm:mvb-int} are parallel to \cite[Theorem 3.7]{GP} and \cite[Theorem 3.8]{GP}.  We note that Proposition \ref{prop:mvb-det} does not involve differentiation and thus takes a simpler form than \cite[Theorem 3.7]{GP}; in particular checking that the expression of Proposition \ref{prop:mvb-det} equals $1$ when $\mu = \rho$ is straightforward, while the corresponding check in \cite[Theorem 3.7]{GP} is rather delicate.  On the other hand, Theorem \ref{thm:mvb-int} is more complicated than \cite[Theorem 3.8]{GP}, as a single contour integral is replaced by a double contour integral.
\end{remark}

The final result of this section, Theorem \ref{thm:multi-asymp} gives asymptotics under certain asymptotic
assumptions on $\lambda$ in the case where finitely many $b_i$ are perturbed.  We first specify these
asymptotic assumptions.  For a sequence $\lambda^1, \lambda^2, \ldots$ where $\lambda^N$ has length $N$, define
$d\rho_N$ and $d\wrho_N$ by
\[
d\rho_N := \frac{1}{N} \sum_{i = 1}^N \delta_{\lambda_i^N} \qquad d\wrho_N := \frac{1}{N} \sum_{i = 1}^N \delta_{e^{\lambda_i^N}}.
\]
as the empirical measure of $\lambda^N$ and its pushforward under the exponential map.  We make the following assumption on $d\rho_N$.

\begin{assump} \label{ass:measure}
The measures $d\rho_N$ have support contained within a fixed finite interval $I$ and converge weakly to a compactly supported measure $d\rho$. 
\end{assump}

For $\wb_1 > \cdots > \wb_k \in \frac{1}{N} \ZZ$ with $\wb_i N \in \{0, \ldots, N - 1\}$, define the $N$-tuple
  \begin{equation} \label{eq:mult-ab-def}
  \mu_{\wa_1, \ldots, \wa_k; \wb_1, \ldots, \wb_k} := \Big(\wa_1 N, \wa_2 N, \ldots, \wa_k N, N - 1, \ldots, \widehat{\wb_1 N}, \ldots, \widehat{\wb_2 N}, \ldots, \widehat{\wb_k N}, \ldots, 0\Big).
  \end{equation}
  We will study the $N \to \infty$ asymptotics of the quantity
  \[
  \wB_N(\wa_1, \ldots, \wa_k; \wb_1, \ldots, \wb_k) := \frac{\cB(\mu_{\wa_1, \ldots, \wa_k; \wb_1, \ldots, \wb_k}, \lambda^N)}{\cB(\rho, \lambda^N)}.
  \]
Our answer will be expressed in terms of the moment generating function of the exponential moments and its $S$-transform, defined by
\[
M_{\wrho_N}(z) := \int \frac{e^{x} z}{1 - e^{x} z} d\rho_N(x) \qquad S_{\wrho_N}(z) := \frac{1 + z}{z} M_{\wrho_N}^{-1}(z),
\]
where by $M_{\wrho_N}^{-1}$ we mean the functional inverse of $M_{\wrho_N}$ as a (uniformly convergent) power series in $z$. We also consider their limiting versions defined by 
\[
M_{\wrho}(z) := \int \frac{e^{x} z}{1 - e^{x}z} d\rho(x) \qquad S_{\wrho}(z) := \frac{1 + z}{z} M^{-1}_{\wrho}(z).
\]
Some analytic properties of $M_{\wrho_N}(z)$, $M_{\wrho}(z)$, $S_{\wrho_N}(z)$, and $S_{\wrho}(z)$ are discussed in Lemma \ref{lem:invert} below. Define finally the function
\begin{equation} \label{eq:f-def}
\wPsi_{\rho_N}(c) := c \log S_{\wrho_N}(c - 1) + \int \log\Big((1 - c)(-M^{-1}_{\wrho_N}(c - 1)^{-1} + e^s)\Big) d\rho_N(s),
\end{equation}
which is holomorphic on a neighborhood of $[0, 1]$ by Lemmas \ref{lem:invert} and \ref{lem:s-val} and the fact that
\[
\wPsi_{\rho_N}(c) = c \log S_{\wrho_N}(c - 1) + \int \log\Big(c S_{\wrho_N}(c - 1)^{-1} + (1 - c) e^s\Big) d\rho_N(s).
\]

\begin{theorem} \label{thm:multi-asymp}
If $\lambda^N$ satisfies Assumption \ref{ass:measure}, for $\wb_1 > \cdots > \wb_k \in \frac{1}{N} \ZZ$ with $\wb_i N \in \{0, \ldots, N - 1\}$, as a function of $\wa_1, \ldots, \wa_k$, the quantity $\wB_N(\wa_1, \ldots, \wa_k; \wb_1, \ldots, \wb_k)$ admits a holomorphic extension to an open neighborhood $U^k$ of $[0, 1]^k$ such that uniformly on compact subsets and uniformly in $\wb_1, \ldots, \wb_k$ we have
  \begin{multline*}
    \wB_N(\wa_1, \ldots, \wa_k; \wb_1, \ldots, \wb_k) = \frac{\prod_{m \neq l} (\wa_m - \wb_l)}{\Delta(\wa) \Delta(\wb)} \frac{\Delta(M^{-1}_{\wrho_N}(\wa_i - 1)) \Delta(M^{-1}_{\wrho_N}(\wb_i - 1))}{\prod_{i, j} [M^{-1}_{\wrho_N}(\wa_i - 1) - M^{-1}_{\wrho_N}(\wb_j - 1)]} \\ \prod_{i = 1}^k \left[\frac{(\wa_i - \wb_i) }{\sqrt{M_{\wrho_N}'(M_{\wrho_N}^{-1}(\wa_i - 1)) M_{\wrho_N}'(M_{\wrho_N}^{-1}(\wb_i - 1))}}  \frac{\sqrt{S_{\wrho_N}(\wb_i - 1)} e^{N \wPsi_{\rho_N}(\wb_i)}}{\sqrt{S_{\wrho_N}(\wa_i - 1)} e^{N \wPsi_{\rho_N}(\wa_i)}}\right][1 + o(1)].
  \end{multline*}
\end{theorem}

\begin{lemma} \label{lem:limit-analytic}
The expression
\begin{multline*}
 \frac{\prod_{m \neq l} (\wa_m - \wb_l)}{\Delta(\wa) \Delta(\wb)} \frac{\Delta(M^{-1}_{\wrho_N}(\wa_i - 1)) \Delta(M^{-1}_{\wrho_N}(\wb_i - 1))}{\prod_{i, j} [M^{-1}_{\wrho_N}(\wa_i - 1) - M^{-1}_{\wrho_N}(\wb_j - 1)]} \\ \prod_{i = 1}^k \left[\frac{(\wa_i - \wb_i) }{\sqrt{M_{\wrho_N}'(M_{\wrho_N}^{-1}(\wa_i - 1)) M_{\wrho_N}'(M_{\wrho_N}^{-1}(\wb_i - 1))}}  \frac{\sqrt{S_{\wrho_N}(\wb_i - 1)} e^{N \wPsi_{\rho_N}(\wb_i)}}{\sqrt{S_{\wrho_N}(\wa_i - 1)} e^{N \wPsi_{\rho_N}(\wa_i)}}\right]
\end{multline*}
in Theorem \ref{thm:multi-asymp} is holomorphic in $\wa_i, \ldots, \wa_k$ on an open complex neighborhood $U^k$ of $[0, 1]^k$.
\end{lemma}
\begin{proof}
Note that
\[
\frac{\prod_{m \neq l} (\wa_m - \wb_l)}{\Delta(\wa) \Delta(\wb)} \frac{\Delta(M_{\wrho_N}^{-1}(\wa - 1))\Delta(M_{\wrho_N}^{-1}(\wb - 1))}{\prod_{i \neq j} (M_{\wrho_N}^{-1}(\wa_i - 1) - M_{\wrho_N}^{-1}(\wb_j - 1))}
\]
has no poles in codimension $1$ as a function of $\wa_1, \ldots, \wa_k, \wb_1, \ldots, \wb_k$ on $U^{2k}$, hence is holomorphic by Riemann's second extension theorem (see \cite[Theorem 7.1.2]{GR84}).  In addition, by Lemma \ref{lem:invert}, $M_{\wrho_N}^{-1}(u - 1)$ is meromorphic on a neighborhood $U \supset [0, 1]$ uniformly in $N$ and avoids $0$ and $\infty$ aside from a simple zero at $1$ and a simple pole at $0$.  Noting that $M_{\wrho_N}^{-1}(u - 1) \in [0, -\infty)$ for $u \in [0, 1]$ and that
\begin{equation} \label{eq:m-der}
M_{\wrho_N}'(M_{\wrho_N}^{-1}(u - 1)) = \int \frac{e^s}{(1 - e^s M_{\wrho_N}^{-1}(u - 1))^2} d\rho_N(s),
\end{equation}
on a neighborhood of $U^2$ the function
\[
\frac{(a - b)^2}{(M_{\wrho_N}^{-1}(a - 1) - M_{\wrho_N}^{-1}(b - 1))^2 M'_{\wrho_N}(M^{-1}_{\wrho_N}(a - 1)) M'_{\wrho_N}(M^{-1}_{\wrho_N}(b - 1))}
\]
has no zeros and has poles contained in the union of hyperplanes $\{a = b\} \cup \{a = 0\} \cup \{a = 1\} \cup \{b = 0\} \cup \{b = 1\}$.  By (\ref{eq:m-der}), we may easily check that it has no poles in codimension $1$ and hence no poles on all of $U^2$.  Similarly, $S_{\wrho_N}(u - 1)$ is holomorphic and non-vanishing on some $U$.  Therefore, for a small enough neighborhood $U$ of $[0, 1]$, we may choose branches of the functions
\[
\frac{a - b}{M_{\wrho_N}^{-1}(a - 1) - M_{\wrho_N}^{-1}(b - 1)} \frac{1}{\sqrt{M_{\wrho_N}'(M_{\wrho_N}^{-1}(a - 1)) M_{\wrho_N}'(M_{\wrho_N}^{-1}(b - 1))}}
\]
and $\sqrt{S_{\wrho_N}(u - 1)}$ so that they are holomorphic on $U^2$ and $U$, respectively.  Combining these with the fact that $\wPsi_{\wrho_N}(u)$ is holomorphic shows that the expression is holomorphic on $U^k$, as needed.
\end{proof}

The remainder of this section is devoted to the proofs of Proposition \ref{prop:mvb-det}, Theorem \ref{thm:mvb-int}, and Theorem \ref{thm:multi-asymp}.   

\subsection{Multivariate Bessel functions and hook Schur functions}

In this section we derive some properties of multivariate Bessel functions through their connection to certain symmetric functions called Schur functions.  Recall that a signature $\mu = (\mu_1, \ldots, \mu_N)$ is an $N$-tuple of integers such that $\mu_1 \geq \cdots \geq \mu_N$.  For such a signature, we may define the Schur function $s_\mu(x)$ by
\begin{equation} \label{eq:schur-def}
s_\mu(x) := \frac{\det(x_i^{\mu_j + N - j})_{i, j = 1}^N}{\prod_{1 \leq i < j \leq N} (x_i - x_j)}
\end{equation}
as a function on $x_1, \ldots, x_N$. By the Weyl character formula, it is the character of the representation of $U(N)$ with highest weight $\mu$.   Recalling that $\rho = (N - 1, \ldots, 0)$, if $\mu - \rho$ is a signature, the multivariate Bessel function may be expressed in terms of Schur functions as
\begin{equation} \label{eq:mvb-schur}
\cB(\mu, \lambda) = \frac{\Delta(e^\lambda) \Delta(\rho)}{\Delta(\lambda) \Delta(\mu)} s_{\mu - \rho}(e^\lambda).
\end{equation}
We will consider cases where $\mu - \rho$ is a partition, meaning that $\mu_N - \rho_N \geq 0$.

For a Young diagram of a partition $\lambda$ with $k$ boxes on its main diagonal, we may represent $\lambda$ in Frobenius notation
\[
\lambda = (\alpha_1, \ldots, \alpha_k \mid \beta_1, \ldots, \beta_k),
\]
where $\alpha_i = \lambda_i - i$, $\beta_i = \lambda_i' - i$, and $\lambda'$ denotes the dual (transposed) partition to $\lambda$.  If $a_1 - N + 1 > \cdots > a_k - N + k > N$, then for the $\mu$ defined in (\ref{eq:mu-def}) we see that
\[
\mu - \rho = (a_1 - N + 1, a_2 - N + 2, \ldots, a_k - N + k, k, \ldots, k, k - 1, \ldots, k - 1, \ldots, 1, \ldots, 1, 0, \ldots, 0),
\]
where $k$ appears $N - b_1 - 1$ times and $k - i$ appears $b_i - b_{i + 1} - 1$ times. In Frobenius notation, we have that
\[
\mu - \rho = \Big(a_1 - N, \ldots, a_k - N \mid N - 1 - b_k, \ldots, N - 1 - b_1\Big).
\]
The following lemma expresses a Schur function corresponding to a partition with $k$ hooks in terms of individual hook Schur functions $s_{\alpha \mid \beta}(x)$, which will allow us to prove Proposition \ref{prop:mvb-det}. 

\begin{lemma}[{\cite[Example I.3.9]{Mac}}] \label{lem:hook-form}
If $\lambda = (\alpha_1, \ldots, \alpha_k \mid \beta_1, \ldots, \beta_k)$ in Frobenius notation, we have that
\[
s_\lambda(x) = \det(s_{\alpha_i \mid \beta_j}(x))_{i, j = 1}^k,
\]
where
\[
s_{\alpha \mid \beta}(x) = \sum_{i = 0}^\beta (-1)^i h_{\alpha + 1 + i}(x) e_{\beta - i}(x),
\]
where $h_k(x)$ and $e_k(x)$ denote the complete homogeneous and elementary symmetric polynomial, respectively.
\end{lemma}

In order to convert Lemma \ref{lem:hook-form} into a similar formula for multivariate Bessel functions, we require the following version of Carlson's Theorem from \cite{Car14} as stated in \cite{Pil05}.  This result allows us to show that identities of analytic functions which hold on certain specializations hold everywhere.

\begin{theorem}[Carlson's Theorem] \label{thm:carlson}
Suppose $f(z)$ is holomorphic in a neighborhood of $\{z \in \CC \mid \Re z \geq 0\}$ and satisfies
\begin{itemize}
\item[(a)] exponential growth, meaning $f(z) = O(e^{C_1|z|})$ for $z \in \CC$ for some $C_1 > 0$;

\item[(b)] exponential growth of order at most $\pi$ along the imaginary axis, meaning $f(\ii y) = O(e^{C_2|y|})$ for $y \in \RR$ for some $0 < C_2 < \pi$.
\end{itemize}
If $f(n) = 0$ for all non-negative integers $n$, then $f(z) = 0$ identically.
\end{theorem}

\begin{proof}[Proof of Proposition \ref{prop:mvb-det}]
First, treating both sides of the desired equality as functions of $a_1 - N + 1, \ldots, a_k - N + k$, by considering the explicit definition (\ref{eq:mvb-def}), we see they satisfy both exponential growth conditions of Carlson's theorem separately in each variable.  Applying Carlson's theorem to the difference in each variable sequentially, it suffices to check this for integer $a_i$ such that $a_1 - N + 1 > \cdots > a_k - N + k > N$ so that $\mu - \rho$ is a partition.  In this case, combining the relation (\ref{eq:mvb-schur}) and Lemma \ref{lem:hook-form}, we see that
\begin{equation} \label{eq:mvb-hook-det}
\frac{\cB(\mu, \lambda)}{\cB(\rho, \lambda)} = \frac{\Delta(\rho)}{\Delta(\mu)} \det\Big(s_{a_i - N \mid N - 1 - b_{k + 1 - j}}(e^\lambda)\Big)_{i, j = 1}^k.
\end{equation}
We see that
\[
\frac{\Delta(\rho)}{\Delta(\mu)} = \frac{\Delta(b_m)}{\Delta(a_m)} \prod_{m = 1}^k (-1)^{N - b_m + m} \prod_{i \notin \{b_1, \ldots, b_k\}} \frac{b_m - i}{a_m - i} = \frac{\prod_{m \neq l} (a_m - b_l)}{\Delta(b_m) \Delta(a_m)} \prod_{m = 1}^k (-1)^{N - b_m + 1} \prod_{i \neq b_m} \frac{b_m - i}{a_m - i}.
\]
Moving these factors into the determinant in (\ref{eq:mvb-hook-det}) and noting that
\[
\frac{\cB(\mu_{a, b}, \lambda)}{\cB(\rho, \lambda)} = \frac{\Delta(\rho)}{\Delta(\mu_{a, b})} s_{a - N \mid N - 1 - b}(e^\lambda) = (-1)^{N - b + 1} \prod_{i \neq b} \frac{b - i}{a - i} s_{a - N \mid N - 1 - b}(e^\lambda),
\]
we find that
\[
\frac{\cB(\mu, \lambda)}{\cB(\rho, \lambda)} = \frac{\prod_{m \neq l}(a_m - b_l)}{\Delta(a_m)\Delta(b_m)} \det\Big(\frac{a_i - b_i}{a_i - b_{k + 1 - j}}\frac{\cB(\mu_{a_i, b_{k + 1 - j}}, \lambda)}{\cB(\rho, \lambda)}\Big)_{i, j = 1}^k,
\]
which yields the result after applying the reindexing $j \mapsto k + 1 - j$.
\end{proof}

\subsection{An integral formula for multivariate Bessel functions}

In this section we prove Theorem \ref{thm:mvb-int} by combining index-variable duality for Schur functions and the following integral formula for hook Schur functions.

\begin{prop} \label{prop:int-form}
We have that
\[
s_{\alpha \mid \beta}(x) = \oint_{\{0\}} \oint_{\{0\}} \frac{dz dw}{(2 \pi \ii)^2} (-1)^\beta z^{-\alpha - \beta - 2} w^{-N - 1} \frac{1 - w^{\beta + 1}}{1 - w} \prod_{i = 1}^N \frac{w - x_i z}{1 - x_i z},
\]
where both contours are small circles around $0$.
\end{prop}
\begin{proof}
We recall that for $H(x, z) := \prod_{i = 1}^N (1 - z x_i)^{-1}$, we have that
\[
h_l(x) = \frac{1}{2 \pi \ii} \oint_{\{0\}} \frac{H(x, z)}{z^{l + 1}} dz.
\]
We conclude by Lemma \ref{lem:hook-form} that
\begin{align*}
s_{\alpha \mid \beta}(x) &=  \oint_{\{0\}} \frac{dz}{2 \pi \ii} \sum_{j = 0}^\beta (-1)^j \frac{e_{\beta - j}(x)}{z^{\alpha + 2 + j}} H(x, z) \\
&= \oint_{\{0\}} \frac{dz}{2\pi \ii} H(x, z) z^{-\alpha - \beta - 2} (-1)^\beta \sum_{j = 0}^\beta (-1)^{\beta - j} e_{\beta - j} z^{\beta - j}]\\
&= \oint_{\{0\}} \oint_{\{0\}} \frac{dz dw}{(2\pi \ii)^2} H(x, z) z^{-\alpha - \beta - 2}(-1)^\beta \prod_{i = 1}^N (w - x_iz) \Big(w^{-N - 1} + \cdots + w^{-N - 1 + \beta}\Big)\\
&= \oint_{\{0\}} \oint_{\{0\}} \frac{dz dw}{(2 \pi \ii)^2} (-1)^\beta z^{-\alpha - \beta - 2} w^{-N - 1} \frac{1 - w^{\beta + 1}}{1 - w} \prod_{i = 1}^N \frac{w - x_i z}{1 - x_i z}. \qedhere
\end{align*}
\end{proof}

We may now use Proposition \ref{prop:int-form} to prove Theorem \ref{thm:mvb-int}.

\begin{proof}[Proof of Theorem \ref{thm:mvb-int}]
  First, suppose that $a \geq N$ is a positive integer.  By (\ref{eq:mvb-schur}) and Proposition \ref{prop:int-form}, we have that
\[
\frac{\cB(\mu_{a, b}, \lambda)}{\cB(\rho, \lambda)} = (-1)^{N - b - 1} \prod_{i \neq b} \frac{b - i}{a - i}\, s_{a - N \mid N - 1 - b}(e^\lambda) = J(a, b) I(a, b),
\]
with $I(a, b)$ and $J(a, b)$ defined by
\begin{align*}
I(a, b) &:= \oint_{\{0\}} \oint_{\{0\}} \frac{dz}{2\pi \ii} \frac{dw}{2\pi \ii} z^{b - a - 1} w^{-N - 1} \frac{1 - w^{N - b}}{1 - w} \prod_{i = 1}^N \frac{w - e^{\lambda_i}z}{1 - e^{\lambda_i}z}\\
J(a, b) &:= \prod_{i \neq b} \frac{b - i}{a - i},
\end{align*}
where both contours are small circles around $0$.  Note that the integrand has $z$-poles at $0$ and $e^{-\lambda_i}$ and $w$-poles at $0$ and $\infty$. We deform the $z$-contour through infinity, so that it now encloses only $e^{-\lambda_i}$ instead of $0$ and the sign changes. Simultaneously we deform the $w$-contour to be a very large circle around the origin and reverse its direction so it now encloses the pole at $\infty$ rather than $0$. Split now the double contour integral into two according to $1-w^{N-b}=(1)+ (-w^{N-b})$.  On the new deformed $w$-contour, the first part vanishes, as there is no longer a pole at $\infty$.  Together, these manipulations yield
\[
I(a, b) = - \oint_{\{e^{-\lambda_i}\}} \frac{dz}{2\pi \ii} \oint_{\{\infty\}} \frac{dw}{2 \pi \ii} z^{-a + b - 1} w^{-b - 1} \frac{1}{1 - w} \prod_{i = 1}^N \frac{w - e^{\lambda_i} z}{1 - e^{\lambda_i} z}.
\]
Let us make the change of variables
\[
 \tz = \frac{1}{z},\quad \tw= \frac{w}{z}; \qquad \qquad z=\frac{1}{\tilde  z},\quad w=\frac{\tw}{\tz}.
\]
We get for positive integer $a \geq N$ that 
\begin{equation}
\label{eq:double_integral_transformed}
I(a, b) = \oint_{\{e^{\lambda_i}\}} \frac{d\tz}{2\pi \ii (\tz)^2} \oint_{\{\infty\}} \frac{d\tw}{2 \pi \ii (\tz)} {\tz}^{a - b + 1+b+1} \tilde{w}^{-b - 1} \frac{1}{1 - \frac{\tw}{\tz}} \prod_{i = 1}^N \frac{\frac{\tw}{\tz} - e^{\lambda_i}/\tilde{z}}{1 - e^{\lambda_i}/{\tz}} = I'(a, b)
\end{equation}
for
\[
I'(a, b) = \oint_{\{e^{\lambda_i}\}} \frac{d\tz}{2\pi \ii} \oint_{\{0\}} \frac{d\tw}{2 \pi \ii } \cdot \frac{{\tz}^{a} \tilde{w}^{-b - 1}}{\tz - \tw}\cdot \prod_{i = 1}^N \frac{\tw- e^{\lambda_i}}{\tz - e^{\lambda_i}},
\]
where the $\tz$-contour encloses the poles at $e^{\lambda_i}$ and avoids the negative real axis, the $\tw$-contour is a large circle enclosing $0$ and the $\tz$-contour, and we use one of the negative signs to switch the orientation of the $\tw$-contour.  In addition, we notice for positive integer $a \geq N$ that
\[
J(a, b) = \prod_{i \neq b} \frac{b - i}{a - i} = J'(a, b) := (-1)^{N - b + 1} \frac{\Gamma(N - b) \Gamma(b + 1) \Gamma(a - N + 1) (a - b)}{\Gamma(a + 1)}.
\]
Therefore, denoting the right hand side of the desired by $B(a, b) := J'(a, b) I'(a, b)$, we find that
\begin{equation} \label{eq:bb-eq}
\frac{\cB(\mu_{a, b}, \lambda)}{\cB(\rho, \lambda)} = B(a, b) \qquad \text{for $a \in \{N, N + 1, \ldots\}$}.
\end{equation}
We now complete the proof by checking the conditions of Carlson's theorem.

As a function of $a$, the expression $\frac{\cB(\mu_{a, b}, \lambda)}{\cB(\rho, \lambda)}$ is evidently analytic, of exponential type, and satisfies
\[
\left|\frac{\cB(\mu_{N + \ii y, b}, \lambda)}{\cB(\rho, \lambda)}\right| < C e^{|y|} \qquad \text{for $y \in \RR$}
\]
for some $C > 0$.  On the other hand, notice that $J'(a, b)$ is evidently meromorphic in $a$ with poles at $\{N - 1, \ldots, 0\} \setminus \{b\}$ and that by contracting first the $\tz$-contour and then the $\tw$-contour that $I'(a, b)$ is an analytic function in $a$ of exponential type with $I'(N + \ii y, b) < C e^{|y|}$ for some $C > 0$.  Further, for $a \in \{N - 1, \ldots, 0\} \setminus \{b\}$, by deforming the $\tz$ contour to $\infty$ we find that the non-residue term vanishes and we have
\[
I'(a, b) = \oint_{\{0\}} \frac{d\tw}{2\pi\ii} \tw^{a - b - 1} = 0.
\]
This shows that $I'(a, b)$ has zeroes at $a \in \{N - 1, \ldots, 0\} \setminus \{b\}$ and hence $B(a, b)$ is analytic.  Finally, $J(a, b)$ is evidently of exponential type and satisfies $J(N + \ii y, b) < C e^{|y|}$ for $y \in \RR$ and some $C > 0$, meaning that both functions satisfy the conditions of Carlson's theorem as a function of $a - N$.  We conclude that they are identically equal, as desired.
\end{proof}

\begin{remark}
The definition (\ref{eq:schur-def}) of Schur functions and the principal evaluation identity
\cite[Example I.1]{Mac} imply the index-variable duality formula
\[
\frac{s_\lambda(q^a)}{s_\lambda(q^\rho)} = \frac{s_{a - \rho}(q^{\lambda + \rho})}{s_{a - \rho}(q^\rho)}.
\]
Using this and the method of proof of Theorem \ref{thm:mvb-int}, we may also obtain the identity
\begin{multline} \label{eq:schur-int}
 \frac{s_\lambda(1, q, \ldots, q^{b-1}, q^a, q^{b+1}, \ldots, q^{N-1})}{s_\lambda(1, \ldots, q^{N -  1})}=  q^{-(N-b-1)(N-b-2)/2} \prod_{i=1}^{N-b-1} \frac{q^{N-1}-q^{N-i-1}}{q^{a}-q^{N-i}}  \prod_{i=1}^{b} \frac{q^{b}-q^{i-1}}{q^{a}-q^{i-1}}  \\
 \times (-1)^{N - b - 1} \oint_{\{q^{\lambda_i+N-i}\}} \frac{d z}{2\pi \ii } \oint_{\{0\}} \frac{d w}{2 \pi \ii } \cdot \frac{{ z}^{a} {w}^{-b - 1}}{z - w}\cdot \prod_{i = 1}^N \frac{w- q^{\lambda_i + N - i}}{z - q^{\lambda_i + N - i}},
\end{multline}
where the $z$-contour encloses $q^{\lambda_i + N - i}$ and avoids the negative real axis and the $w$-contour encloses the $z$-contour and $0$.  This identity and its generalizations to other root systems are of key importance in \cite{CG18}.
\end{remark}

\subsection{Single variable asymptotics}

The proof of Theorem \ref{thm:multi-asymp} will proceed through the single variable case, which gives
asymptotics for the quantity
\[
\wB_N(\wa, \wb) := \frac{\cB(\mu_{\wa N, \wb N}, \lambda^N)}{\cB(\rho, \lambda^N)}.
\]
In this case, we will apply the method of steepest descent to derive the following specialization of Theorem \ref{thm:multi-asymp} to a single variable.

\begin{theorem} \label{thm:single-asymp}
If $\lambda^N$ satisfies Assumption \ref{ass:measure}, for $\wb \in \{0, \frac{1}{N}, \ldots, \frac{N - 1}{N}\}$, as a function of $\wa$, the quantity $\wB_N(\wa, \wb)$ admits a holomorphic extension to an open neighborhood $U$ of $[0, 1]$ such that uniformly on compact subsets of $U$ and uniformly in $\wb$ we have
\[
  \wB_N(\wa, \wb) = \frac{1}{M^{-1}_{\wrho_N}(\wa - 1) - M^{-1}_{\wrho_N}(\wb - 1)} \frac{\wa - \wb}{\sqrt{M_{\wrho_N}'(M^{-1}_{\wrho_N}(\wa - 1)) M_{\wrho_N}'(M^{-1}_{\wrho_N}(\wb - 1))}} \frac{\sqrt{S_{\wrho_N}(\wb - 1)} e^{N\wPsi_{\rho_N}(\wb)}}{\sqrt{S_{\wrho_N}(\wa - 1)} e^{N\wPsi_{\rho_N}(\wa)}} [1 + o(1)].
\]
\end{theorem}
\begin{remark}
By Lemma \ref{lem:limit-analytic}, the expression
\[
\frac{1}{M^{-1}_{\wrho_N}(\wa - 1) - M^{-1}_{\wrho_N}(\wb - 1)} \frac{\wa - \wb}{\sqrt{M_{\wrho_N}'(M^{-1}_{\wrho_N}(\wa - 1)) M_{\wrho_N}'(M^{-1}_{\wrho_N}(\wb - 1))}} \frac{\sqrt{S_{\wrho_N}(\wb - 1)} e^{N\wPsi_{\rho_N}(\wb)}}{\sqrt{S_{\wrho_N}(\wa - 1)} e^{N\wPsi_{\rho_N}(\wa)}}
\]
in Theorem \ref{thm:single-asymp} is analytic in $\wa$ on a complex neighborhood of $[0, 1]$.
\end{remark}

To prove Theorem \ref{thm:single-asymp}, we need a few asymptotic preliminaries.

\begin{lemma}[{\cite[Proposition 3.1 and 3.3]{BV92}}] \label{lem:invert}
There is a neighborhood $U \supset [0, 1]$ such that $M_{\wrho_N}^{-1}(z - 1)$ is well-defined, bijective, and meromorphic on $U$ with unique pole at $1$ and zero at $0$ and $S_{\wrho_N}(z - 1)$ is holomorphic with no zeros on $U$.  Further, the functions $M_{\wrho_N}^{-1}(z - 1)$ converge uniformly on compact subsets of $U \setminus \{1\}$ to $M_{\wrho}^{-1}(z - 1)$, and the functions $S_{\wrho_N}(z - 1)$ converge uniformly on compact subsets of $U$ to $S_{\wrho}(z - 1)$.  
\end{lemma}

\begin{lemma} \label{lem:s-val}
  We have that $S_{\wrho_N}(z) > 0$ for $z \in [-1, 0]$ and that
  \[
  S_{\wrho_N}(0) = \left[\int e^{s} d\rho_N(s)\right]^{-1} \qquad \text{ and } \qquad S_{\wrho_N}(-1) = \int e^{-s} d\rho_N(s).
  \]
\end{lemma}
\begin{proof}
The first statement follows by \cite[Proposition 3.1(3)]{BV92} and the latter two are straightforward by series expansion in the definition of $S_{\wrho_N}(u)$. 
\end{proof}

\begin{proof}[Proof of Theorem \ref{thm:single-asymp}]
The argument is based on the steepest descent method, cf. \cite{Erd56, Cop65}. We deform the $z$- and $w$-contours in the double contour integral representation of Theorem \ref{thm:mvb-int} so that they pass through the critical points the $z$ and $w$ parts of the integrand (multiplied by $(z-w)$). The leading contribution to the asymptotics is then given by integrals over small neighborhoods of these critical points, where we can consider Taylor expansions of the integrand. We now present the technical details.  Define the function
\[
h_x(y) := xy - \int \log(e^y - e^s) d\rho_N(s)
\]
for which we have that 
\begin{align*}
h_x'(y) &= x - 1 - M_{\wrho_N}(e^{-y})\\
h_{x}''(y) &= e^{-y} M_{\wrho_N}'(e^{-y})\\
h_{x}'''(y) &= -e^{-y} M_{\wrho_N}'(e^{-y}) - e^{-2y} M_{\wrho_N}''(e^{-y}).
\end{align*}
In particular, for $x \in U \setminus \{0, 1\}$, by Lemma \ref{lem:invert}, $h_x(y)$ has critical points at $- \log M_{\wrho_N}^{-1}(x - 1) + 2\pi \ii\ZZ$.  We define
\[
y_x := - \log M_{\wrho_N}^{-1}(x - 1)
\]
to be the unique critical point of this form with $\Im(y_x) \in [0, 2\pi)$.  Notice in particular that
\begin{align} \label{eq:hyx}
h''_x(y_x + t) &= \int \frac{e^{s + t} M_{\wrho_N}^{-1}(x - 1)}{(1 - e^{s + t} M^{-1}_{\wrho_N}(x - 1))^2} d\rho_N(s)\\ \label{eq:hyx2}
h'''_x(y_x + t) &= \int \left[\frac{e^{s + t} M^{-1}_{\wrho_N}(x - 1)(1 + e^{s + t} M^{-1}_{\wrho_N}(x - 1))}{(1 - e^{s + t} M^{-1}_{\wrho_N}(x - 1))^3}\right] d\rho_N(s).
\end{align}
For $1/2 > \eps > 0$, define the counterclockwise contour
\begin{multline*}
\Upsilon_\eps := \{x - \ii \eps \mid x \in [0, 1]\} \cup \{1 + \eps \exp(2\pi \ii x) \mid x \in [-1/4, 1/4]\} \cup\\ \{x + \ii \eps \mid x \in [0, 1]\} \cup \{\eps \exp(2\pi \ii x) \mid x \in [1/4, 3/4]\}
\end{multline*}
shown in Figure \ref{fig:upsilon} which encloses $[0, 1]$ and maintains a distance $\eps$ from it.

\begin{figure}[h]
  \includegraphics{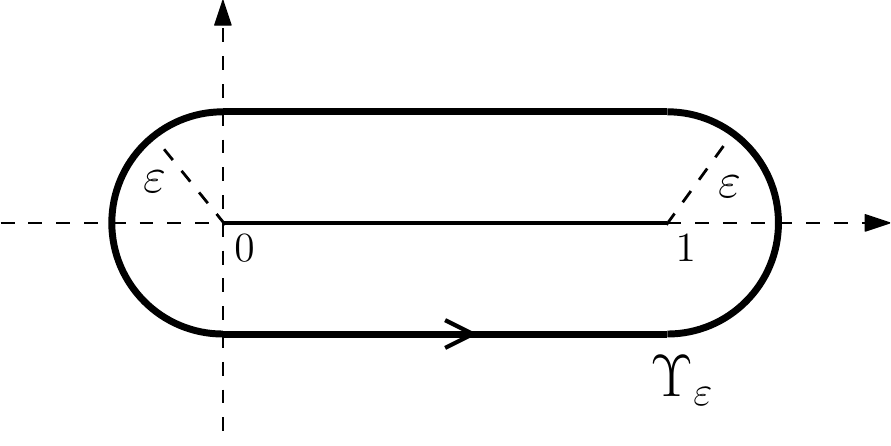}
  \caption{The contour $\Upsilon_\eps$} \label{fig:upsilon}
\end{figure}

Plugging in the value of $y_x$, we obtain the expressions 
\[
h_x''(y_x) = \frac{M_{\wrho_N}^{-1}(x - 1)}{\partial_x M_{\wrho_N}^{-1}(x - 1)} \qquad \text{ and } \qquad h_x'''(y_x) = [\partial_x^2 M_{\wrho_N}^{-1}(x - 1)] [\partial_x M_{\wrho_N}^{-1}(x-1)]^3 [M_{\wrho_N}^{-1}(x - 1)]^2 - \frac{M_{\wrho_N}^{-1}(x - 1)}{\partial_x M_{\wrho_N}^{-1}(x - 1)}
\]
from which Lemma \ref{lem:invert} implies that $h_x''(y_x)$ and $h_x'''(y_x)$ are meromorphic as functions of $x$.  Furthermore, for $x \in [0, 1]$, by the expression (\ref{eq:hyx}) and the fact that $M_{\wrho_N}^{-1}(x - 1)$ is non-positive, we see that $h_x''(y_x)$ is differentiable in $x$ and $h_x''(y_x) \leq 0$ with zeroes only at $x \in \{0, 1\}$.  In addition, by the expression (\ref{eq:hyx2}) we see that $h_x'''(y_x)$ is differentiable and bounded on $x \in [0, 1]$, with zeroes at $x \in \{0, 1\}$.  By Lemma \ref{lem:invert}, all of these properties hold uniformly in $N$ for large enough $N$.  We conclude that there exists $C_1, C_1', C_1'' > 0$, $\eps > 0$, and $\tau > 0$ so that for large enough $N$ we have uniformly in $N$ and $\wb$ by the fact that $|h_{\wa}''(y_{\wa})|$ is bounded for $\wa \in [0, 1]$ that 
\begin{equation} \label{eq:sd-bound}
|h_{\wa}''(y_{\wa})| \in [C_1^{-1}, C_1] \qquad \text{for all $\wa \in \Upsilon_\eps$}
\end{equation}
and by the boundedness of $|h_{\wa}'''(y_{\wa})|$ for $\wa \in [0, 1]$ that
\begin{equation} \label{eq:td-bound}
|h_{\wa}'''(y_{\wa})| \leq C_1 \qquad \text{ for all $\wa \in \Upsilon_\eps$}
\end{equation}
and by the boundedness of $|h_{\wb}''(y_{\wb})|$ for $\wb \in [0, 1]$ that
\begin{equation} \label{eq:bsd-bound}
h_{\wb}''(y_{\wb}) \in [-C_1, -C_1^{-1}] \qquad \text{ for all $\wb \in [\tau, 1 - \tau]$}
\end{equation}
and by the fact that $h_{\wb}''(y_{\wb})$ is non-positive on $[0, 1]$ and has zeroes and is differentiable at $\wb \in \{0, 1\}$ that
\begin{equation} \label{eq:bsd-bound2}
h_{\wb}''(y_{\wb}) \leq - C_1' \wb \text{ for $\wb \in [0, \tau]$} \qquad h_{\wb}''(y_{\wb}) \leq -C_1' (1 - \wb) \text{ for $\wb \in [1 - \tau, 1]$}
\end{equation}
and by the boundedness of $h_{\wb}'''(y_{\wb})$ for $\wb \in [0, 1]$ that
\begin{equation} \label{eq:btd-bound}
|h_{\wb}'''(y_{\wb})| \leq C_1 \qquad \text{for $\wb \in [\tau, 1 - \tau]$}
\end{equation}
and by (\ref{eq:hyx2}) and Lemma \ref{lem:invert} that uniformly in $|t| < \tau$, we have
\begin{equation} \label{eq:btd-bound2}
|h_{\wb}'''(y_{\wb} + t)| \leq C_1'' \wb \text{ for $\wb \in [0, \tau]$} \qquad |h_{\wb}'''(y_{\wb} + t)| \leq C_1''(1 - \wb) \text{ for $\wb \in [1 - \tau, 1]$}.
\end{equation}
From now on, we fix this $\eps > 0$.  Define the domain $V_\eps$ to be the interior of the region enclosed by $\Upsilon_\eps$.

Notice that both $\wB_N(\wa, \wb)$ and the claimed asymptotic expression are analytic in $\wa$ for each $N$ by Lemma \ref{lem:limit-analytic}.  Therefore, if we can show uniform convergence for $\wa \in \Upsilon_\eps$, then this will show that their differences form a sequence of holomorphic functions on a neighborhood of $\Upsilon_\eps \cup V_\eps$ which converge to $0$ on $\Upsilon_\eps$.  By the Cauchy integral formula, this would imply that the difference also converges to $0$ on $V_\eps$.  Thus, to prove the desired statement for $U = V_\eps$, it suffices for us to prove it for $\wa \in \Upsilon_\eps$.  For the rest of the proof, we will therefore assume that $\wa \in \Upsilon_\eps$. 

We now define
\[
I(\wa, \wb) := \oint_{\{e^{\lambda_i^N}\}} \frac{dz}{2\pi\ii} \oint_{\{0\}} \frac{dw}{2\pi\ii} \frac{z^{\wa N} w^{-\wb N - 1}}{z - w} \prod_{i = 1}^N \frac{w - e^{\lambda_i^N}}{z - e^{\lambda_i^N}}
\]
so that
\[
\wB_N(\wa, \wb) = (-1)^{(1 - \wb)N + 1} \frac{\Gamma(N - \wb N) \Gamma(\wb N + 1)\Gamma(\wa N - N + 1)(\wa - \wb) N}{\Gamma(\wa N + 1)} I(\wa, \wb)
\]
by Theorem \ref{thm:mvb-int}.  Choose $M > 1$ so that $\max\{e^s, e^{-s}\} < M$ for $s \in I$, where $I$ is the interval from Assumption \ref{ass:measure}.  Given this $M$, choose $k_N$ so that as $N \to \infty$ we have
\begin{equation} \label{eq:kn-choice}
k_N = \omega(1) \qquad \text{ and } \qquad M^{k_N} \frac{k_N^2}{N} = o(N^{1/2}).
\end{equation}
We divide the analysis into several cases depending on the value of $\wb$, where we wish to obtain convergence uniform in $\wb$ in each case.

\noindent \textbf{Case 1:} Suppose that $k_N/N < \wb < 1 - k_N/N$.  Make the change of variables $w = e^{\ww}$ and $z = e^{\wz}$ so that
\[
  I(\wa, \wb) = \oint_{\Gamma_{\wa}'} \frac{d\wz}{2 \pi \ii} \oint_{\Sigma_{\wb}'} \frac{d\ww}{2\pi \ii} \frac{1}{1 - e^{\ww - \wz}} \exp\Big(N(h_{\wa}(\wz) - h_{\wb}(\ww))\Big),
\]
where the $\wz$-contour $\Gamma_{\wa}'$ is a positively oriented contour around $\{\lambda_i^N\}$ and the $\ww$-contour $\Sigma_{\wb}'$ connects $y_{\wb}$ to $y_{\wb} + 2 \pi \ii$ while staying to the right of $\Gamma_{\wa}'$.

\noindent \textbf{Case 1, Step 1.} We deform $\Gamma_{\wa}'$ and $\Sigma_{\wb}'$ to steepest descent contours.  For $\Sigma_{\wb}'$, notice that $h_{\wb}(y)$ has a unique critical point at $y_{\wb}$.  In addition, because $\wb \in (0, 1)$, we see that $y_{\wb} \in (-\infty, 0) + \pi\ii$ and that 
\[
\Re h_{\wb}(y) = \wb \cdot \Re[y] - \int \log|e^y - e^s| d\rho_N(s),
\]
which implies that on the vertical line segment passing through $y_{\wb}$, $\Re h_{\wb}(y)$ is minimized at $y \in y_{\wb} + 2 \pi \ii \ZZ$.  Putting these facts together, we deform $\Sigma_{\wb}'$ to the line segment $\Sigma_{\wb}$ between $y_{\wb}$ and $y_{\wb} + 2 \pi \ii$ so that the resulting contour $\Sigma_{\wb}$ satisfies the following properties uniformly in $N$:
\begin{itemize}
\item $\Sigma_{\wb}$ lies within $\{\Re h_{\wb}(y) > \Re h_{\wb}(y_{\wb})\}$ away from its endpoints, where we note that $\Re h_{\wb}(y)$ may be infinite on $\Sigma_{\wb}$;

\item for points on $\Sigma_{\wb}$, in constant-size neighborhood of $y_{\wb}$ and $y_{\wb} + 2\pi\ii$,
the value of $\Re [h_{\wb}(y) - h_{\wb}(y_{\wb})]$ is smaller on this neighborhood than its complement;

\item $\Sigma_{\wb}$ has bounded length.
\end{itemize}
For $\Gamma_{\wa}'$, notice that $h_{\wa}(y)$ has a unique critical point at $y_{\wa}$ and that $\Re h_{\wa}(y)$ is harmonic away from $I + 2 \pi \ii \ZZ$.  Therefore, at least one of the components of $\{\Re h_{\wa}(y) > \Re h_{\wa}(y_{\wa})\}$ has compact closure and contains either $\supp d\rho_N$ or $\supp d\rho_N + 2\pi\ii$.  In addition, for $\eta > 0$ to be determined later, we see by the weak convergence of $d\rho_N$ to $d\rho$ that $\Re h_{\wa}(y)$ converges uniformly on compact subsets of
\[
D_\eta := \CC \setminus \{z \mid |z - w| > \eta \text{ for $w \in I + 2\pi\ii \ZZ$}\}
\]
to the function $\Re h_{\wa}^\infty(y)$ for 
\[
h_{\wa}^\infty(y) := \wa y - \int \log(e^y - e^s) d\rho(s).
\]
As a consequence, the intersection of the level set $\{\Re h_{\wa}(y) = \Re h_{\wa}(y_{\wa})\}$ with $D_\eta$ converges to a fixed contour.  Because $y_{\wa} \in D_\eta$ for small enough $\eta$, for small enough $\eta > 0$, there exists some small $\nu > 0$ so that the line through $y_{\wa}$ in the direction of $[-h''_{\wa}(y_{\wa})]^{-1/2}$ intersects $\{\Re h_{\wa}(y) = \Re h_{\wa}(y_{\wa}) - \nu\}$ at least a fixed distance away from $y_{\wa}$ for large enough $N$.  Notice that the union of $\{\Re h_{\wa}(y) = \Re h_{\wa}(y_{\wa}) - \nu\}$ and the line through $y_{\wa}$ in the direction of $[-h''_{\wa}(y_{\wa})]^{-1/2}$ contains a loop which encloses either $\supp d\rho_N$ or $\supp d\rho_N + 2 \pi \ii$.  We deform $\Gamma_{\wa}'$ to the closed contour $\Gamma_{\wa}$ defined as the pieces of the contour forming this loop.

Notice that $\Gamma_{\wa} \cap D_\eta$ converges to a fixed contour, while $\Gamma_{\wa} \cap (\CC \setminus D_\eta)$ lies within a single connected component of $\CC \setminus D_\eta$, hence has bounded length.  We conclude that $\Gamma_{\wa}$ satisfies the following properties uniformly in $N$:
\begin{itemize}
\item in a constant size neighborhood of $y_{\wa}$, $\Gamma_{\wa}$ is a line segment in the direction of $[-h''_{\wa}(y_{\wa})]^{-1/2}$, and the value of $\Re[h_{\wa}(y_{\wa}) - h_{\wa}(y)]$ is smaller on this neighborhood than its complement;

\item $\Gamma_{\wa}$ is a closed loop passing through $y_{\wa}$ which encloses $\supp d\rho_N$ or $\supp d\rho_N + 2 \pi \ii$;
  
\item $\Gamma_{\wa}$ lies within $\{\Re h_{\wa}(y) < \Re h_{\wa}(y_{\wa})\}$ outside of $y_{\wa}$ and $D_\eta$;

\item $\Gamma_{\wa}$ has bounded length.
\end{itemize}
We illustrate the two contours $\Sigma_{\wb}$ and $\Gamma_{\wa}$ in Figure \ref{fig:asymp-contours}.
After these two constructions, we see that
\begin{itemize}
\item $\Re h_{\wb}(y)$ achieves a global minimum on $\Sigma_{\wb}$ at $y_{\wb}$ and $y_{\wb} + 2 \pi \ii$;

\item $\Re h_{\wa}(y)$ achieves a global maximum on $\Gamma_{\wa}$ at $y_{\wa}$.
\end{itemize}

\begin{figure}[h] 
\begin{subfigure}[b]{0.46\textwidth}
\includegraphics[height=2in]{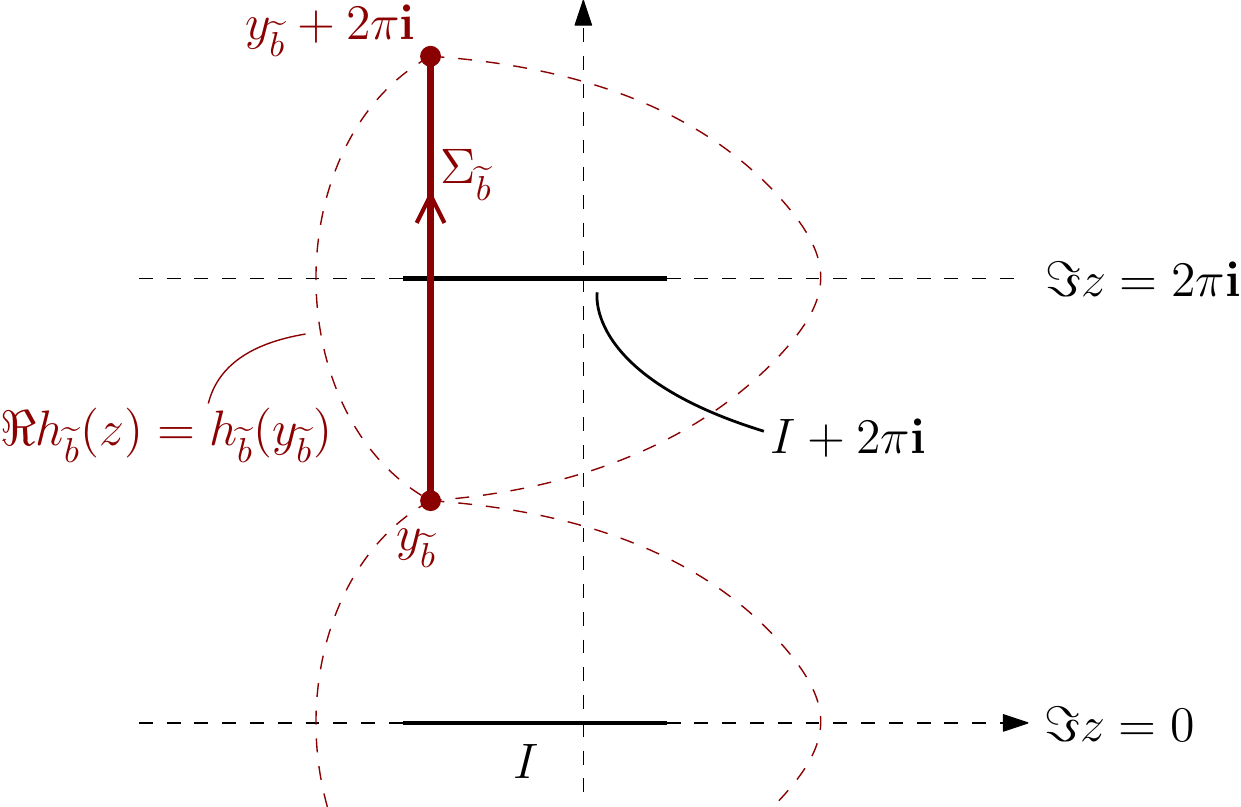}
\end{subfigure} \hfill
\begin{subfigure}[b]{0.46\textwidth}
\includegraphics[height=2in]{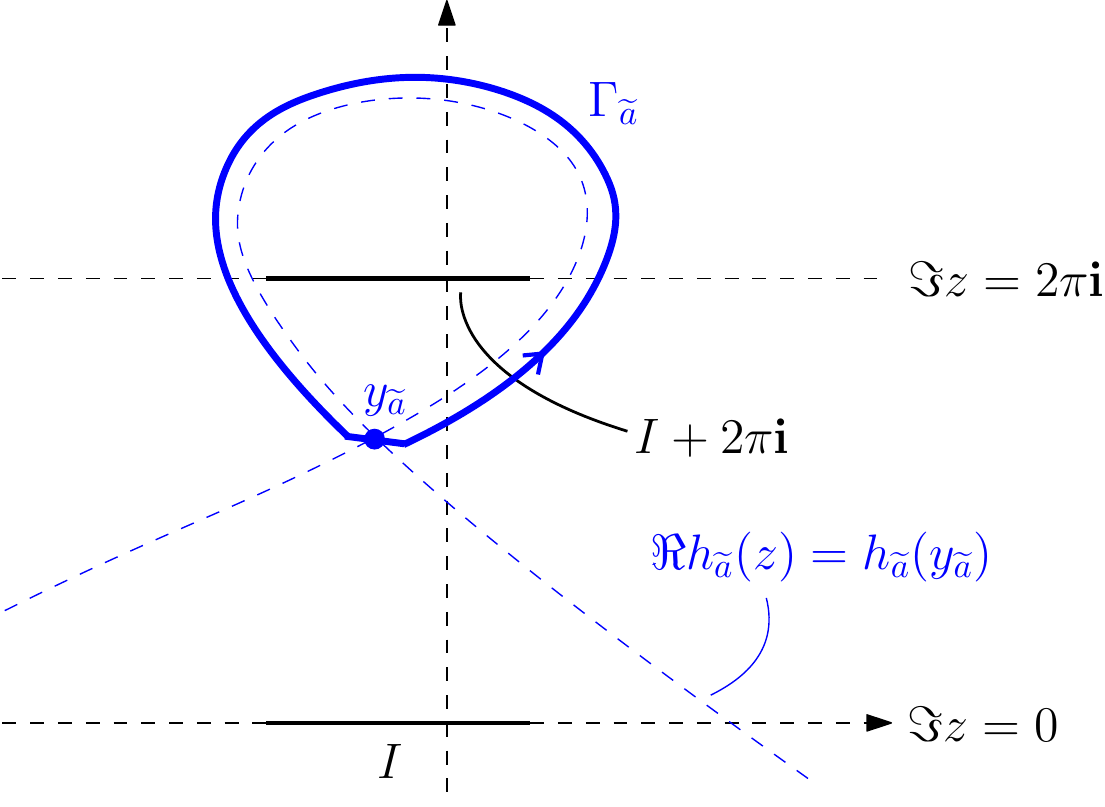}
\end{subfigure}
  \caption{Schematic configuration of the contours $\Sigma_{\wb}$ and $\Gamma_{\wa}$} \label{fig:asymp-contours}
\end{figure}

In the course of the deformations described above, we may pick up some residues.  Denoting these residues
by $R(\wa, \wb)$, we find that
\[
I(\wa, \wb) = I'(\wa, \wb) + R(\wa, \wb)
\]
for
\begin{align*}
I'(\wa, \wb) &:= \oint_{\Gamma_{\wa}} \frac{d\wz}{2 \pi \ii} \oint_{\Sigma_{\wb}} \frac{d\ww}{2\pi\ii} \frac{1}{1 - e^{\ww - \wz}} \exp \Big(N (h_{\wa}(\wz) - h_{\wb}(\ww))\Big) \\
R(\wa, \wb) &:= \oint_\Xi \frac{d\ww}{2\pi\ii} \exp\Big(N(\wa - \wb) \ww\Big),
\end{align*}
where $\Xi$ is a closed contour which is either empty or has endpoints on $\Gamma_{\wa} \cap \Sigma_{\wb}$, and we adopt the convention that $\exp(N h_{\wb}(\ww)) = 0$ if $\Re h_{\wb}(\ww) = \infty$.  Note that the contours $\Gamma_{\wa}, \Sigma_{\wb}, \Xi$ may depend on $N$, but in what follows, we will use only the properties specified above, which hold for all large enough $N$.  Define $\delta_{\wa} = \delta_{\wa}(N) > 0$
and $\delta_{\wb} = \delta_{\wb}(N) > 0$ by
\begin{equation} \label{eq:delta-size}
\delta_{\wa} = N^{-2/5} \qquad \delta_{\wb} = N^{-2/5} |h''_{\wb}(y_{\wb})|^{-2/5}
\end{equation}
so that $\delta_{\wb} = o(1)$ by (\ref{eq:bsd-bound}), (\ref{eq:bsd-bound2}), and the fact that $\wb \in [k_N/N, 1 - k_N/N]$. 
For these choices, define the decompositions $\Gamma_{\wa} = \Gamma_{\wa}^1 \sqcup \Gamma_{\wa}^2$ and $\Sigma_{\wb} = \Sigma_{\wb}^1 \sqcup \Sigma_{\wb}^2$ for
\[
\Gamma_{\wa}^1 = B(y_{\wa}, \delta_{\wa}) \cap \Gamma_{\wa}, \qquad \Gamma_{\wa}^2 = \Gamma_{\wa} \setminus \Gamma_{\wa}^1, \qquad \Sigma_{\wb}^1 = [B(y_{\wb}, \delta_{\wb}) \cup B(y_{\wb} + 2 \pi \ii, \delta_{\wb})] \cap \Sigma_{\wb}, \qquad \Sigma_{\wb}^2 = \Sigma_{\wb} \setminus \Sigma_{\wb}^1.
\]
Since $\delta_{\wa}^2 = o(1)$ and $\delta_{\wb}^2 |h''_{\wb}(y_{\wb})| = o(1)$, for large enough $N$, we have uniformly in $\wa, \wb$ that
\begin{align} \label{eq:ha-real}
  \Re[h_{\wa}(y_{\wa}) - h_{\wa}(y)] &> \frac{1}{4} |h''_{\wa}(y_{\wa})| \delta_{\wa}^2 \qquad \text{for $y \in \Gamma_{\wa}^2$}\\ \label{eq:hb-real}
  \Re[h_{\wb}(y) - h_{\wb}(y_{\wb})] &> \frac{1}{4} |h''_{\wb}(y_{\wb})| \delta_{\wb}^2 \qquad \text{for $y \in \Sigma_{\wb}^2$}.
\end{align}
We now analyze $I'(\wa, \wb)$ and $R(\wa, \wb)$ separately.

\noindent \textbf{Case 1, Step 2.} We claim that
\begin{equation} \label{eq:iprime-exp}
I'(\wa, \wb) = \frac{-\ii}{1 - e^{y_{\wb} - y_{\wa}}} \frac{1}{2\pi N \sqrt{h_{\wa}''(y_{\wa}) h_{\wb}''(y_{\wb})}} e^{N[h_{\wa}(y_{\wa}) - h_{\wb}(y_{\wb})]} (1 + o(1)),
\end{equation}
where we take the standard branch of the square root.  Consider the decomposition
\[
I'(\wa, \wb) = I_1(\wa, \wb) + I_2(\wa, \wb) + I_3(\wa, \wb) + I_4(\wa, \wb)
\]
for
\begin{align*}
  I_1(\wa, \wb) &:= \frac{1}{1 - e^{y_{\wb} - y_{\wa}}} \left[\oint_{\Gamma_{\wa}^1} \exp\Big(N h_{\wa}(\wz)\Big)\frac{d\wz}{2 \pi \ii} \right] \left[\oint_{\Sigma_{\wb}^1} \exp \Big(-N (h_{\wb}(\ww)\Big) \frac{d\ww}{2\pi\ii} \right] \\
  I_2(\wa, \wb) &:= \oint_{\Gamma_{\wa}^2} \frac{d\wz}{2 \pi \ii} \oint_{\Sigma_{\wb}^1} \frac{d\ww}{2\pi\ii} \frac{1}{1 - e^{\ww - \wz}} \exp \Big(N (h_{\wa}(\wz) - h_{\wb}(\ww))\Big) \\
I_3(\wa, \wb) &:= \oint_{\Gamma_{\wa}} \frac{d\wz}{2 \pi \ii} \oint_{\Sigma_{\wb}^2} \frac{d\ww}{2\pi\ii} \frac{1}{1 - e^{\ww - \wz}} \exp \Big(N (h_{\wa}(\wz) - h_{\wb}(\ww))\Big) \\
I_4(\wa, \wb) &:= \oint_{\Gamma_{\wa}^1} \frac{d\wz}{2 \pi \ii} \oint_{\Sigma_{\wb}^1} \frac{d\ww}{2\pi\ii} \left[\frac{1}{1 - e^{\ww - \wz}} - \frac{1}{1 - e^{y_{\wb} - y_{\wa}}}\right] \exp \Big(N (h_{\wa}(\wz) - h_{\wb}(\ww))\Big)
\end{align*}
where we note that $y_{\wa} = - \log M^{-1}_{\wrho_N}(\wa - 1)$ and $y_{\wb} = - \log M^{-1}_{\wrho_N}(\wb - 1)$ are not equal because the domains $\Upsilon_\eps$ and $\{0, 1/N, \ldots, (N - 1)/N\}$ of $\wa$ and $\wb$ do not intersect and $M_{\wrho_N}^{-1}(u - 1)$ is bijective on $U$ by Lemma \ref{lem:invert}, meaning that $1 - e^{y_{\wb} - y_{\wa}} \neq 0$. We first analyze $I_1(\wa, \wb)$.  For $x \in \{\wa, \wb\}$ and $|y - y_x| < \delta_x$, we have the Taylor expansion
\[
h_x(y) = h_x(y_x) + \frac{1}{2} h_x''(y_x) (y - y_x)^2 + \frac{1}{6} h_x'''(\xi_x(y)) (y - y_x)^3,
\]
where $\xi_x(y) \in B(y_x, \delta_x)$.  Consider the resulting decomposition
\begin{equation} \label{eq:first-int}
J(\wa) := \oint_{\Gamma_{\wa}^1} \exp\Big(N h_{\wa}(\wz)\Big)\frac{d\wz}{2 \pi \ii} = J_1(\wa) - J_2(\wa) + J_3(\wa)
\end{equation}
for
\begin{align*}
J_1(\wa) &:= e^{N h_{\wa}(y_{\wa})} \oint_{\ell_{\wa}} e^{\frac{N}{2}h_{\wa}''(y_{\wa}) (\wz - y_{\wa})^2} \frac{d\wz}{2\pi\ii}\\
J_2(\wa) &:= e^{N h_{\wa}(y_{\wa})}\oint_{\ell_{\wa} - \Gamma_{\wa}^1} e^{\frac{N}{2}h_{\wa}''(y_{\wa}) (\wz - y_{\wa})^2} \frac{d\wz}{2\pi\ii} \\
J_3(\wa) &:= e^{N h_{\wa}(y_{\wa})} \oint_{\Gamma_{\wa}^1} e^{\frac{N}{2} h_{\wa}''(y_{\wa})(\wz - y_{\wa})^2} \Big(e^{\frac{N}{6} h_{\wa}'''(\xi_{\wa}(\wz)) (\wz - y_{\wa})^3} - 1\Big) \frac{d\wz}{2\pi\ii},
\end{align*}
where $\ell_{\wa}$ is the line passing through $y_{\wa}$ in the direction of $[-h''_{\wa}(y_{\wa})]^{-1/2}$.  The main contributing factor will be
\[
J_1(\wa) = e^{N h_{\wa}(y_{\wa})} [-h''_{\wa}(y_{\wa})]^{-1/2} \int_{-\infty}^\infty e^{-\frac{N}{2} t^2} \frac{dt}{2\pi\ii} = \frac{-\ii}{\sqrt{2 \pi N h''_{\wa}(y_{\wa})}} e^{N h_{\wa}(y_{\wa})}.
\]
In this expression, notice that 
\[
h''_{\wa}(y_{\wa}) = M_{\wrho_N}^{-1}(\wa - 1) M_{\wrho_N}'(M^{-1}_{\wrho_N}(\wa - 1)) = \frac{\wa - 1}{\wa} S_{\wrho_N}(\wa - 1) M_{\wrho_N}'(M^{-1}_{\wrho_N}(\wa - 1)),
\]
where the loop each term in the final expression traces out for $\wa \in \Upsilon_\eps$ does not enclose $0$, meaning that we may choose a branch of the logarithm such that $\sqrt{h_{\wa}''(y_{\wa})}$ is well-defined for all $\wa \in \Upsilon_\eps$.  We now bound each other term relative to $|J_1(\wa)|$.  Notice that for large enough $N$, we have by a Gaussian tail bound that
\begin{align*}
|J_2(\wa)| &\leq |e^{N h_{\wa}(y_{\wa})}| |-h''_{\wa}(y_{\wa})|^{-1/2} \int_{(-\infty, -\delta_{\wa} |h''_{\wa}(y_{\wa})|^{1/2}|] \cup [\delta_{\wa} |h''_{\wa}(y_{\wa})|^{1/2}, \infty)} |e^{-\frac{N}{2} t^2}| \frac{dt}{2\pi}\\
&\leq 2|e^{N h_{\wa}(y_{\wa})}| |h''_{\wa}(y_{\wa})|^{-1/2} \cdot e^{-\frac{N}{2} \delta_{\wa}^2 |h''_{\wa}(y_{\wa})|}\\
&\leq e^{-C_2 N^{1/5}} |J_1(\wa)|
\end{align*}
for some $C_2 > 0$, where the last inequality follows from (\ref{eq:sd-bound}) and (\ref{eq:delta-size}).  Similarly, we have for large enough $N$ that
\begin{align*}
  |J_3(\wa)| &\leq |e^{N h_{\wa}(y_{\wa})}| |h_{\wa}''(y_{\wa})|^{-1/2} \int_{-\delta_{\wa} |h''_{\wa}(y_{\wa})|^{1/2}}^{\delta_{\wa} |h''_{\wa}(y_{\wa})|^{1/2}} e^{-\frac{N}{2}t^2} \left|e^{\frac{N}{6} h_{\wa}'''(\xi_{\wa}(y_{\wa} + [-h''_{\wa}(y_{\wa})]^{-1/2}t)) [-h_{\wa}''(y_{\wa})]^{-3/2} t^3} - 1\right| \frac{dt}{2\pi}\\
  &\leq |e^{N h_{\wa}(y_{\wa})}| |h_{\wa}''(y_{\wa})|^{-1/2} N^{-1/2} \int_{-\delta_{\wa} |h''_{\wa}(y_{\wa})|^{1/2} \sqrt{N}}^{\delta_{\wa} |h''_{\wa}(y_{\wa})|^{1/2} \sqrt{N}} e^{-\frac{1}{2} s^2}
  \left|e^{\frac{1}{6} h_{\wa}'''(\xi_{\wa}(y_{\wa} + [-h''_{\wa}(y_{\wa})]^{-1/2}t)) [-h_{\wa}''(y_{\wa})]^{-3/2} s^3 N^{-1/2}} - 1\right| \frac{ds}{2\pi}\\
  &\leq C_3 \delta_{\wa}^3 N^{1/2} |J_1(\wa)|\\
  &= C_3 N^{-7/10} |J_1(\wa)|
\end{align*}
for some $C_3 > 0$, where we use (\ref{eq:td-bound}) and (\ref{eq:delta-size}) in the last two inequalities.  We conclude that
\[
J(\wa) = \frac{-\ii}{\sqrt{2 \pi N h''_{\wa}(y_{\wa})}} e^{N h_{\wa}(y_{\wa})}[1 + o(1)].
\]
Consider now
\begin{equation} \label{eq:second-int}
K(\wb) := \oint_{\Sigma_{\wb}^1} \exp(- N(h_{\wb}(\ww))) \frac{d\ww}{2\pi\ii} = K_1(\wb) - K_2(\wb) + K_3(\wb)
\end{equation}
for
\begin{align*}
  K_1(\wb) &:= e^{-N h_{\wb}(y_{\wb})} \left[\oint_{\ell_{\wb}^+} e^{-\frac{N}{2} h_{\wb}''(y_{\wb}) (\ww - y_{\wb})^2} \frac{d\ww}{2\pi\ii} + \oint_{\ell_{\wb}^-} e^{-\frac{N}{2} h_{\wb}''(y_{\wb}) (\ww - y_{\wb} - 2 \pi \ii)^2} \frac{d\ww}{2\pi\ii} \right]\\
  K_2(\wb) &:= e^{-N h_{\wb}(y_{\wb})} \left[\oint_{\ell_{\wb}^+ - B(y_{\wb}, \delta)} e^{-\frac{N}{2} h_{\wb}''(y_{\wb}) (\ww - y_{\wb})^2} \frac{d\ww}{2\pi\ii} + \oint_{\ell_{\wb}^- - B(y_{\wb} + 2 \pi \ii, \delta)} e^{-\frac{N}{2} h_{\wb}''(y_{\wb}) (\ww - y_{\wb} - 2 \pi \ii)^2} \frac{d\ww}{2\pi\ii}\right]\\
  K_3(\wb) &:= e^{-N h_{\wb}(y_{\wb})} \left[\oint_{\Sigma^1_{\wb} \cap B(y_{\wb}, \delta)} e^{-\frac{N}{2} h_{\wb}''(y_{\wb}) (\ww - y_{\wb})^2} \Big(e^{\frac{N}{6} h_{\wb}'''(\xi_{\wb}(\ww)) (\ww - y_{\wb})^3} - 1\Big) \frac{d\ww}{2\pi\ii}\right.\\
    &\phantom{====}\left.+ \oint_{\Sigma^1_{\wb} \cap B(y_{\wb} + 2 \pi \ii, \delta)} e^{-\frac{N}{2} h_{\wb}''(y_{\wb}) (\ww - y_{\wb} - 2 \pi\ii)^2} \Big(e^{\frac{N}{6} h_{\wb}'''(\xi_{\wb}(\ww)) (\ww - y_{\wb} - 2 \pi \ii)^3} - 1\Big) \frac{d\ww}{2\pi\ii}\right],
\end{align*}
where $\ell_{\wb}^+$ is the upwards vertical ray through $y_{\wb}$, $\ell_{\wb}^-$ is the downwards vertical ray through $y_{\wb} + 2 \pi \ii$, and we note $h_{\wb}''(y + 2 \pi \ii) = h_{\wb}(y)$ and $N[h_{\wb}(y_{\wb} + 2 \pi \ii) - h_{\wb}(y_{\wb})] = 2\pi\ii N \wb \in 2 \pi \ii \ZZ$ for $\wb N \in \ZZ$.  As before, the main contributing factor is
\[
K_1(\wb) = e^{-N h_{\wb}(y_{\wb})} [h_{\wb}''(y_{\wb})]^{-1/2} \int_{-\infty}^\infty e^{-\frac{N}{2} t^2} \frac{dt}{2\pi \ii} = \frac{1}{\sqrt{2\pi N h_{\wb}''(y_{\wb})}} e^{-N h_{\wb}(y_{\wb})}.
\]
For large enough $N$, we have by a Gaussian tail bound that
\begin{align*}
  |K_2(\wb)| &\leq e^{-N h_{\wb}(y_{\wb})} |h_{\wb}''(y_{\wb})|^{-1/2} \int_{(-\infty, - \delta_{\wb} |h''_{\wb}(y_{\wb})|^{1/2}] \cup [\delta_{\wb} |h''_{\wb}(y_{\wb})|^{1/2}, \infty)} e^{-\frac{N}{2} t^2} \frac{dt}{2\pi} \\
  &\leq 2 e^{-N h_{\wb}(y_{\wb})} |h_{\wb}''(y_{\wb})|^{-1/2} e^{-\frac{N}{2} \delta_{\wb}^2 |h''_{\wb}(y_{\wb})|} \\
  &=  2e^{-N h_{\wb}(y_{\wb})} |h_{\wb}''(y_{\wb})|^{-1/2} e^{-[N |h''_{\wb}(y_{\wb})|]^{1/5}/2}\\
  &= |K_1(\wb)| o(1),
\end{align*}
where we apply (\ref{eq:delta-size}) and the fact that $N |h''_{\wb}(\wb)| = \omega(1)$. Similarly, for large enough $N$, we have
\begin{align*}
  |K_3(\wb)| &\leq e^{-N h_{\wb}(y_{\wb})} |h''_{\wb}(y_{\wb})|^{-1/2} \left[\int_{0}^{\delta_{\wb} |h''_{\wb}(y_{\wb})|^{1/2}} e^{-\frac{N}{2}t^2} \left|e^{-\ii\frac{N}{6} h_{\wb}'''(\xi_{\wb}(y_{\wb} + \ii [-h''_{\wb}(y_{\wb})]^{-1/2}t)) [-h_{\wb}''(y_{\wb})]^{-3/2} t^3} - 1\right| \frac{dt}{2\pi}\right.\\
  &\phantom{====}+ \left.\int_{-\delta_{\wb} |h''_{\wb}(y_{\wb})|^{1/2}}^{0} e^{-\frac{N}{2}t^2} \left|e^{-\ii\frac{N}{6} h_{\wb}'''(\xi_{\wb}(y_{\wb} + 2\pi\ii + \ii[-h''_{\wb}(y_{\wb})]^{-1/2}t)) [-h_{\wb}''(y_{\wb})]^{-3/2} t^3} - 1\right| \frac{dt}{2\pi}\right]\\
 &\leq  e^{-N h_{\wb}(y_{\wb})} |h''_{\wb}(y_{\wb})|^{-1/2} N^{-1/2} \left[\int_{0}^{\delta_{\wb} |h''_{\wb}(y_{\wb})|^{1/2} \sqrt{N}} e^{-\frac{1}{2}s^2} \left|e^{-\ii\frac{1}{6} h_{\wb}'''(\xi_{\wb}(y_{\wb} + \ii[-h''_{\wb}(y_{\wb})]^{-1/2} s / \sqrt{N})) [-h_{\wb}''(y_{\wb})]^{-3/2} s^3/\sqrt{N}} - 1\right| \frac{ds}{2\pi}\right.\\
  &\phantom{====}+ \left.\int_{-\delta_{\wb} |h''_{\wb}(y_{\wb})|^{1/2} \sqrt{N}}^{0} e^{-\frac{1}{2}s^2} \left|e^{-\ii\frac{1}{6} h_{\wb}'''(\xi_{\wb}(y_{\wb} + 2\pi\ii + \ii[-h''_{\wb}(y_{\wb})]^{-1/2} s/\sqrt{N})) [-h_{\wb}''(y_{\wb})]^{-3/2} s^3 / \sqrt{N}} - 1\right| \frac{dt}{2\pi}\right].
\end{align*}
If $\wb \in [\tau, 1 - \tau]$, then uniformly in $s \in [0, \delta_{\wb} |h''_{\wb}(y_{\wb})|^{1/2} \sqrt{N}]$ we have
\[
\delta_{\wb}^3 h_{\wb}'''(\xi_{\wb}(y_{\wb} + \ii[-h''_{\wb}(y_{\wb})]^{-1/2} s/\sqrt{N})) N = o(1).
\]
Otherwise, by (\ref{eq:btd-bound2}), we have uniformly in complex $t \in B(0, \delta_{\wb})$ that
\[
|h'''_{\wb}(y_{\wb} + t)| \leq C_1'' \wb \text{ if $\wb \in [0, \tau]$} \qquad |h'''_{\wb}(y_{\wb} + t)| \leq C_1'' (1 - \wb) \text{ if $\wb \in [1 - \tau, 1]$},
\]
meaning that uniformly in $s \in [0, \delta_{\wb} |h''_{\wb}(y_{\wb})|^{1/2} \sqrt{N}]$ we have for $\wb \in [0, \tau]$ that
\[
|\delta_{\wb}^3 h_{\wb}'''(\xi_{\wb}(y_{\wb} + \ii[-h''_{\wb}(y_{\wb})]^{-1/2} s/\sqrt{N})) N| \leq  C_1'' N^{-6/5} |h''_{\wb}(y_{\wb})|^{-6/5} \wb N = O((\wb N)^{-1/5}).
\]
A similar argument shows this bound for $\wb \in [1 - \tau, 1]$ and for $|\delta_{\wb}^3 h_{\wb}'''(\xi_{\wb}(y_{\wb} + 2 \pi \ii + \ii[-h''_{\wb}(y_{\wb})]^{-1/2} s/\sqrt{N})) N|$.  We conclude that 
\[
|K_3(\wb)| = O((\wb N)^{-1/5}) |K_1(\wb)| = O(k_N^{-1/5}) |K_1(\wb)| = |K_1(\wb)| \cdot o(1).
\]
We conclude that
\[
K(\wb) = \frac{1}{\sqrt{2\pi N h''_{\wb}(y_{\wb})}} e^{-Nh_{\wb}(y_{\wb})} (1 + o(1)).
\]
Putting these computations together, we conclude that
\begin{equation} \label{eq:i1ans}
I_1(\wa, \wb) = \frac{-\ii}{1 - e^{y_{\wb} - y_{\wa}}} \frac{1}{2\pi N \sqrt{h_{\wa}''(y_{\wa})}\sqrt{h_{\wb}''(y_{\wb})}} e^{N[h_{\wa}(y_{\wa}) - h_{\wb}(y_{\wb})]} (1 + o(1)).
\end{equation}
We now bound $I_2(\wa, \wb)$, $I_3(\wa, \wb)$, and $I_4(\wa, \wb)$.  For $I_2(\wa, \wb)$, by (\ref{eq:ha-real}) and (\ref{eq:sd-bound}) and the construction of $\Gamma_{\wa}$, on $\Gamma_{\wa}^2$ we have
\[
\left|\exp\Big(N (h_{\wa}(\wz) - h_{\wa}(y_{\wa}))\Big)\right| < \exp\Big(-\frac{1}{4} |h''_{\wa}(y_{\wa})| \delta_{\wa}^2 N\Big) \leq e^{- C N^{1/5}}
\]
for some $C > 0$, which implies by the fact that $\Gamma_{\wa}$ and $\Sigma_{\wb}$ have bounded length that
\[
|I_2(\wa, \wb)| \leq e^{-CN^{1/5}} \exp\Big(N (h_{\wa}(y_{\wa}) - h_{\wb}(y_{\wb}))\Big) = |I_1(\wa, \wb)| \cdot o(1)
\]
uniformly in $\wa, \wb$. Similarly, for $I_3(\wa, \wb)$, by (\ref{eq:hb-real}), (\ref{eq:bsd-bound}), (\ref{eq:bsd-bound2}), and the construction of $\Gamma_{\wb}$, on $\Gamma_{\wb}^2$ we have for large enough $N$ that
\[
\left|\exp\Big(N(h_{\wb}(y_{\wb}) - h_{\wb}(\ww))\Big)\right| < \exp\Big(-\frac{1}{4} |h''_{\wb}(y_{\wb})| \delta_{\wb}^2 N\Big) = e^{-[N |h''_{\wb}(y_{\wb})|]^{1/5}/4}.
\]
This implies by the fact that $\Gamma_{\wa}$ and $\Sigma_{\wb}$ have bounded length that
\[
  |I_3(\wa, \wb)| \leq e^{-[N |h''_{\wb}(y_{\wb})|]^{1/5}/4} e^{N h_{\wa}(y_{\wa})} \oint_{\Gamma_{\wa}} \frac{d\wz}{2\pi\ii} \oint_{\Sigma_{\wb}} \frac{d\ww}{2\pi\ii} \left|\frac{1}{1 - e^{\ww - \wz}} \exp\Big(N h_{\wa}(\wz)\Big)\right| = o(1) \cdot |I_1(\wa, \wb)|,
\]
where in the second step we apply a standard steepest descent analysis in $\wz$. 

For $I_4(\wa, \wb)$, for $\wz \in \Gamma^1_{\wa}$ and $\ww \in \Gamma^1_{\wb}$, since $\delta_{\wa} = o(1)$ and $\delta_{\wb} = o(1)$, we have for large enough $N$ that
\[
  \left|\frac{1}{1 - e^{\ww - \wz}} - \frac{1}{1 - e^{y_{\wb} - y_{\wa}}}\right| \leq 3\left|\frac{e^{y_{\wb} - y_{\wa}}}{(1 - e^{y_{\wb} - y_{\wa}})^2}\right| (\delta_{\wa} + \delta_{\wb}) = o(1)
\]
uniformly in $\wb$.  We conclude that 
\begin{equation} \label{eq:i4ans}
  |I_4(\wa, \wb)| = o(1) \left[\oint_{\Gamma^1_{\wa}} e^{N h_{\wa}(\wz)} \frac{d\wz}{2\pi \ii}\right] \left[\oint_{\Gamma^1_{\wb}} e^{-Nh_{\wb}(\ww)} \frac{d\ww}{2\pi \ii}\right] = o(1) \cdot |I_1(\wa, \wb)|.
\end{equation}
Combining our expressions for $I_1(\wa, \wb)$, $I_2(\wa, \wb)$, $I_3(\wa, \wb)$, and $I_4(\wa, \wb)$ yields (\ref{eq:iprime-exp}). 

\noindent \textbf{Case 1, Step 3.} We now account for $R(\wa, \wb)$, which is given by 
\[
R(\wa, \wb) = \frac{1}{(\wa - \wb) N} \left[e^{N(\wa - \wb) c_1} - e^{N(\wa - \wb)c_2}\right],
\]
where $c_1, c_2$ are the endpoints of $\Xi$ which lie on $\Gamma_{\wa} \cap \Sigma_{\wb}$.  For large enough $N$, these endpoints
must lie on $\Gamma_{\wa}^2 \cap \Sigma_{\wb}^2$, so we see by (\ref{eq:ha-real}) that
\[
\Re[(\wa - \wb) c_i] = \Re [h_{\wa}(c_i) - h_{\wb}(c_i)] < \Re [h_{\wa}(y_{\wa}) - h_{\wb}(y_{\wb})] - \frac{1}{4} |h''_{\wa}(y_{\wa})| \delta_{\wa}^2,
\]
which implies that
\[
|R(\wa, \wb)| < e^{-\frac{1}{4}|h_{\wa}''(y_{\wa})| \delta_{\wa}^2 N}\left| \exp\Big(N(h_{\wa}(y_{\wa}) - h_{\wb}(y_{\wb}))\Big)\right| = o(1) \cdot |I'(\wa, \wb)|.
\]
We conclude finally that
\[
I(\wa, \wb) =  \frac{-\ii}{1 - e^{y_{\wb} - y_{\wa}}} \frac{1}{2\pi N \sqrt{h_{\wa}''(y_{\wa})} \sqrt{h_{\wb}''(y_{\wb})}} e^{N[h_{\wa}(y_{\wa}) - h_{\wb}(y_{\wb})]} (1 + o(1)).
\]

\noindent \textbf{Case 1, Step 4.} It remains to analyze the prefactor $\frac{J(\wa, \wb)}{I(\wa, \wb)}$.  Applying Stirling's approximation, we see that
\[
\Gamma(N - \wb N) \Gamma(\wb N + 1) = 2\pi \sqrt{\frac{\wb}{1 - \wb}} N^{N} e^{-N} (1 - \wb)^{(1 - \wb) N} \wb^{\wb N} [1 + o(1)].
\]
For each $\wa \in \CC$, choose one of two branches of the logarithm such that the branch line avoids the set $[\wa - 1, \wa]$ in the complex plane.  For this branch, we see that
\begin{align*}
\log \frac{\Gamma(\wa N - N + 1)}{\Gamma(\wa N + 1)} &= - \sum_{i = 0}^{N - 1} \log(\wa N - i)\\
&= - N \log N - \sum_{i = 0}^{N - 1} \log(\wa - i/N)\\
&= - N \log N - \frac{1}{2}[\log(\wa) - \log(\wa - 1)] - N \int_0^1 \log(\wa - x) dx + O(N^{-1})\\
&= - N \log N + N - \frac{1}{2}[\log(\wa) - \log(\wa - 1)] - (1 - \wa) \log(\wa - 1) N - \wa \log(\wa) N + O(N^{-1})
\end{align*}
uniformly in $\wa$, where we apply the trapezoid rule.  Exponentiating, we find that
\[
\frac{\Gamma(\wa N - N + 1)}{\Gamma(\wa N + 1)} = N^{-N} e^N \sqrt{\frac{\wa - 1}{\wa}} \Big(\frac{\wa - 1}{\wa}\Big)^{\wa N} \frac{1}{(\wa - 1)^{N}}[1 + o(1)],
\]
where we take the standard branch of the square root in $\sqrt{\frac{\wa - 1}{\wa}}$ and note that $\frac{\wa - 1}{\wa}$ does not enclose the origin for $\wa \in \Upsilon_\eps$ and therefore has holomorphic square root.  We thus conclude that
\[
(-1)^{(1 - \wb)N + 1} \frac{\Gamma(N - \wb N) \Gamma(\wb N + 1)\Gamma(\wa N - N + 1)(\wa - \wb)N}{\Gamma(\wa N + 1)} = -2\pi (\wa - \wb) N\frac{\sqrt{\frac{\wb}{1 - \wb}} (\wb - 1)^{(1 - \wb)N} \wb^{\wb N}}{\sqrt{\frac{\wa}{\wa - 1}}  (1 - \wa^{-1})^{-\wa N} (\wa - 1)^N} [1 + o(1)]
\]
and hence
\[
\wB_N(\wa, \wb) = \frac{1}{1 - e^{y_{\wb} - y_{\wa}}} \frac{\wa - \wb}{\sqrt{h_{\wa}''(y_{\wa})} \sqrt{h_{\wb}''(y_{\wb})}} \frac{\sqrt{\frac{\wb}{\wb - 1}} (\wb - 1)^{(1 - \wb)N} \wb^{\wb N} e^{-N h_{\wb}(y_{\wb})}}{\sqrt{\frac{\wa}{\wa - 1}} (1 - \wa^{-1})^{-\wa N} (\wa - 1)^N e^{-Nh_{\wa}(y_{\wa})}} [1 + o(1)],
\]
where the $o(1)$ is uniform in $\wa$ and $\wb$.  Substituting $y_{\wa} = - \log M_{\wrho_N}^{-1}(\wa - 1)$ and $y_{\wb} = - \log M_{\wrho_N}^{-1}(\wb - 1)$ and noting that 
\[
(1 - c) \log(c - 1) + c \log(c) - h_{c}(y_{c}) = \wPsi_{\rho_N}(c),
\]
we find that
\begin{multline*}
\wB_N(\wa, \wb) = \frac{M_{\wrho_N}^{-1}(\wb - 1)}{M_{\wrho_N}^{-1}(\wb - 1) - M_{\wrho_N}^{-1}(\wa - 1)}\\ \frac{\wa - \wb}{\sqrt{M^{-1}_{\wrho_N}(\wa - 1)M'_{\wrho_N}(M^{-1}_{\wrho_N}(\wa - 1))}\sqrt{M^{-1}_{\wrho_N}(\wb - 1) M'_{\wrho_N}(M^{-1}_{\wrho_N}(\wb - 1))}}\frac{\sqrt{\frac{\wb}{\wb - 1}} e^{N \wPsi_{\wrho_N}(\wb)}}{\sqrt{\frac{\wa}{\wa - 1}} e^{N\wPsi_{\wrho_N}(\wa)}} [1 + o(1)].
\end{multline*}
Finally, since $\wb \in (0, 1)$, for the standard branch of the square root we have
\[
\frac{M_{\wrho_N}^{-1}(\wb - 1)}{\sqrt{M_{\wrho_N}^{-1}(\wb - 1) M'_{\wrho_N}(M^{-1}_{\wrho_N}(\wb - 1))}} \cdot \sqrt{\frac{\wb}{\wb - 1}} = - \frac{\sqrt{\frac{\wb}{\wb - 1} M_{\wrho_N}^{-1}(\wb - 1)}}{\sqrt{M_{\wrho_N}'(M_{\wrho_N}^{-1}(\wb - 1))}} = - \frac{\sqrt{S_{\wrho_N}(\wb - 1)}}{\sqrt{M_{\wrho_N}'(M_{\wrho_N}^{-1}(\wb - 1))}}.
\]
Similarly, for $\wa \in \Upsilon_\eps$ we see that
\[
\sqrt{M^{-1}_{\wrho_N}(\wa - 1)M'_{\wrho_N}(M^{-1}_{\wrho_N}(\wa - 1))} \cdot \sqrt{\frac{\wa}{\wa - 1}} = \sqrt{S_{\wrho_N}(\wa - 1)} \cdot \sqrt{M'_{\wrho_N}(M^{-1}_{\wrho_N}(\wa - 1))},
\]
where each of the arguments for the square roots on the left lie in the right half plane, meaning that this equality holds for the standard square root.  Substituting these equalities yields the desired result.

\noindent \textbf{Case 2: $0 \leq \wb \leq k_N/N$:} Contracting the $w$-contour through $0$, we find that 
\[
I(\wa, \wb) = e^{\sum_i \lambda_i^N} \sum_{k = 0}^{\wb N} \oint_{\{e^{\lambda_i^N}\}} \frac{dz}{2\pi\ii} (-1)^{N - k} e_{k}(e^{-\lambda_i^N}) z^{(\wa - \wb)N + k} \prod_{i = 1}^N \frac{1}{z - e^{\lambda_i^N}},
\]
where $e_k$ is the $k^\text{th}$ elementary symmetric polynomial.  Because $\lambda_i^N$ are all contained
in the fixed finite interval $I$ of Assumption \ref{ass:measure},  we have a uniform upper bound $e^{\lambda_i^N} < M$.
This implies that
\[
\left|\frac{p_1(e^{-\lambda_i^N})^k}{k!} - e_k(e^{-\lambda_i^N})\right| \leq \frac{M^k}{k!} \Big(N^k - N(N - 1) \cdots (N - k + 1)\Big) = \frac{N^k M^k}{k!} \Big(1 - \prod_{i = 0}^{k - 1} (1 - i/N)\Big).
\]
Notice now that
\[
\log \prod_{i = 0}^{k - 1} (1 - i/N) \geq N \int_0^{\frac{k - 1}{N}} \log(1 - x) dx = - (k - 1) + (k - N - 1) \log \frac{N - k + 1}{N},
\]
which implies using the fact that $1 - \eps \leq e^{-\eps}$ that
\[
\prod_{i = 0}^{k - 1} (1 - i/N) \geq e^{1 - k} \Big(1 - \frac{k - 1}{N}\Big)^{k - N - 1} \geq \exp\Big(1 - k - \frac{(k - 1)(k - N - 1)}{N}\Big) = e^{- \frac{(k - 1)^2}{N}}.
\]
These imply that if $M^k \frac{k^2}{N} = o(1)$, then uniformly in $k$ we have
\[
e_k(e^{-\lambda_i^N}) = \frac{p_1(e^{-\lambda_i^N})^k}{k!}[1 + o(1)].
\]
By the bound on $0 \leq \wb \leq k_N/N$ and (\ref{eq:kn-choice}), we find that 
\begin{align*}
  I(\wa, \wb) &= e^{\sum_i \lambda_i^N} \sum_{k = 0}^{\wb N} \oint_{\{e^{\lambda_i^N}\}} \frac{dz}{2\pi \ii} (-1)^{N - k} \frac{N^k}{k!} \left(\frac{p_1(e^{-\lambda_i^N})}{N}\right)^k z^{(\wa - \wb) N + k} \prod_{i = 1}^N \frac{1}{z - e^{\lambda_i^N}} [1 + o(1)]\\
  &=  e^{\sum_i \lambda_i^N} \oint_{\Gamma_{\wa}} \frac{d\wz}{2\pi \ii} (-1)^{N - \wb N} \frac{N^{\wb N}}{(\wb N)!} \left(\frac{p_1(e^{-\lambda_i^N})}{N}\right)^{\wb N} e^{N h_{\wa}(\wz)} [1 + o(1)]\\
  &= - e^{\sum_i \lambda_i^N} (-1)^{N - \wb N} \frac{N^{\wb N}}{(\wb N)!} \left(\frac{p_1(e^{-\lambda_i^N})}{N}\right)^{\wb N} \frac{1}{\sqrt{2\pi N h_{\wa}''(y_{\wa})}} e^{N h_{\wa}(y_{\wa})}[1 + o(1)],
\end{align*}
where the second step holds uniformly in $\wb$ because we replace the sum by the asymptotically largest term and the penultimate step is by a steepest descent analysis in $\wz$ similar to Case 1.  In this setting, we conclude as before that
\[
\frac{\Gamma(\wa N - N + 1)}{\Gamma(\wa N + 1)} = N^{-N} e^N \sqrt{\frac{\wa - 1}{\wa}} \frac{1}{(\wa - 1)^{(1 - \wa)N} \wa^{\wa N}}[1 + o(1)]
\]
and that
\[
\Gamma(N - \wb N) = \sqrt{\frac{2\pi}{(1 - \wb)N}} N^{N - \wb N} (1 - \wb)^{(1 - \wb)N} e^{-N + \wb N}[1 + o(1)].
\]
We find that
\begin{align*}
\wB_N(\wa, \wb) &= -\frac{\wa - \wb}{\sqrt{h_{\wa}''(y_{\wa})}} \left(\frac{p_1(e^{-\lambda_i^N})}{N}\right)^{\wb N}\frac{e^{\sum_i \lambda_i^N} e^{\wb N} \frac{1}{\sqrt{1 - \wb}} (1 - \wb)^{(1 - \wb)N} }{\sqrt{\frac{\wa}{\wa - 1}} e^{N\wPsi_{\rho_N}(\wa)}}[1 + o(1)]
\end{align*}
uniformly in $\wa$ and $\wb$. Notice now that
\[
M_{\wrho_N}^{-1}(\wb - 1)^{-1} = - \frac{\wb}{p_1(e^{-\lambda_i^N}) / N} [1 + O(k_N / N)] \qquad \text{ and } \qquad \frac{1}{N} p_1(e^{-\lambda_i^N}) = \int e^{-s} d\rho_N(s).
\]
Noting that $k_N^2 / N = o(1)$, we find that
\begin{align*}
  \exp\Big(N h_{\wb}(y_{\wb})\Big) &= (-1)^N M^{-1}_{\wrho_N}(\wb - 1)^{-\wb N} e^{-\sum_i \lambda_i^N} \exp\Big(-N \int \log(1 - M^{-1}_{\wrho_N}(\wb - 1)^{-1} e^{-s}) d\rho_N(s)\Big)\\
  &= (-1)^{(1 - \wb)N} \wb^{\wb N} (p_1(e^{-\lambda_i^N}) / N)^{-\wb N} e^{-\sum_i \lambda_i^N} \exp\Big(N M^{-1}_{\wrho_N}(\wb - 1)^{-1} \int e^{-s} d\rho_N(s) + O(N \wb^2)\Big) [1 + o(1)]\\
  &= (-1)^{(1 - \wb)N} \wb^{\wb N} (p_1(e^{-\lambda_i^N}) / N)^{-\wb N} e^{-\sum_i \lambda_i^N} e^{-\wb N} [1 + o(1)]
\end{align*}
and that 
\begin{align*}
  h_{\wb}''(y_{\wb}) &= \int \frac{e^{s} M_{\wrho_N}^{-1}(\wb - 1)}{(1 - e^{s} M_{\wrho_N}^{-1}(\wb - 1))^2} d\rho_N(s) = - \wb + O(\wb^2) = -\wb [1 + o(1)]\\
e^{y_{\wb} - y_{\wa}} &= M^{-1}_{\wrho_N}(\wb - 1)^{-1} e^{-y_{\wa}} = o(1).
\end{align*}
uniformly in $\wb$.  Together, these imply that
\[
e^{N \wPsi_{\rho_N}(\wb)} \frac{\wb^{1/2}}{\sqrt{h''_{\wb}(y_{\wb})}} \frac{1}{1 - e^{y_{\wb} - y_{\wa}}} = - \ii (1 - \wb)^{(1 - \wb)N} (p_1(e^{-\lambda_i^N}) / N)^{\wb N} e^{\sum_i \lambda_i^N} e^{\wb N}[1 + o(1)]
\]
Substituting in, we find that
\[
\wB_N(\wa, \wb) = \frac{\ii}{1 - e^{y_{\wb} - y_{\wa}}} \frac{\wa - \wb}{\sqrt{h''_{\wa}(y_{\wa})}\sqrt{h''_{\wb}(y_{\wb})}} \frac{\sqrt{\frac{\wb}{1 - \wb}} e^{N \wPsi_{\rho_N}(\wb)}}{\sqrt{\frac{\wa}{\wa - 1}} e^{N \wPsi_{\rho_N}(\wa)}} [1 + o(1)],
\]
which simplifies to the desired expression.

\noindent \textbf{Case 3: $1 - k_N/N \leq \wb < 1$:} After deforming the $w$-contour to $\infty$, we obtain
\[
I(\wa, \wb) = \sum_{k = 0}^{(1 - \wb)N} \oint_{e^{\lambda_i^N}} \frac{dz}{2\pi\ii} (-1)^{k + 1} e_k(e^{\lambda_i^N}) z^{(\wa + 1 - \wb)N - k - 1} \prod_{i = 1}^N \frac{1}{z - e^{\lambda_i^N}}.
\]
The remainder of the analysis then proceeds similarly to Case 2.
\end{proof}

\subsection{Proof of Theorem \ref{thm:multi-asymp}}

  By Proposition \ref{prop:mvb-det}, we have that
  \[
  \wB_N(\wa_1, \ldots, \wa_k; \wb_1, \ldots, \wb_k) = J_N(\wa_1, \ldots, \wa_k; \wb_1, \ldots, \wb_k) I_N(\wa_1, \ldots, \wa_k; \wb_1, \ldots, \wb_k)
  \]
  for
  \begin{align*}
    J_N(\wa_1, \ldots, \wa_k; \wb_1, \ldots, \wb_k) &:= \frac{\prod_{m \neq l} (\wa_mN - \wb_lN)}{\Delta(\wa_m N)\Delta(\wb_m N)} =  \frac{\prod_{m \neq l} (\wa_m - \wb_l)}{\Delta(\wa) \Delta(\wb)} \\
    I_N(\wa_1, \ldots, \wa_k; \wb_1, \ldots, \wb_k) &:= (-1)^{\frac{k(k-1)}{2}}\det\left(\frac{(\wa_i - \wb_i) N}{(\wa_i N - \wb_j N)}\frac{\cB(\mu_{\wa_i N, \wb_j N}, \lambda^N)}{\cB(\rho, \lambda^N)}\right)_{i, j = 1}^k.
  \end{align*}
  By Theorem \ref{thm:single-asymp}, we have
  \begin{align*}
  I_N(\wa_1, \ldots,& \wa_k; \wb_1, \ldots, \wb_k) = (-1)^{\frac{k(k-1)}{2}}\det\left(\frac{\wa_i - \wb_i}{\wa_i - \wb_{j}} \wB_N(\wa_i, \wb_{j}) \right)_{i, j = 1}^k\\
    &= (-1)^{\frac{k(k-1)}{2}}\det\left(\frac{1}{M_{\wrho_N}^{-1}(\wa_i - 1) - M_{\wrho_N}^{-1}(\wb_j - 1)}\right.\\&\phantom{=======}\left. \frac{\wa_i - \wb_i}{\sqrt{M'_{\wrho_N}(M_{\wrho_N}^{-1}(\wa_i - 1)) M_{\wrho_N}'(M_{\wrho_N}^{-1}(\wb_j - 1))}} \frac{\sqrt{S_{\wrho_N}(\wb_j - 1)} e^{N \wPsi_{\rho_N}(\wb_{j})}}{\sqrt{S_{\wrho_N}(\wa_i - 1)} e^{N \wPsi_{\rho_N}(\wa_i)}} [1 + o(1)]\right)_{i, j = 1}^k \\
    &=  (-1)^{\frac{k(k-1)}{2}}\det\left(\frac{1}{M_{\wrho_N}^{-1}(\wa_i - 1) - M_{\wrho_N}^{-1}(\wb_j - 1)} [1 + o(1)]\right)_{i, j = 1}^k\\
    &\phantom{======} \prod_{i = 1}^k \left[\frac{(\wa_i - \wb_i) \sqrt{S_{\wrho_N}(\wb_i - 1)} e^{N \wPsi_{\rho_N}(\wb_i)}}{\sqrt{M'_{\wrho_N}(M_{\wrho_N}^{-1}(\wa_i - 1)) M'_{\wrho_N}(M_{\wrho_N}^{-1}(\wb_i - 1))} \sqrt{S_{\wrho_N}(\wa_i - 1)} e^{N \wPsi_{\rho_N}(\wa_i)}}\right]\\
    &=  \frac{\Delta(M_{\wrho_N}^{-1}(\wa_i - 1)) \Delta(M_{\wrho_N}^{-1}(\wb_i - 1))}{\prod_{i, j} (M_{\wrho_N}^{-1}(\wa_i - 1) - M_{\wrho_N}^{-1}(\wb_j - 1))} \\
    &\phantom{======} \prod_{i = 1}^k \left[\frac{(\wa_i - \wb_i) \sqrt{S_{\wrho_N}(\wb_i - 1)} e^{N \wPsi_{\rho_N}(\wb_i)}}{\sqrt{M'_{\wrho_N}(M_{\wrho_N}^{-1}(\wa_i - 1)) M'_{\wrho_N}(M_{\wrho_N}^{-1}(\wb_i - 1))} \sqrt{S_{\wrho_N}(\wa_i - 1)} e^{N \wPsi_{\rho_N}(\wa_i)}}\right][1 + o(1)],
  \end{align*}
  where uniformity in $\wb_1, \ldots, \wb_k$ follows from the corresponding uniformity in Theorem \ref{thm:single-asymp} and where in the last step we use the Cauchy determinant formula.  Combining these asymptotics yields the desired.

\section{Products of random matrices} \label{sec:rmt-prod}

\subsection{Multivariate Bessel generating functions for matrix products}

For a $N \times N$ Hermitian positive-definite random matrix $X$ with measure $d\sigma(X)$, let $\mu = (\mu_1 \geq \cdots \geq \mu_N > 0)$ be its spectrum, which inherits a measure $d\sigma(\mu)$.  If the log-spectral measure of $X$ is $\rho$-smooth for $\rho = (N - 1, \ldots, 1, 0)$, then its multivariate Bessel generating function with respect to $\rho$ is given by
\begin{equation} \label{eq:mvb-spectral-def}
\phi_X(s) := \int \frac{\cB(s, \log \mu)}{\cB(\rho, \log \mu)} d\sigma(\mu).
\end{equation}
In this section, we will only consider multivariate Bessel generating functions with respect to $\rho$, so we will omit the $N$-tuple $\chi$ of (\ref{eq:mvb-gf}) from the notations.  In the following two lemmas, we prove the key property that the multivariate Bessel generating function is multiplicative over products of matrices.

\begin{lemma} \label{lem:mvb-matrix}
  When $s$ is an integral signature such that $s - \rho$ is a partition, the multivariate Bessel generating function of a random matrix $X$ with smooth log-spectral measure takes the form
  \[
  \phi_X(s) = \frac{1}{\dim L_{s - \rho}} \int \chi_{s - \rho}(X) d\sigma(X),
  \]
  where $L_{s - \rho}$ is the highest weight representation of $\gl_N$ corresponding to $s - \rho$ and $\chi_{s - \rho}$ is its character.
\end{lemma}
\begin{proof}
  By the Weyl character formula, for a matrix $X$ with spectrum $\mu$, we have
  \[
  \chi_{s - \rho}(X) = \frac{\det(\mu_i^{s_j})_{i, j = 1}^N}{\Delta(\mu)}.
  \]
  Combining this with (\ref{eq:mvb-def}) and (\ref{eq:mvb-spectral-def}), we find that 
  \[
\phi_X(s) = \int \frac{\det(\mu_i^{s_j}) \Delta(\rho)}{\Delta(\mu) \Delta(s)} d\sigma(\mu) = \frac{1}{\dim L_{s - \rho}} \int \chi_{s - \rho}(X) d\sigma(X). \qedhere
\]
\end{proof}

\begin{lemma} \label{lem:mult}
Let $X_1 = Y_1^*Y_1$ and $X_2 = Y_2^* Y_2$ be Hermitian random matrices with spectral measures $d\sigma_1(X_1)$ and $d\sigma_2(X_2)$, and let $X_3 = (Y_1 Y_2)^* (Y_1Y_2)$.  If the distribution of $X_2$ is unitarily invariant and the log-spectral measures of $X_1$ and $X_2$ are smooth, then we have the identity
\[
\phi_{X_3}(s) = \phi_{X_1}(s) \phi_{X_2}(s)
\]
of multivariate Bessel generating functions.
\end{lemma}
\begin{proof}
Both sides satisfy the exponential growth conditions of Carlson's theorem (Theorem \ref{thm:carlson}), so it suffices to check this for integral signatures $s$ such that $s - \rho$ is a partition.  Noting that $X_3$ has the same spectrum as $X_1X_2$ and the functional relation
\[
\int_U \chi_{s - \rho}(X_1 U X_2 U^*) d\Haar_U = \frac{1}{\dim L_{s - \rho}} \chi_{s - \rho}(X_1) \chi_{s - \rho}(X_2)
\]
from \cite[Lemma VII.4.2]{Mac} and \cite[Proposition 13.4.2]{For}, we may compute by Lemma \ref{lem:mvb-matrix} that
\begin{align*}
\phi_{X_3}(s) &= \frac{1}{\dim L_{s - \rho}} \iint\chi_{s - \rho}(X_1 X_2) d\sigma_1(X_1) d\sigma_2(X_2) \\
&= \frac{1}{\dim L_{s - \rho}} \int_{U} \iint \chi_{s - \rho}(X_1 UX_2U^*) d\sigma_1(X_1)d\sigma_2(X_2) d\Haar_U\\
&= \frac{1}{(\dim L_{s - \rho})^2} \int \chi_{s - \rho}(X_1) d\sigma_1(X_1) \cdot \int \chi_{s - \rho}(X_2) d\sigma_2(X_2) \\
&= \phi_{X_1}(s) \phi_{X_2}(s),
\end{align*}
where we apply unitary invariance in the second equality and note that the log-spectral measure of $X_3$ is smooth by this expression for the defining integral of its multivariate Bessel generating function.
\end{proof}

\subsection{LLN-appropriate and CLT-appropriate measures}

For each $N$, let $Y_N^1, \ldots, Y_N^k$ be $N \times N$ random matrices which are right unitarily invariant, and let $X^i_N := (Y_N^i)^* Y_N^i$.  We will study the log-spectral measure of
\[
X_N := (Y_N^1 \cdots Y_N^k)^* (Y_N^1 \cdots Y_N^k),
\]
which is a $N \times N$ positive definite Hermitian random matrix.  For this, notice that the multivariate Bessel generating function of such measures takes a simple form.

\begin{corr} \label{corr:mult-many}
The multivariate Bessel generating function of the log-spectral measure of $X_N$ is given by
  \[
  \phi_{X_N}(s) = \prod_{i = 1}^k \phi_{X^i_N}(s).
  \]
\end{corr}
\begin{proof}
This follows by iteratively applying Lemma \ref{lem:mult}.
\end{proof}

\begin{lemma} \label{lem:prod-spec}
If the log-spectral measures of $X^1_N, \ldots, X^k_N$ are LLN-appropriate (CLT-appropriate) for some $\Psi_i(u)$ ($\Lambda_i(u, w)$), then so is the log-spectral measure of $X_N$ for
\[
\Psi(u) := \sum_{i = 1}^k \Psi_i(u) \qquad \text{ and } \qquad \Lambda(u, w) := \sum_{i = 1}^k \Lambda_i(u, w).
\]
\end{lemma}
\begin{proof}
This follows directly from Corollary \ref{corr:mult-many}.
\end{proof}

\begin{remark}
Corollary \ref{corr:mult-many} reflects the fact that the log-spectral measure of $X_N$ depends only on the distributions of $X^i_N = (Y_N^i)^* Y_N^i$.  In what follows, we will state our results in terms of the distribution of $X^i_N$ instead of $Y^i_N$, meaning that they will hold for any distribution on $Y_N^i$ for which $X^i_N \overset{d} = (Y^i_N)^* Y^i_N$.  
\end{remark}

We now give several situations in which the spectral measure of a single random matrix is LLN-appropriate and CLT-appropriate.

\begin{theorem} \label{thm:unitary-inv}
Let $X_N$ be a sequence of unitarily invariant $N \times N$ positive-definite Hermitian random matrices with spectrum $\lambda^N$.  If the empirical measure of $\lambda^N$ has support contained within a fixed finite interval and converges weakly to a compactly supported measure $d\rho$, then the log-spectral measure of $X_N$ is CLT-appropriate with
  \begin{align*}
    \Psi(u) &= -\log S_{\wrho}(u - 1)\\
    \Lambda(u, w) &= \frac{1}{M_{\wrho}'(M^{-1}_{\wrho}(u - 1)) M_{\wrho}'(M^{-1}_{\wrho}(w - 1)) (M^{-1}_{\wrho}(u - 1) - M^{-1}_{\wrho}(w - 1))^2} - \frac{1}{(u - w)^2}.
  \end{align*}
\end{theorem}
\begin{proof}
  By definition, we have that
  \[
  \phi_{X_N}(s) = \frac{\cB(s, \log \lambda^N)}{\cB(\rho, \log \lambda^N)},
  \]
  which means that for $s = \mu_{\wa_1, \ldots, \wa_k; \wb_1, \ldots, \wb_k}$ from (\ref{eq:mult-ab-def}), we have that
  \[
  \phi_{X_N}(\mu_{\wa_1, \ldots, \wa_k; \wb_1, \ldots, \wb_k}) = \wB_N(\wa_1, \ldots, \wa_k; \wb_1, \ldots, \wb_k).
  \]
  For $I = \{i_1, \ldots, i_k\}$, notice by Theorem \ref{thm:multi-asymp} and the fact that uniformly
  convergent sequences of holomorphic functions can be differentiated that
  \[
    \frac{1}{N} \partial_{r_{i_1}} \log \phi_{X_N}^I(r N) = \frac{1}{N} \left. \partial_{\wa_1} [\log \wB_N(\wa_1, \ldots, \wa_k; \wb_1, \ldots, \wb_k)]\right|_{\wa_m = r_{i_m}, \wb_m = \frac{N - i_m}{N}} = - \wPsi_{\rho_N}'(r_{i_1}) + o(1)
  \]
  uniformly in $i$, $I$, and $r^I$ in a neighborhood of $[0, 1]^{|I|}$, where we recall from (\ref{eq:f-def}) that
  \[
  \wPsi_{\rho_N}(r) = r \log S_{\wrho_N}(r - 1) + \log(r - 1) + \int \log(M^{-1}_{\wrho_N}(r - 1)^{-1} - e^s) d\rho(s).
  \]
  Taking the limit $N \to \infty$ and applying Lemma \ref{lem:invert}, we find that
  \begin{align*}
  \lim_{N \to \infty} -\wPsi_{\wrho_N}'(u) &=
  -\log S_{\wrho}(u - 1) - u \frac{S'_{\wrho}(u - 1)}{S_{\wrho}(u - 1)} - \frac{1}{u - 1} - \int \frac{\frac{\partial}{\partial u} M_{\wrho}^{-1}(u - 1)^{-1}}{M^{-1}_{\wrho}(u - 1)^{-1} - e^s} d\rho(s)\\
    &= -\log S_{\wrho}(u - 1) - u \frac{S'_{\wrho}(u - 1)}{S_{\wrho}(u - 1)} - \frac{1}{u - 1} - u M_{\wrho}^{-1}(u - 1) \frac{\partial}{\partial u} M_{\wrho}^{-1}(u - 1)^{-1}\\
    &= -\log S_{\wrho}(u - 1). 
  \end{align*}
  Again by Theorem \ref{thm:multi-asymp}, we have that
  \begin{align*}
    \partial_{r_{i_1}} \partial_{r_{i_2}} \log \phi_{X_N}^I(rN) &= \left.\partial_{\wa_1} \partial_{\wa_2} [\log \wB_N(\wa_1, \ldots, \wa_k; \wb_1, \ldots, \wb_k)]\right|_{\wa_m = r_{i_m}, \wb_m = \frac{N - i_m}{N}}\\
    &= F_N^{(1, 1)}(r_{i_1}, r_{i_2}) + o(1)
  \end{align*}
  uniformly in $i_1, i_2$, $I$, and $r^I$ in a neighborhood of $[0, 1]^{|I|}$ for
  \[
  F_N(u, w) = \log(M_{\wrho_N}^{-1}(u - 1) - M_{\wrho_N}^{-1}(w - 1)) - \log(u - w) = F(u, w) + o(1).
  \]
  Noting that $\Lambda(u, w) = F^{(1, 1)}(u, w)$, we conclude that $X_N$ is CLT-appropriate for $\Psi(u)$ and $\Lambda(u, w)$. 
\end{proof}

We now consider the Jacobi and Wishart ensembles.  At $\beta = 2$, the Jacobi ensemble with parameters $\alpha$ and $R \geq N$ is the process on $N$ points in $[0, 1]$ with density proportional to 
\[
\prod_{i < j} (\lambda_i - \lambda_j)^2 \prod_{i = 1}^N \lambda_i^{\alpha - 1} (1 - \lambda_i)^{R - N}. 
\]
Let $Y$ be a $N \times (\alpha + N - 1)$ submatrix of a Haar distributed matrix from the unitary group $U_{\alpha + R + N - 1}$; then by \cite[Proposition 3.8.2]{For}, the Jacobi ensemble is the distribution of the eigenvalues of $Y^*Y$.  Alternatively, as discussed in \cite[Proposition 3.6.1]{For}, it corresponds to the spectral distribution of the matrix
\[
(X^{AR})^*X^{AR} \Big((X^{AR})^*X^{AR} + (Y^{NR})^* Y^{NR}\Big)^{-1},
\]
where $X^{AR}$ and $Y^{NR}$ are $A \times R$ and $N \times R$ matrices of i.i.d.~complex Gaussians and $A = \alpha + R - 1$.  We now show that this ensemble fits into our framework, for which we must introduce the Heckman-Opdam hypergeometric function at $\beta = 2$.\footnote{While for $\beta = 2$ these functions are essentially the same as multivariate Bessel functions, for other values of $\beta$ the difference is more significant.}  It and its normalized version are given by
\[
\FF_r(s) := \Delta(r) \cB(r, s) \qquad \text{ and } \qquad \hat{\FF}_r(s) := \frac{\FF_r(s)}{\FF_r(-\rho)} = \frac{\cB(-s, -r)}{\cB(\rho, -r)}.
\]
For parameters $a_1, \ldots, a_N, b_1, \ldots, b_R$ so that $a_i + b_j < 0$, the Heckman-Opdam measure on $r_1 \geq \cdots \geq r_N > 0$  is given by the density $\FF_r(a) \FF_r(b) Z(a, b)^{-1}$ for the partition function
\begin{equation} \label{eq:ho-part}
Z(a, b) := \prod_{i = 1}^N \prod_{j = 1}^R \frac{1}{- a_i - b_j} = \int_{r_1 \geq \cdots \geq r_N > 0} \FF_r(a) \FF_r(b) dr,
\end{equation}
where the second equality is a continuous version of the Cauchy-Littlewood identity; we refer the reader to \cite{BG15} for a generalization to general $\beta$ random matrices.

\begin{theorem} \label{thm:jacobi}
  Let $X_{\alpha, R, N}$ be a random matrix from the Jacobi ensemble with parameters $\alpha$ and $R$.  If $R/N = \hat{R} + O(N^{-1})$ and $\alpha / N = \hat{\alpha} + O(N^{-1})$ as $N \to \infty$ for some $\hat{R} \geq 1$ and $\hat{\alpha} > 0$, then the log-spectral measure of $X_{\alpha, R, N}$ is CLT-appropriate with
  \[
    \Psi(u) = \log \frac{\hat{\alpha} + u}{\hat{\alpha} + \hat{R} + u} \qquad \text{ and } \qquad \Lambda(u, w) = 0.
  \]
\end{theorem}
\begin{proof}
We first compute the multivariate Bessel generating function for $X_{\alpha, R, N}$.  It was shown in \cite[Theorem 2.8]{BG15} that if $\lambda_i^N = e^{-r_i^N}$ are drawn from the Jacobi ensemble with parameters $\alpha$ and $R$, then $r_i^N$ have the law of a Heckman-Opdam measure with $a_i = 1 - i$ and $b_j = 1 - j - \alpha$.  Therefore, using (\ref{eq:ho-part}) we see that
\[
\EE_{r^N}\left[\hat{\FF}_{r^N}(s)\right] = \frac{Z(s, b)}{Z(-\rho, b)} = \prod_{i = 1}^N \prod_{j = 1}^R \frac{2 - i - j - \alpha}{s_i + 1 - j - \alpha} = \prod_{i = 1}^N \frac{(1 - i - \alpha)_R}{(s_i - \alpha)_R},
\]
where $(x)_n := x(x - 1) \cdots (x - n + 1)$.  This implies that the defining integral of the multivariate Bessel generating function converges and equals
\begin{align*}
\phi_{X_{\alpha, R, N}}(s) = \EE_{r^N}\left[\frac{\cB(s, -r^N)}{\cB(\rho, -r^N)}\right] = \EE_{r^N}\left[\hat{\FF}_{r^N}(-s)\right] = \prod_{i = 1}^N \frac{(1 - i - \alpha)_R}{(-s_i - \alpha)_R}.
\end{align*}
In particular, we see that
\[
\log \phi_{X_{\alpha, R, N}}(s) = \sum_{i = 1}^N \log \frac{(- N + i - \alpha)_R}{(-s_i - \alpha)_R},
\]
where we see that for $\hat{\rho_i} = \frac{N - i}{N}$ by Stirling's approximation and the fact that $\alpha > 0$ we have
\begin{align*}
N^{-1}\log \frac{(- N + i - \alpha)_R}{(-r N - \alpha)_R} &= N^{-1}\log \frac{\Gamma(\alpha + N - i + R) \Gamma(\alpha + r N)}{\Gamma(\alpha + N - i) \Gamma(\alpha + r N + R)}\\
&= \wPsi(r) - \wPsi(\hat{\rho_i}) + O(N^{-1})
\end{align*}
uniformly on a neighborhood of $[0, 1]$ for
\[
\wPsi(r) := (\hat{\alpha} + r) \log(\hat{\alpha} + r) - (\hat{\alpha} + \hat{R} + r) \log(\hat{\alpha} + \hat{R} + r).
\]
This implies that 
\[
\frac{1}{N} \partial_{r_i} \log \phi_{X_{\alpha, R, N}}^I(r N) = \wPsi'(r_i) + O(N^{-1})
\]
uniformly in $r_i$ and $I$.  Since all mixed partials of $\phi_{X_{\alpha, R, N}}(s)$ vanish identically, this implies the desired.
\end{proof}

Using the interpretation of the Jacobi ensemble as the squared singular values of submatrices of unitary matrices, we may take a limit to Wishart ensembles.  Recall that the $L \times N$ complex Ginibre ensemble is the $L \times N$ random matrix of i.i.d.~centered complex Gaussians $\xi_{ij}$ with variance $\EE[|\xi_{ij}|^2] = N^{-1}$ and that the $N \times N$ complex Wishart ensemble with parameter $L$ is the corresponding random matrix $X_{L, N} := G_{L, N}^* G_{L, N}$. 

\begin{theorem} \label{thm:ginibre}
  Let $G_{L, N}$ be a sequence of $L \times N$ complex Ginibre matrices, and let $X_{L, N} = G_{L, N}^* G_{L, N}$ be the corresponding sequence of complex Wishart matrices.  If $L/N = \gamma + O(N^{-1})$ for some $\gamma > 1$, then the log-spectral measure of $X_{L, N}$ is CLT-appropriate with
\[
\Psi(u) = \log(u + \gamma - 1) \qquad \text{ and } \qquad \Lambda(u, w) = 0.
\]
\end{theorem}
\begin{proof}
  Set $\alpha = L - N + 1$.  As $R \to \infty$ with all other dimensions fixed, we see by \cite[Theorem 6.1]{PR04} that a $L \times N$ submatrix of a Haar distributed element of $U_{L + R}$ converges in distribution to a $L \times N$ complex Ginibre matrix scaled by $(L + R)^{-1/2} N^{1/2}$.  Thus if $X_{\alpha, R, N}$ is drawn from the Jacobi ensemble, taking the limit as $R \to \infty$ yields a matrix $X_{L, N}$ from the corresponding Wishart ensemble.  This implies that the defining integral for $\phi_{X_{L, N}}(s)$ converges and equals
\begin{align*}
\phi_{X_{L, N}}(s) &= \lim_{R \to \infty} \EE_r \left[\frac{\cB(s, \log(L + R) - \log(N) - r)}{\cB(\rho, \log(L + R) - \log(N) - r)}\right]\\
&= \lim_{R \to \infty} \prod_{i = 1}^N \left[N^{\rho_i - s_i} (L + R)^{s_i - \rho_i} \frac{(-i - L + N)_{L + R}}{(-s_i - L + N - 1)_{L + R}}\right]\\
&= \lim_{R \to \infty} \prod_{i = 1}^N \left[N^{\rho_i - s_i} (L + R)^{s_i - \rho_i} \frac{(- L + i - 1)_{L + R}}{(-s_i - L + N - 1)_{L + R}}\right]\\
&= \lim_{R \to \infty} \prod_{i = 1}^N \left[N^{\rho_i - s_i} (L + R)^{s_i - \rho_i} \frac{\Gamma(2L + R - i + 1)\Gamma(L - N + s_i + 1)}{\Gamma(L - i + 1)\Gamma(2L + R - N + s_i + 1)}\right]\\
&= \prod_{i = 1}^N \frac{\Gamma(L - N + s_i + 1)}{\Gamma(L - i + 1)} N^{\rho_i - s_i}.
\end{align*}
Applying Stirling's approximation and using the fact that $\gamma > 1$, we see that
\[
\lim_{N \to \infty} \frac{1}{N} \log \left[\frac{\Gamma(L - N + r N + 1)}{\Gamma(L - i + 1)} N^{\rho_i - r N}\right] = \wPsi(r) - \wPsi(\hat{\rho_i}) + O(N^{-1})
\]
uniformly in $r$ on a neighborhood of $[0, 1]$ for $\hat{\rho_i} = \frac{N - i}{N}$ and
\[
\wPsi(u) := (u + \gamma - 1) \log(u + \gamma - 1) - (u + \gamma - 1).
\]
As before, this implies that
\[
\frac{1}{N} \partial_{r_i} \log \phi_{X_{L, N}}^I(rN) = \wPsi'(r_i) + O(N^{-1})
\]
uniformly in $r_i$ and $I$.  Again, since all mixed partials of $\phi_{X_{L, N}}(s)$ vanish, this implies the desired.
\end{proof}

\subsection{Products of finitely many random matrices}

We now apply Theorems \ref{thm:unitary-inv}, \ref{thm:jacobi}, and \ref{thm:ginibre} to prove results on products of random matrices in several different settings.  Suppose that $Y_N^1, \ldots, Y^M_N$ are $N \times N$ random matrices which are right unitarily invariant, and let $X_N^i := (Y^i_N)^* Y^i_N$.  Under different assumptions on the spectral measure of $X_N^i$, we study the spectral measure of
\[
X_N := (Y_N^1 \cdots Y_N^M)^*(Y_N^1 \cdots Y_N^M)
\]
where $N \to \infty$ and $M$ is fixed.  Let the eigenvalues of $X_N$ be $\mu_1^N \geq \cdots \geq \mu_N^N > 0$, let their logarithms be $\lambda_i^N := \log \mu_i^N$, let
\[
d\lambda^N := \frac{1}{N} \sum_{i = 1}^N \delta_{\lambda_i^N}
\]
denote the empirical log-spectral measure of $X_N$, and define the corresponding height function by
\[
\cH_N(t) := \#\{\lambda_i^N \leq t\}.
\]
Finally, recall that the free product of two measures $\rho_1$ and $\rho_2$ was defined in \cite{Voi85, Voi87} as the unique operation $(d\rho_1, d\rho_2) \mapsto d\rho_1 \boxtimes d\rho_2$ on probability measures which is compatible with multiplication of $S$-transforms, meaning that
\[
S_{\rho_1 \boxtimes \rho_2}(z) = S_{\rho_1}(z) S_{\rho_2}(z).
\]

\begin{theorem} \label{thm:finite-uni-product}
If the $X_N^i$ have deterministic spectrum with support contained within fixed finite intervals and converging
weakly to compactly supported measures $d\rho_i$, then as $N \to \infty$ with $M$ fixed, the empirical log-spectral
measure $d\lambda^N$ converges in probability in the sense of moments to the measure $d\rho$ whose pushforward
$d\wrho$ under the exponential map satisfies 
  \[
  d\wrho = d \wrho_1 \boxtimes \cdots \boxtimes d\wrho_M.
  \]
  Its centered moments $\{p_k(\lambda) - \EE[p_k(\lambda)]\}_{k \in \NN}$ converge in probability to a Gaussian vector with covariance
  \begin{multline*}
    \Cov\Big(p_k(\lambda), p_l(\lambda)\Big) = \oint \oint \Big(\log(u/(u - 1)) - \log S_{\wrho}(u - 1)\Big)^k \Big(\log(w/(w - 1)) - \log S_{\wrho}(w - 1)\Big)^l \\
    \left(\sum_{i = 1}^M \frac{1}{M_{\wrho_i}'(M^{-1}_{\wrho_i}(u - 1)) M_{\wrho_i}'(M^{-1}_{\wrho_i}(w - 1)) (M^{-1}_{\wrho_i}(u - 1) - M^{-1}_{\wrho_i}(w - 1))^2} - \frac{M - 1}{(u - w)^2}\right).
  \end{multline*}
\end{theorem}
\begin{proof}
  By Theorem \ref{thm:unitary-inv}, the log-spectral measure of each $X_N^i$ is CLT-appropriate, so by Lemma \ref{lem:prod-spec} the log-spectral measure of $X_N$ is CLT-appropriate with
  \begin{align*}
    \Psi(u) &= -\log\Big(S_{\wrho_1}(u - 1) \cdots S_{\wrho_M}(u - 1)\Big)\\
    \Lambda(u, w) &= \sum_{i = 1}^M \left[\frac{1}{M_{\wrho_i}'(M^{-1}_{\wrho_i}(u - 1)) M_{\wrho_i}'(M^{-1}_{\wrho_i}(w - 1)) (M^{-1}_{\wrho_i}(u - 1) - M^{-1}_{\wrho_i}(w - 1))^2} - \frac{1}{(u - w)^2}\right].
  \end{align*}
  Define the measure $d\wrho$ as the free product
  \[
  d\wrho := d\wrho_1 \boxtimes \cdots \boxtimes d\wrho_M
  \]
  so that
  \[
  S_{\wrho}(u) = S_{\wrho_1}(u) \cdots S_{\wrho_M}(u).
  \]
  By Theorem \ref{thm:lln} and Lemma \ref{lem:chi-comp}, the log-spectral measure of $X_N$ converges in probability in the sense of moments to a measure with $k^\text{th}$ moment 
  \begin{align*}
    \fp_k &= \frac{1}{k + 1} \oint \Big(\log(u/(u - 1)) - \log S_{\wrho}(u - 1)\Big)^{k + 1} \frac{du}{2\pi \ii}\\
    &= \frac{1}{k + 1} \oint [-\log M_{\wrho}^{-1}(u - 1)]^{k + 1} \frac{du}{2\pi \ii}\\
    &= \frac{1}{k + 1} \int \oint [-\log z]^{k + 1}  \frac{e^{-s}}{(e^{-s} - z)^2} \frac{dz}{2\pi \ii} d\rho(s)\\
    &= \int s^k d\rho(s),
  \end{align*}
  where we make the change of variable $z = M_{\wrho}^{-1}(u - 1)$ so that $u = M_{\wrho}(z) + 1$, $du = M_{\wrho}'(z) dz$, and the $z$-contour is clockwise around $\{z^{-1} \mid z \in \supp d\wrho\}$.  Because $d\rho(s)$ is compactly supported, convergence of moments implies convergence of random measures, as desired.  The fluctuations of moments now follow from Theorem \ref{thm:clt}, Lemma \ref{lem:chi-comp}, and the fact that the log-spectral measure is CLT-appropriate.
\end{proof}

If $d\rho(s) = p(s) ds$ is absolutely continuous with $\alpha$-H\"older continuous density $p(s)$ with respect to Lebesgue measure, meaning that $|p(x) - p(y)| < C |x - y|^\alpha$ for $|x - y| < \delta$ for some $\alpha > 0$, $\delta > 0$ and $C > 0$, we may define the Cauchy principal value integral
\[
M_{\wrho}(e^{-t}) := \text{p.v.} \int \frac{e^{s - t}}{1 - e^{s - t}} d\rho(s),
\]
which satisfies
\begin{equation} \label{eq:m-stieltjes}
\lim_{\eps \to 0^{\pm}} M_{\wrho}(e^{-t + \ii \eps}) = M_{\wrho}(e^{-t}) \pm \ii \pi p(t)
\end{equation}
by virtue of the relation
\[
M_{\wrho}(e^{-t}) = \int \frac{1}{e^t - r} p(\log r) dr
\]
between $M_{\wrho}(e^{-t})$ and the Stieltjes transform of the measure with density $p(\log r)$.  

\begin{corr} \label{corr:finite-uni-density}
 Suppose that the measure $d\rho$ is absolutely continuous with respect to Lebesgue measure with $\alpha$-H\"older continuous density $d\rho = p(s) ds$. The centered height function $\cH_N(t) - \EE[\cH_N(t)]$ converges in the sense of moments to the Gaussian random field on $\RR$ with covariance
  \begin{multline*}
  K(t, s) = - \frac{1}{2\pi^2} \log\left|\frac{M_{\wrho}(e^{-t}) - M_{\wrho}(e^{-s}) + \ii \pi (p(t) - p(s))}{M_{\wrho}(e^{-t}) - M_{\wrho}(e^{-s}) + \ii \pi (p(t) + p(s))}\right| \\
  - \frac{1}{2\pi^2} \sum_{i = 1}^M \left[\log\left|\frac{M_{\wrho_i}^{-1}(M_{\wrho}(e^{-t}) + \ii \pi p(t)) - M_{\wrho_i}^{-1}(M_{\wrho}(e^{-s}) + \ii \pi p(s))}{M_{\wrho_i}^{-1}(M_{\wrho}(e^{-t}) + \ii \pi p(t)) - M_{\wrho_i}^{-1}(M_{\wrho}(e^{-s}) - \ii \pi p(s))}\right| - \log\left|\frac{M_{\wrho}(e^{-t}) - M_{\wrho}(e^{-s}) + \ii \pi (p(t) - p(s))}{M_{\wrho}(e^{-t}) - M_{\wrho}(e^{-s}) + \ii \pi (p(t) + p(s))}\right|\right]. 
  \end{multline*}
\end{corr}
\begin{proof}
  Choose an interval $I = [t_1, t_2]$ so that $\supp d\lambda^N \subset I$.  By integration by parts we see that
  \begin{align*}
    \int_I \cH_N(t) t^k dt &= \frac{\cH_N(t_2)}{k + 1} t_2^{k + 1} - \frac{\cH_N(t_1)}{k + 1} t_1^{k + 1} - \int_I \cH_N'(t) \frac{t^{k + 1}}{k + 1} dt \\
    &= \frac{\cH_N(t_2)}{k + 1} t_2^{k + 1} - \frac{\cH_N(t_1)}{k + 1} t_1^{k + 1} - \frac{1}{k + 1} p_{k + 1}(\lambda^N),
  \end{align*}
  which means that
  \[
  H_k := \int_I (\cH_N(t) - \EE[\cH_N(t)]) t^k dt = - \frac{1}{k + 1}\Big(p_{k + 1}(\lambda^N) - \EE[p_{k + 1}(\lambda^N)]\Big).
  \]
  By Theorem \ref{thm:finite-uni-product}, we find that $H_k$ are asymptotically Gaussian with covariance $\Cov(H_k, H_l)$ given by
  \begin{align*}
  \frac{1}{(k + 1)(l + 1)}&\fCov_{k + 1, l + 1} = \frac{1}{(k + 1)(l + 1)} \oint \oint \Big(\log(u/(u - 1)) - \log S_{\wrho}(u - 1)\Big)^{l + 1} \\
  &\phantom{===============} \Big(\log(w/(w - 1)) - \log S_{\wrho}(w - 1)\Big)^{k + 1} G^{(1, 1)}(u, w)\frac{du}{2\pi\ii} \frac{dw}{2\pi\ii}\\
    &= \frac{1}{(k + 1)(l + 1)} \oint \oint [-\log M_{\wrho}^{-1}(u - 1)]^{l + 1} [-\log M_{\wrho}^{-1}(w - 1)]^{k + 1} G^{(1, 1)}(u, w) \frac{du}{2\pi\ii} \frac{dw}{2\pi\ii}\\
    &= \frac{1}{(k + 1)(l + 1)} \oint \oint z^{l + 1} v^{k + 1} \partial_z \partial_v G\Big(M_{\wrho}(e^{-z}) + 1, M_{\wrho}(e^{-v}) + 1\Big) \frac{dz}{2\pi\ii} \frac{dv}{2\pi \ii}\\
    &= \oint \oint z^l v^k H(z, v) \frac{dz}{2\pi\ii} \frac{dv}{2\pi \ii}\\
  &= \frac{1}{(2\pi\ii)^2} \lim_{\eps \to 0} \int_{\supp d\rho} \int_{\supp d\rho} z^l v^k \Big[H(z + \ii\eps, v + \ii \eps) - H(z - \ii\eps, v + \ii \eps)\\
  &\phantom{=================}- H(z + \ii\eps, v - \ii \eps) + H(z - \ii\eps, v - \ii \eps)\Big] dz dv
  \end{align*}
  for
  \begin{align*}
    G(u, w) &= \sum_{i = 1}^M \left[\log(M_{\wrho_i}^{-1}(u - 1) - M_{\wrho_i}^{-1}(w - 1)) - \log(u - w)\right] + \log(u - w)\\
    G^{(1, 1)}(u, w) &= \left(\sum_{i = 1}^M \frac{1}{M_{\wrho_i}'(M^{-1}_{\wrho_i}(u - 1)) M_{\wrho_i}'(M^{-1}_{\wrho_i}(w - 1)) (M^{-1}_{\wrho_i}(u - 1) - M^{-1}_{\wrho_i}(w - 1))^2} - \frac{M - 1}{(u - w)^2}\right)\\
    H(z, v) &= G\Big(M_{\wrho}(e^{-z}) + 1, M_{\wrho}(e^{-v}) + 1\Big),
  \end{align*}
  which means that the covariance of the relevant Gaussian field is given by
  \begin{align*}
    K(t, s) &:= \lim_{\eps \to 0^+} -\frac{1}{4\pi^2}\left[H(t + \ii\eps, s + \ii \eps) - H(t + \ii\eps, s - \ii \eps) - H(t - \ii\eps, s + \ii \eps) + H(t - \ii\eps, s - \ii \eps)\right]\\
    &= \lim_{\eps \to 0^+} -\frac{1}{2\pi^2}\, \Re\left[H(t + \ii\eps, s + \ii \eps) - H(t + \ii\eps, s - \ii \eps)\right].
  \end{align*}
By (\ref{eq:m-stieltjes}), we find that 
\[
\lim_{\eps \to 0^+} \log\left|\frac{M_{\wrho}(e^{-t + \ii \eps}) - M_{\wrho}(e^{-s + \ii \eps})}{M_{\wrho}(e^{-t + \ii \eps}) - M_{\wrho}(e^{-s - \ii \eps})}\right| = \log\left|\frac{M_{\wrho}(e^{-t}) - M_{\wrho}(e^{-s}) + \ii \pi (p(t) - p(s))}{M_{\wrho}(e^{-t}) - M_{\wrho}(e^{-s}) + \ii \pi (p(t) + p(s))}\right|.
\]
Similarly, we see that
\begin{align*}
\lim_{\eps \to 0^+} \log\left|\frac{M_{\wrho_i}^{-1}(M_{\wrho}(e^{-t+ \ii \eps})) - M_{\wrho_i}^{-1}(M_{\wrho}(e^{-s + \ii \eps}))}{M_{\wrho_i}^{-1}(M_{\wrho}(e^{-t + \ii \eps})) - M_{\wrho_i}^{-1}(M_{\wrho}(e^{-s - \ii \eps}))}\right| = \log\left|\frac{M_{\wrho_i}^{-1}(M_{\wrho}(e^{-t}) + \ii \pi p(t)) - M_{\wrho_i}^{-1}(M_{\wrho}(e^{-s}) + \ii \pi p(s))}{M_{\wrho_i}^{-1}(M_{\wrho}(e^{-t}) + \ii \pi p(t)) - M_{\wrho_i}^{-1}(M_{\wrho}(e^{-s}) - \ii \pi p(s))}\right|.
\end{align*}
Putting these together, we conclude that
\begin{multline*}
  K(t, s) = -\frac{1}{2\pi^2} \log\left|\frac{M_{\wrho}(e^{-t}) - M_{\wrho}(e^{-s}) + \ii \pi (p(t) - p(s))}{M_{\wrho}(e^{-t}) - M_{\wrho}(e^{-s}) + \ii \pi (p(t) + p(s))}\right| \\
  - \frac{1}{2\pi^2} \sum_{i = 1}^M \left[\log\left|\frac{M_{\wrho_i}^{-1}(M_{\wrho}(e^{-t}) + \ii \pi p(t)) - M_{\wrho_i}^{-1}(M_{\wrho}(e^{-s}) + \ii \pi p(s))}{M_{\wrho_i}^{-1}(M_{\wrho}(e^{-t}) + \ii \pi p(t)) - M_{\wrho_i}^{-1}(M_{\wrho}(e^{-s}) - \ii \pi p(s))}\right| - \log\left|\frac{M_{\wrho}(e^{-t}) - M_{\wrho}(e^{-s}) + \ii \pi (p(t) - p(s))}{M_{\wrho}(e^{-t}) - M_{\wrho}(e^{-s}) + \ii \pi (p(t) + p(s))}\right|\right]. \qedhere
\end{multline*}
\end{proof}

\begin{remark}
  The function $\omega_i := M_{\wrho_i}^{-1} \circ M_{\wrho}$ on $\CC - [0, \infty)$ which appears in the expression
  \[
  M_{\wrho_i}^{-1}(M_{\wrho}(e^{-t}) \pm \ii \pi p(t)) = \lim_{\eps \to 0^\pm} M_{\wrho_i}^{-1}(M_{\wrho}(e^{-t + \ii \eps}))
  \]
  in Theorem \ref{thm:finite-uni-product} is the subordination function corresponding to multiplicative convolution of measures as defined in \cite{Voi93, Bia98, BB07}.  We relate this function to $\supp d\wrho$ in the following Lemma \ref{lem:sub-extend}.
  
  \begin{lemma} \label{lem:sub-extend}
  If the measures $d\wrho$ and $d\wrho_i$ are absolutely continuous with respect to Lebesgue measure with $\alpha$-H\"older continuous densities, then $\lim_{\eps \to 0^+} \omega_i(t + \ii \eps)$ is real for some $t \in [0, \infty)$ if and only if $t^{-1} \notin \supp d\wrho$.
  \end{lemma}
  \begin{proof}
  If $t^{-1} \notin \supp d\wrho$, the statement follows from \cite[Lemma 3.2]{BBCF17}.  Now, suppose that $\lim_{\eps \to 0^+} \omega_i(t + \ii \eps)$ is real; if $t = 0$, we are done since $d\wrho$ is compactly supported.  Otherwise, suppose for the sake of contradiction that $t^{-1} \in \supp d\wrho$.  Define
  \[
  d\wrho_i^* = d\wrho_1 \boxtimes \cdots \boxtimes \widehat{d\wrho_i} \boxtimes \cdots \boxtimes d\wrho_M
  \]
  so that
  \[
  d\wrho = d\wrho_i \boxtimes d\wrho_i^*.
  \]
  Letting $\omega_i^*$ denote the subordination function for $d\wrho_i^*$, we get from \cite[Equation 3.7]{BBCF17} that for $z \in \CC \setminus [0, \infty)$ we have
  \begin{equation} \label{eq:sub-func}
  \omega_i(z) \omega_i^*(z) = z \eta_{\wrho}(z) = z \, \eta_{\wrho_i}(\omega_i^*(z)) = z \, \eta_{\wrho_i^*}(\omega_i(z))
  \end{equation}
  for
  \[
  \eta_{\wrho}(z) := \frac{M_{\wrho}(z)}{1 + M_{\wrho}(z)}.
  \]
  Equating the second and fourth terms in (\ref{eq:sub-func}), we find that $M_{\wrho}(z) = M_{\wrho_i}(\omega_i(z))$ for $z \in \CC \setminus [0, \infty)$.  In particular, this means that
  \[
  \lim_{\eps \to 0^+} [M_{\wrho}(t + \ii \eps) - M_{\wrho}(t - \ii \eps)] = \lim_{\eps \to 0^+} [M_{\wrho_i}(\omega_i(t + \ii \eps)) - M_{\wrho_i}(\omega_i(t - \ii \eps))].
  \]
  Since $t^{-1} \in \supp d\wrho$ and $d\wrho$ has a $\alpha$-H\"older continuous density, this expression is nonzero by (\ref{eq:m-stieltjes}), hence we must have $\lim_{\eps \to 0^+} \omega_i(t + \ii \eps) \in [0, \infty)$.  In this case, we see that 
  \[
\lim_{\eps \to 0^+}\arg \omega_i^*(t + \ii \eps) = \lim_{\eps \to 0^+}\arg \frac{(t + \ii \eps)\eta_{\wrho_i}(\omega_i^*(t + \ii \eps))}{\omega_i(t + \ii \eps)} = \lim_{\eps \to 0^+} \arg \eta_{\wrho_i}(\omega_i^*(t + \ii \eps)).
\]
Now, if $w = \lim_{\eps \to 0^+} \omega_i^*(t + \ii \eps) \notin [0, \infty)$, then $\arg \eta_{\wrho_i}$ is continuous at $w$ and hence
\[
\arg w = \arg \eta_{\wrho_i}(w),
\]
which implies by property (a) of $\eta_{\wrho_i}(z)$ in \cite[Section 3.2]{BBCF17} that $\wrho_i$ is a point mass, which is excluded by the given.  We conclude that $\lim_{\eps \to 0^+} \omega_i^*(t + \ii \eps) \in [0, \infty)$.  Substituting $z = t + \ii \eps$ into the first equality in (\ref{eq:sub-func}) and taking the limit as $\eps \to 0^+$, we find that
\[
\lim_{\eps \to 0^+} \eta_{\wrho}(t + \ii \eps) \in \RR,
\]
which implies by virtue of the $\alpha$-H\"older continuous density of $d\wrho$ that $t^{-1} \notin \supp d\wrho$, as desired.
  \end{proof}
\end{remark}

\begin{corr} \label{corr:log-density}
Suppose that the density $p(s)$ for $d\rho(s)$ is $C^{1 + \alpha}$ for some $\alpha > 0$ on the interior of $\supp d\rho$ and no $d\wrho_i$ contains a point mass.  Then for $s$ in the interior of $\supp d\rho$, the covariance of Corollary \ref{corr:finite-uni-density} satisfies
  \begin{align*}
    &K(t, s) = -\frac{1}{2\pi^2}\log|t - s| - \frac{1}{2\pi^2} \log \left|\frac{\partial_t[M_{\wrho}(e^{-t})] + \ii \pi p'(t)}{2\pi p(t)}\right| \\
    & + \frac{1}{2\pi^2} \sum_{i = 1}^M \log\left|\frac{(M_{\wrho_i}^{-1}(M_{\wrho}(e^{-t}) + \ii \pi p(t)) -  M_{\wrho_i}^{-1}(M_{\wrho}(e^{-t}) - \ii \pi p(t))) M_{\wrho_i}'(M_{\wrho_i}^{-1}(M_{\wrho}(e^{-t}) + \ii \pi p(t)))}{2\pi p(t)}\right| + o(t - s)
  \end{align*}
  as $(t - s) \to 0$, meaning that it matches that of a log-correlated Gaussian field along the diagonal.
\end{corr}
\begin{proof}
By exchanging the derivative and Cauchy principal value integral, we see that 
\[
\partial_t M_{\wrho}(e^{-t}) = \text{p.v.} \int \frac{1}{1 - e^x} p'(x + t) dx
\]
exists if $p(s)$ is $C^{1 + \alpha}$ near $s = t$. This implies that for $s$ in the interior of $\supp d\rho$, as $t \to s$ we have
  \begin{align} \label{eq:log-term}
    \log \left| \frac{M_{\wrho}(e^{-t}) - M_{\wrho}(e^{-s}) + \ii \pi(p(t) - p(s))}{M_{\wrho}(e^{-t}) - M_{\wrho}(e^{-s}) + \ii \pi(p(t) + p(s))}\right| &= \log \left|\frac{(t - s)[\partial_t[M_{\wrho}(e^{-t})] + \ii \pi p'(t)]}{(t - s)[\partial_t[M_{\wrho}(e^{-t})] + \ii \pi p'(t)] - 2 \ii \pi p(t)} \right| + o(t - s)\\ \nonumber
      &= \log|t - s| + \log \left|\frac{\partial_t[M_{\wrho}(e^{-t})] + \ii \pi p'(t)}{2\pi p(t)}\right| + o(t - s).
  \end{align}
  Noting that $\partial_t M_{\wrho_i}^{-1}(t) = \frac{1}{M_{\wrho_i}'(M_{\wrho_i}^{-1}(t))}$, we find therefore that as $t \to s$ we have
  \begin{align*}
    M_{\wrho_i}^{-1}(M_{\wrho}(e^{-t}) + \ii \pi p(t)) -& M_{\wrho_i}^{-1}(M_{\wrho}(e^{-s}) + \ii \pi p(s))\\ &= (t - s) [\partial_t M_{\wrho}(e^{-t}) + \ii \pi p'(t)] \frac{1}{M_{\wrho_i}'(M_{\wrho_i}^{-1}(M_{\wrho}(e^{-t}) + \ii \pi p(t)))} + o(t - s)\\
    M_{\wrho_i}^{-1}(M_{\wrho}(e^{-t}) + \ii \pi p(t)) -& M_{\wrho_i}^{-1}(M_{\wrho}(e^{-s}) - \ii \pi p(s))\\ &= M_{\wrho_i}^{-1}(M_{\wrho}(e^{-t}) + \ii \pi p(t)) -  M_{\wrho_i}^{-1}(M_{\wrho}(e^{-t}) - \ii \pi p(t)) + o(t - s).
  \end{align*}
  Putting these together, if $M = 1$, we see that $K(t, s) = 0$; if $M > 1$, we see that as $t \to s$, we have that
  \begin{multline} \label{eq:rem-term}
  \log\left|\frac{M_{\wrho_i}^{-1}(M_{\wrho}(e^{-t}) + \ii \pi p(t)) - M_{\wrho_i}^{-1}(M_{\wrho}(e^{-s}) + \ii \pi p(s))}{M_{\wrho_i}^{-1}(M_{\wrho}(e^{-t}) + \ii \pi p(t)) - M_{\wrho_i}^{-1}(M_{\wrho}(e^{-s}) - \ii \pi p(s))}\right| - \log \left| \frac{M_{\wrho}(e^{-t}) - M_{\wrho}(e^{-s}) + \ii \pi(p(t) - p(s))}{M_{\wrho}(e^{-t}) - M_{\wrho}(e^{-s}) + \ii \pi(p(t) + p(s))}\right| \\
  = \log\left|\frac{(M_{\wrho_i}^{-1}(M_{\wrho}(e^{-t}) + \ii \pi p(t)) -  M_{\wrho_i}^{-1}(M_{\wrho}(e^{-t}) - \ii \pi p(t))) M_{\wrho_i}'(M_{\wrho_i}^{-1}(M_{\wrho}(e^{-t}) + \ii \pi p(t)))}{2\pi p(t)}\right| + o(t - s),
  \end{multline}
  where since $e^t \in \supp d\wrho$, by Lemma \ref{lem:sub-extend} we have for $M > 1$ that
  \[
  \lim_{\eps \to 0^+} [\omega_i(t + \ii \eps) - \omega_i(t - \ii \eps)] =  M_{\wrho_i}^{-1}(M_{\wrho}(e^{-t}) + \ii \pi p(t)) -  M_{\wrho_i}^{-1}(M_{\wrho}(e^{-t}) - \ii \pi p(t)) \neq 0.
  \]
  Adding (\ref{eq:log-term}) and the summation of (\ref{eq:rem-term}) over all $i$ yields the claim.
\end{proof}

\begin{remark}
  The additive analogue of Theorem \ref{thm:finite-uni-product} and Corollary \ref{corr:finite-uni-density} for two matrices was studied in \cite{MSS07}, \cite{CMSS}, \cite[Chapter 5]{MS17}, \cite[Chapter 10]{PS11}, and \cite[Proposition 9.8]{BG16}.  They considered unitarily invariant Hermitian random matrices $A_N^1, \ldots, A_N^M$ such that the empirical spectral measures $d\rho_{A^i}^N$ of $A_N^i$ converge weakly to deterministic measures $d\rho_{A^i}$.  Define the $R$-transform and Cauchy transform of a measure $\rho$ by 
  \[
  R_\rho(z) := G_\rho^{-1}(z) - z^{-1} \qquad \text{ and } \qquad G_\rho(z) := \int \frac{1}{z - s} d\rho(s)
  \]
  Then the empirical spectral measure of $X_M = \sum_{i = 1}^M A_N^i$ concentrates on a measure $d\rho_{X_M} = d\rho_{A^1} \boxplus \cdots \boxplus d\rho_{A^M}$ for which
  \[
  R_{\rho_{X_M}}(z) = \sum_{i = 1}^M R_{\rho_{A^i}}(z).
  \]
  In \cite[Chapter 10]{PS11}, it was shown further that the fluctuations of the spectrum $\lambda^N$ of $C_N$ satisfy
  \begin{multline*}
\lim_{N \to \infty} \Cov\Big(p_k(\lambda^N), p_l(\lambda^N)\Big) = \oint \oint z^k w^l \partial_z \partial_w \Big[\sum_{i = 1}^M \log\Big(G_{\rho_{A^i}}^{-1}(G_{\rho_{X_M}}(z)) - G_{\rho_{A^i}}^{-1}(G_{\rho_{X_M}}(w))\Big) \\- \log(z - w) - (M - 1) \log\Big(G_{\rho_{X_M}}(z)^{-1} - G_{\rho_{X_M}}(w)^{-1}\Big)\Big] \frac{dz}{2\pi\ii} \frac{dw}{2\pi\ii},
\end{multline*}
  where both contours enclose only the poles at $\infty$.  Applying the same reasoning as in the proof of Theorem \ref{thm:finite-uni-product}, this implies that the height function
  \[
  \cH_N(t) := \#\{\lambda_i^N \leq t\}
  \]
  has moments
  \[
  H_k := \int (\cH_N(t) - \EE[\cH_N(t)]) t^k dt
  \]
  with limiting covariance
  \begin{align*}
& \lim_{N \to \infty} \Cov(H_k, H_l) = \frac{1}{(k + 1)(l + 1)} \oint \oint z^{k + 1} w^{l + 1} \partial_z \partial_w \Big[ \sum_{i = 1}^M \log\Big(G_{\rho_{A^i}}^{-1}(G_{\rho_{X_M}}(z)) - G_{\rho_{A^i}}^{-1}(G_{\rho_{X_M}}(w))\Big)\\
      &\phantom{=================} - \log(z - w) - (M - 1)\log\Big(G_{\rho_{X_{M}}}(z)^{-1} - G_{\rho_{X_M}}(w)^{-1}\Big)\Big] \frac{dz}{2\pi\ii} \frac{dw}{2\pi\ii}\\
 &= \oint \oint z^k w^l  \Big[ \sum_{i = 1}^M \log\Big(G_{\rho_{A^i}}^{-1}(G_{\rho_{X_M}}(z)) - G_{\rho_{A^i}}^{-1}(G_{\rho_{X_M}}(w))\Big)\\
   &\phantom{====} - \log(z - w) - (M - 1)\log\Big(G_{\rho_{X_{M}}}(z)^{-1} - G_{\rho_{X_M}}(w)^{-1}\Big)\Big] \frac{dz}{2\pi\ii} \frac{dw}{2\pi\ii}\\
 &= - \frac{1}{2\pi^2}\int_I \int_I z^k w^l \sum_{i = 1}^M \left[ \log \left| \frac{G_{\rho_{A^i}}^{-1}(H_{X_M}(z)) - G_{\rho_{A^i}}^{-1}(H_{X_M}(w))}{G_{\rho_{A^i}}^{-1}(H_{X_M}(z)) - \overline{G_{\rho_{A^i}}^{-1}(H_{X_M}(w))}} \right| - \log \left|\frac{H_{X_M}(z)^{-1} - H_{X_M}(w)^{-1}}{H_{X_M}(z)^{-1} - \overline{H_{X_M}(w)}^{-1}}\right|\right] dz dw\\
 &\phantom{====} - \frac{1}{2\pi^2} \int_I\int_I z^k w^l \log \left|\frac{H_{X_M}(z)^{-1} - H_{X_M}(w)^{-1}}{H_{X_M}(z)^{-1} - \overline{H_{X_M}(w)^{-1}}}\right| dz dw,
  \end{align*}
  where $I = \supp d\rho_{X_M}$ and
  \[
  H_{X_M}(z) := \lim_{\eps \to 0^+} G_{\rho_{X_M}}(z + \ii \eps).
  \]
  This means that $\cH_N(t) - \EE[\cH_N(t)]$ converges to the Gaussian field on $\supp d\rho_{X_M}$ with covariance
  \begin{multline} \label{eq:add-covar}
  K(t, s) := - \frac{1}{2\pi^2} \sum_{i = 1}^M \left[ \log \left| \frac{G_{\rho_{A^i}}^{-1}(H_{X_M}(t)) - G_{\rho_{A^i}}^{-1}(H_{X_M}(s))}{G_{\rho_{A^i}}^{-1}(H_{X_M}(t)) - \overline{G_{\rho_{A^i}}^{-1}(H_{X_M}(s))}} \right| - \log \left|\frac{H_{X_M}(t)^{-1} - H_{X_M}(s)^{-1}}{H_{X_M}(t)^{-1} - \overline{H_{X_M}(s)}^{-1}}\right|\right]\\ - \frac{1}{2\pi^2} \log \left|\frac{H_{X_M}(t)^{-1} - H_{X_M}(s)^{-1}}{H_{X_M}(t)^{-1} - \overline{H_{X_M}(s)}^{-1}}\right|.
  \end{multline}
\end{remark}

\begin{theorem} \label{thm:jacobi-single}
If the $X^i_N$ have spectrum which is a Jacobi ensemble with parameters $\alpha$ and $R$ with $R/N = \hat{R} + O(N^{-1})$ and $\alpha/N = \hat{\alpha} + O(N^{-1})$ for $\hat{R} \geq 1$ and $\halpha > 0$, then as $N \to \infty$ the empirical log-spectral measure $d\lambda^N$ of $X_N$ converges in probability in the sense of moments to the measure $d\rho$ whose pushforward $d\wrho$ under the exponential map has $S$-transform
\[
S_{\wrho}(u) = \Big(\frac{\halpha + \hat{R} + u + 1}{\halpha + u + 1}\Big)^M.
\]
For the height function defined by
\[
\cH_N(t) := \#\{e^{\lambda_i^N} \leq t\},
\]
the centered height function $\cH_N(t) - \EE[\cH_N(t)]$ converges in the sense of moments to the Gaussian random field on $[0, 1]$ with covariance
\[
K(t, s) = - \frac{1}{2\pi^2} \log\left|\frac{M_{\wrho}(t^{-1}) - M_{\wrho}(s^{-1}) + \ii \pi (p(\log t) - p(\log s))}{M_{\wrho}(t^{-1}) - M_{\wrho}(s^{-1}) + \ii \pi (p(\log t) + p(\log s))}\right|,
\]
where $p(s)$ is the density of $d\wrho$ and $M_{\wrho}(t)$ is defined on $[0, 1]$ in the Cauchy principal value sense by 
\[
M_{\wrho}(t) := \text{p.v.} \int \frac{e^{s} t}{1 - e^{s} t} d\rho(s) \qquad \text{ and } \qquad d\rho(s) = p(s) ds.
\]
\end{theorem}
\begin{proof}
By Theorem \ref{thm:jacobi}, the log-spectral measure of each $X^i_N$ is CLT-appropriate, so by Lemma \ref{lem:prod-spec} the log-spectral measure $d\lambda^N$ of $X_N$ is CLT-appropriate with
\[
\Psi(u) = M \log \frac{\halpha + u}{\halpha + \hR + u} = -\log S_{\wrho}(u - 1) \qquad \text{ and } \qquad \Lambda(u, w) = 0.
\]
By Theorem \ref{thm:lln} and Lemma \ref{lem:chi-comp}, the log-spectral measure of $X_N$ converges in probability in the sense of moments to a measure with $k^\text{th}$ moment
\[
\fp_k = \frac{1}{k + 1} \oint \Big(\log(u/(u - 1)) - \log S_{\wrho}(u - 1)\Big)^{k + 1} \frac{du}{2\pi \ii} = \int s^k d\rho(s)
\]
as in the proof of Theorem \ref{thm:finite-uni-product}. Because $d\rho(s)$ is compactly supported, convergence of moments implies convergence of random measures, as desired.

For the central limit theorem, by integration by parts we see that
\begin{align*}
\int_{0}^1 \cH_N(t) t^k dt &= \int_{-\infty}^0 \cH_N(e^s) e^{s(k + 1)} ds\\
&= \frac{N}{k + 1} - \frac{1}{k + 1} \int_{-\infty}^0 \cH_N'(e^s) e^{s(k + 1)} ds\\
&= \frac{N}{k + 1} - \frac{p_{k + 1}(e^{\lambda^N})}{k + 1},
\end{align*}
which means that
\[
H_k := \int_0^1 (\cH_N(t) - \EE[\cH_N(t)]) t^k dt = - \frac{1}{k + 1} \Big(p_{k + 1}(e^{\lambda^N}) - \EE[p_{k + 1}(e^{\lambda^N})]\Big).
\]
By Theorem \ref{thm:clt} and Lemma \ref{lem:chi-comp}, we see that
\begin{align*}
\Cov(H_k, H_l) &= \frac{1}{(k + 1)(l + 1)} \sum_{m, n = 0}^\infty \Cov\Big(\frac{p_n(\lambda^N) (k + 1)^n}{n!}, \frac{p_m(\lambda^N) (l + 1)^m}{m!}\Big) \\
&= \frac{1}{(k + 1)(l + 1)} \sum_{m, n = 0}^\infty \frac{(k + 1)^n (l + 1)^m}{n! m!} \oint \oint \Big(\log(u/(u - 1)) - \log S_{\wrho}(u - 1)\Big)^{n}\\
&\phantom{===}\Big(\log(w/(w - 1)) - \log S_{\wrho}(w - 1)\Big)^{m} \frac{1}{(u - w)^2} \frac{du}{2\pi\ii} \frac{dw}{2\pi\ii}\\
&= \frac{1}{(k + 1)(l + 1)} \oint \oint M^{-1}_{\wrho}(u - 1)^{-k - 1} M^{-1}_{\wrho}(w - 1)^{-l - 1} \frac{1}{(u - w)^2} \frac{du}{2\pi\ii} \frac{dw}{2\pi\ii}\\
&= \frac{1}{(k + 1)(l + 1)} \oint \oint z^{k + 1} v^{l + 1} \frac{M_{\wrho}'(z^{-1}) M_{\wrho}'(v^{-1}) z^{-2} v^{-2}}{(M_{\wrho}(z^{-1}) - M_{\wrho}(v^{-1}))^2} \frac{dz}{2\pi\ii} \frac{dv}{2\pi\ii}\\
&= \oint \oint z^k v^l \log\Big(M_{\wrho}(z^{-1}) - M_{\wrho}(v^{-1})\Big) \frac{dz}{2\pi\ii} \frac{dv}{2\pi\ii}.
\end{align*}
Applying the same argument as in the proof of Corollary \ref{corr:finite-uni-density}, we find that the covariance of the relevant Gaussian field is given by
\[
K(t, s) = - \frac{1}{2\pi^2} \log\left|\frac{M_{\wrho}(t^{-1}) - M_{\wrho}(s^{-1}) + \ii \pi (p(\log t) - p(\log s))}{M_{\wrho}(t^{-1}) - M_{\wrho}(s^{-1}) + \ii \pi (p(\log t) + p(\log s))}\right|. \qedhere
\]
\end{proof}

\begin{remark}
  When $M = 1$, our proof of Theorem \ref{thm:jacobi-single} recovers the central limit theorem for global fluctuations of the Jacobi ensemble.  There are now several different proofs of this result (c.f. the discussion before \cite[Proposition 1.3]{BG15}).  We now match our covariance for this case with that of \cite{BG15} for $\beta = 2$.   More specifically, we may compute
\[
S_{\wrho}(u) = \frac{\hat{\alpha} + \hR + u + 1}{\halpha + u + 1},
\]
which implies that our computation of $\Cov(H_k, H_l)$ coincides with the expression in \cite[Theorem 4.1]{BG15} after changing variables to $u \mapsto -u$ and $w \mapsto -w$.  Furthermore, we see that 
\[
M_{\wrho}^{-1}(-u - 1)^{-1} = \frac{u}{u + 1} \frac{u - \halpha}{u - \halpha - \hR},
\]
which means that $- 1 - M_{\wrho}(u^{-1})$ coincides with $\Omega(u; 1)$ for $\hat{M} = \hat{R}$ in \cite[Definition 4.11]{BG15}.  If we deform the contours differently in the last step of Theorem \ref{thm:jacobi-single} and make the change of variables $u \mapsto -u$ and $w \mapsto -w$, our result coincides with \cite[Theorem 4.13]{BG15}. 
\end{remark}

\begin{theorem} \label{thm:ginibre-single}
If $X^i_N = (G^i_N)^* G^i_N$ with $G^i_N$ a Ginibre ensemble with parameters $L$ and $N$ with $L/N = \gamma + O(N^{-1})$, then as $N \to \infty$, the empirical log-spectral measure $d\lambda^N$ of $X_N$ converges in probability in the sense of moments to the measure $d\rho$ whose pushforward
under the exponential map has $S$-transform
\[
S_{\wrho}(u) = \frac{1}{(u + \gamma)^M}.
\]
For the height function defined by
\[
\cH_N(t) := \#\{e^{\lambda^N_i} \leq t\},
\]
the centered height function $\cH_N(t) - \EE[\cH_N(t)]$ converges in the sense of moments to the Gaussian random field on $[0, \infty)$ with covariance
\[
K(t, s) = - \frac{1}{2\pi^2} \log\left|\frac{M_{\wrho}(t^{-1}) - M_{\wrho}(s^{-1}) + \ii \pi (p(\log t) - p(\log s))}{M_{\wrho}(t^{-1}) - M_{\wrho}(s^{-1}) + \ii \pi (p(\log t) + p(\log s))}\right|,
\]
where $p(s)$ is the density of $d\wrho$ and $M_{\wrho}(t)$ is defined on $[0, \infty)$ by the Cauchy principal value integral
\[
M_{\wrho}(t) = \text{p.v.} \int \frac{e^{s} t}{1 - e^{s} t} d\rho(t) \qquad \text{ and } \qquad d\rho(s) = p(s) ds.
\]
\end{theorem}
\begin{proof}
By Theorem \ref{thm:ginibre}, the log-spectral measure of each $X_N^i$ is CLT-appropriate, so by Lemma \ref{lem:prod-spec} the measure $d\lambda^N$ is CLT-appropriate with
\[
\Psi(u) = M \log(u + \gamma - 1) = - \log S_{\wrho}(u - 1) \qquad \text{ and } \qquad \Lambda(u, w) = 0.
\]
As in Theorem \ref{thm:jacobi-single}, by Theorem \ref{thm:lln}, the log-spectral measure of $X_N$ converges in probability in the sense of moments to a measure with $k^\text{th}$ moment which matches those of $d\rho$.  Because $d\rho(s)$ is compactly supported, convergence of moments implies convergence of random measures, as desired.

For the central limit theorem, similarly to Theorem \ref{thm:jacobi}, by integration by parts we see that
\[
\int_0^\infty \cH_N(t) t^k dt = \frac{N}{k + 1} - \frac{p_{k+1}(e^{\lambda^N})}{k + 1},
\]
which means that
\[
H_k := \int_0^\infty \Big(\cH_N(t) - \EE[\cH_N(t)]\Big) t^k dt = - \frac{1}{k + 1} \Big(p_{k + 1}(e^{\lambda^N}) - \EE[p_{k+1}(e^{\lambda^N})]\Big).
\]
By Theorem \ref{thm:clt} and Lemma \ref{lem:chi-comp}, we see that
\begin{align*}
\Cov(H_k, H_l) &= \frac{1}{(k + 1)(l + 1)} \sum_{m, n = 0}^\infty \frac{(k + 1)^n (l + 1)^m}{n! m!} \Cov(p_n(\lambda^N), p_m(\lambda^N)\\
&= \frac{1}{(k + 1)(l + 1)} \sum_{m, n = 0}^\infty \frac{(k + 1)^n (l + 1)^m}{n! m!} \oint\oint \Big(\log(u/(u - 1)) - \log S_{\wrho}(u - 1)\Big)^n \\
&\phantom{==} \Big(\log(w/(w - 1) - S_{\wrho}(w - 1)\Big)^m \frac{1}{(u - w)^2} \frac{du}{2\pi\ii} \frac{dw}{2\pi\ii}\\
&= \oint \oint z^k v^l \log\Big(M_{\wrho}(z^{-1}) - M_{\wrho}(v^{-1})\Big) \frac{dz}{2\pi\ii} \frac{dv}{2\pi \ii},
\end{align*}
where we apply the exact same arguments as in the proof of Theorem \ref{thm:jacobi}.  We therefore obtain that the covariance of the relevant Gaussian field is given by
\[
K(t, s) = - \frac{1}{2\pi^2} \log\left|\frac{M_{\wrho}(t^{-1}) - M_{\wrho}(s^{-1}) + \ii \pi (p(\log t) - p(\log s))}{M_{\wrho}(t^{-1}) - M_{\wrho}(s^{-1}) + \ii \pi (p(\log t) + p(\log s))}\right|. \qedhere
\]
\end{proof}

\begin{remark}
When $M = 1$, $d\wrho$ is the spectral measure of the Marchenko-Pastur distribution, as expected.  Our computation of the fluctuations of the moments of the spectral measure agrees with the single-level case of \cite[Proposition 1.2]{DP18} after an integration by parts and deformation of contours. 
\end{remark}

\subsection{Lyapunov exponents for matrices with fixed spectrum} \label{sec:lya-fixed}

We now apply our theorems to products of infinitely many random matrices.  Suppose that $Y_N^1, \ldots, Y_N^M$ are $N \times N$ random matrices which are right unitarily invariant, and let $X^i_N := (Y^i_N)^* Y^i_N$.  As before, we study the spectral measure of
\[
X_N := (Y_N^1 \cdots Y_N^M)^* (Y_N^1 \cdots Y_N^M),
\]
where we now make the assumption that $M, N \to \infty$ simultaneously.  In this setting, we expect the eigenvalues to grow exponentially in $M$, so we study the Lyapunov exponents
\[
\lambda^N_i := \frac{1}{M} \log \mu_i^N,
\]
where $\mu_1^N \geq \cdots \geq \mu_N^N$ are the eigenvalues of $X_N$.  Define their empirical measure by
\[
d\lambda^N := \frac{1}{N} \sum_{i = 1}^N \delta_{\lambda^N_i}
\]
In this setting, we may identify the limiting measure as the sum of a deterministic measure and an explicit Gaussian process under rescaling.  For $t \in \RR$, define the height function
  \[
  \cH_N(t) := \#\{\lambda^N_i \leq t\}.
  \]

\begin{theorem} \label{thm:lya}
If the $X^i_N$ have deterministic spectrum satisfying Assumption \ref{ass:measure} for a compactly supported non-atomic measure $d\rho$, then as $M, N \to \infty$, the empirical measure $d\lambda^N$ of the Lyapunov exponents converges in probability in the sense of moments to the measure
  \[
  d\lambda^\infty := \frac{-e^{-z}}{S_{\wrho}'(S_{\wrho}^{-1}(e^{-z}))} \bI_{[-\log S_{\wrho}(-1), -\log S_{\wrho}(0)]} dz.
  \]
  The rescaled centered height function $M^{1/2}(\cH_N(t) - \EE[\cH_N(t)])$ converges in the sense of moments to the Gaussian random field on $[-\log S_{\wrho}(-1), - \log S_{\wrho}(0)]$ with covariance
  \[
  K(t, s) = H\Big(S_{\wrho}^{-1}(e^{-t}) + 1, S_{\wrho}^{-1}(e^{-s}) + 1\Big) \frac{e^{-t} e^{-s}}{S_{\wrho}'(S^{-1}_{\wrho}(e^{-t})) S_{\wrho}'(S^{-1}_{\wrho}(e^{-s}))} + \delta(t - s)
  \]
  for
  \[
  H(u, w) := \frac{1}{M'_{\wrho}(M^{-1}_{\wrho}(u - 1)) M'_{\wrho}(M^{-1}_{\wrho}(w - 1)) (M^{-1}_{\wrho}(u - 1) - M^{-1}_{\wrho}(w - 1))^2} - \frac{1}{(u - w)^2}.
  \]
\end{theorem}
\begin{proof}
  By Theorem \ref{thm:lln-many}, Lemma \ref{lem:chi-comp}, and Theorem \ref{thm:unitary-inv}, $d\lambda^N$ converges in probability to the measure with moments
  \begin{align*}
    \lim_{N \to \infty} \frac{1}{N} \EE[p_k(\lambda^N)] &= \oint \log(u/(u - 1)) [-\log S_{\wrho}(u - 1)]^k \frac{du}{2\pi\ii}\\
    &= \int_0^1 [-\log S_{\wrho}(u - 1)]^k du.
  \end{align*}
Recall by \cite[Theorem 4.4]{HL00} that for a compactly supported measure $d\rho$ which is not a single atom, $S_{\wrho}(z)$ is strictly decreasing on $[-1, 0]$.  Therefore, we may change variables to obtain
  \begin{align*}
\lim_{N \to \infty} \frac{1}{N} \EE[p_k(\lambda^N)]    &= \int^{-\log S_{\wrho}(0)}_{-\log S_{\wrho}(-1)} z^k \frac{-e^{-z}}{S_{\wrho}'(S_{\wrho}^{-1}(e^{-z}))} dz,
  \end{align*}
  where the values of $S_{\wrho}(0)$ and $S_{\wrho}(-1)$ are given by Lemma \ref{lem:s-val}.   These are the moments of $d\lambda^\infty$, which uniquely identify it since it is compactly supported.

  For the central limit theorem, notice by integration by parts that
  \begin{align*}
&    \int_{-\log S_{\wrho}(-1)}^{-\log S_{\wrho}(0)}\cH_N(t) t^k dt \\&= \frac{\cH_N(-\log S_{\wrho}(0)) [-\log S_{\wrho}(0)]^{k + 1}}{k + 1} - \frac{\cH_N(-\log S_{\wrho}(-1)) [-\log S_{\wrho}(-1)]^{k + 1}}{k + 1} - \int_{-\log S_{\wrho}(-1)}^{-\log S_{\wrho}(0)} \cH_N'(t) \frac{t^{k + 1}}{k + 1} dt\\
    &= \frac{\cH_N(-\log S_{\wrho}(0)) [-\log S_{\wrho}(0)]^{k + 1}}{k + 1} - \frac{\cH_N(-\log S_{\wrho}(-1)) [-\log S_{\wrho}(-1)]^{k + 1}}{k + 1} - \frac{1}{k + 1} p_{k + 1}(\lambda^N),
  \end{align*}
  which means that
  \[
  H_k := \int_{-\log S_{\wrho}(-1)}^{-\log S_{\wrho}(0)} M^{1/2}(\cH_N(t) - \EE[\cH_N(t)]) t^k dt = - \frac{M^{1/2}}{k + 1} \Big(p_{k + 1}(\lambda^N) - \EE[p_{k + 1}(\lambda^N)]\Big).
  \]
  By Theorem \ref{thm:clt-many}, Lemma \ref{lem:chi-comp}, and Theorem \ref{thm:unitary-inv}, we conclude that $H_k$ are asymptotically Gaussian with covariance
  \begin{align*}
    \lim_{N \to \infty} \Cov(H_k, H_l) &= \oint\oint \log(u/(u - 1)) \log(w/(w - 1)) [-\log S_{\wrho}(u - 1)]^{k} [-\log S_{\wrho}(w - 1)]^{l} F^{(1, 1)}(u, w) \frac{du}{2\pi\ii} \frac{dw}{2\pi\ii}\\
    &\phantom{=} + \oint \log(u/(u - 1)) [-\log S_{\wrho}(u - 1)]^{k + l} \left[- \frac{S_{\wrho}'(u - 1)}{S_{\wrho}(u - 1)}\right] \frac{du}{2\pi\ii} \\
    &= \int_0^1 \int_0^1 [-\log S_{\wrho}(u - 1)]^{k} [-\log S_{\wrho}(w - 1)]^{l} H(u, w) du dw - \int_0^1 [-\log S_{\wrho}(u - 1)]^{k + l} \frac{S_{\wrho}'(u - 1)}{S_{\wrho}(u - 1)} du \\
    &= \int_{-\log S_{\wrho}(-1)}^{-\log S_{\wrho}(0)} \int_{-\log S_{\wrho}(-1)}^{-\log S_{\wrho}(0)} t^k s^l K(t, s) dt ds
  \end{align*}
  for
  \begin{align*}
  H(u, w) &:= \frac{1}{M'_{\wrho}(M^{-1}_{\wrho}(u - 1)) M'_{\wrho}(M^{-1}_{\wrho}(w - 1)) (M^{-1}_{\wrho}(u - 1) - M^{-1}_{\wrho}(w - 1))^2} - \frac{1}{(u - w)^2}\\
  K(t, s) &:= H\Big(S_{\wrho}^{-1}(e^{-t}) + 1, S_{\wrho}^{-1}(e^{-s}) + 1\Big) \frac{e^{-t} e^{-s}}{S_{\wrho}'(S^{-1}_{\wrho}(e^{-t})) S_{\wrho}'(S^{-1}_{\wrho}(e^{-s}))} + \delta(t - s),
  \end{align*}
  which yields the desired claim.
\end{proof}

\begin{remark}
The law of large numbers for Lyapunov exponents shown in Theorem \ref{thm:lya} agrees with the previous results of \cite[Theorem 2.11]{New84}, \cite[Theorem 2]{Kar08}, and \cite[Theorem 5.1]{Tuc}.
\end{remark}

\begin{remark}
By Theorem \ref{thm:lya}, the fluctuations of Lyapunov exponents have a white noise component when $M, N \to \infty$, while by Corollary \ref{corr:finite-uni-density}, they form a log-correlated Gaussian field when $M$ stays finite.  We demonstrate a formal limit transition between these cases.  For $M$ finite, suppose that all $d\rho_i$ are identical and equal to a non-atomic measure $d\mu$ in Corollary \ref{corr:finite-uni-density}. Recall by \cite[Theorem 4.4]{HL00} that this means $S_{\wmu}(z)$ is strictly decreasing on $[-1, 0]$.  Let $d\wrho^M = d\wmu^{\boxtimes M}$, and let $d\lambda^M$ be the empirical measure of the Lyapunov exponents, which is the pushforward of the corresponding $d\rho^M$ under the map $x \mapsto \frac{1}{M} x$.  Applying \cite[Theorem 5.1]{Tuc} and adjusting for a normalization factor of $2$, we see that as $M \to \infty$, the measure $d\lambda^M$ converges to
\[
d\lambda^\infty := \frac{-e^{-z}}{S_{\wmu}'(S_{\wmu}^{-1}(e^{-z}))} \bI_{[-\log S_{\wmu}(-1), -\log S_{\wmu}(0)]} dz,
\]
which coincides with the result of Theorem \ref{thm:lya}.  Now define the height function
\[
\cH_{N, M}(t) := \#\{\rho^M_i \leq t\}.
\]
By Corollary \ref{corr:finite-uni-density}, its centered version $\cH_{N, M}(t) - \EE[\cH_{N, M}(t)]$ converges as $N \to \infty$ to the Gaussian random field on $\RR$ with covariance
\begin{multline*}
K^M(t, s) = -\frac{M}{2\pi^2} \log\left|\frac{M_{\wmu}^{-1}(M_{\wrho^M}(e^{-t}) + \ii \pi p^M(t)) - M_{\wmu}^{-1}(M_{\wrho^M}(e^{-s}) + \ii \pi p^M(s))}{M_{\wmu}^{-1}(M_{\wrho^M}(e^{-t}) + \ii \pi p^M(t)) - M_{\wmu}^{-1}(M_{\wrho^M}(e^{-s}) - \ii \pi p^M(s))}\right|\\ + \frac{M - 1}{2\pi^2} \log \left|\frac{M_{\wrho^M}(e^{-t}) - M_{\wrho^M}(e^{-s}) + \ii \pi (p^M(t) - p^M(s))}{M_{\wrho^M}(e^{-t}) - M_{\wrho^M}(e^{-s}) + \ii \pi (p^M(t) + p^M(s))}\right|,
\end{multline*}
where $p^M(s)$ is the density of $d\wrho^M(s)$, which satisfies
\[
q(t) := \lim_{M \to \infty} M p^M(t M) = \frac{-e^{-t}}{S'_{\wmu}(S_{\wmu}^{-1}(e^{-t}))} \bI_{[-\log S_{\wmu}(-1), -\log S_{\wmu}(0)]}.
\]
The height function of $d\lambda^M$ is given by $\cH_{N, M}(t M)$ and hence its limiting covariance of
\[
M^{1/2}\Big(\cH_{N, M}(tM) - \EE[\cH_{N, M}(tM)]\Big)
\]
is given by $\lim_{M \to \infty} M K^M(t M, s M)$.  To analyze this limit, we consider the limits of the Stieltjes transform.

\begin{lemma} \label{lem:stieltjes-m}
  For $r \in (-\log S_{\wmu}(-1), -\log S_{\wmu}(0))$, we have the following:
  \begin{itemize}
    \item[(a)] uniformly on compact subsets we have
      \[
      \lim_{\eps \to 0^{\pm}} M_{\wrho^M}(e^{(-r + \ii \eps)M}) = \Big(S_{\wmu}^{-1}(e^{-r}) + C_1\Big) \pm \ii \Big(\pi M^{-1} \frac{e^{-r}}{S_{\wmu}'(S_{\wmu}^{-1}(e^{-r}))} + C_2 \Big),
      \]
      where $C_1 = O(M^{-1})$ and $C_2 = O(M^{-2})$ are real; 
    \item[(b)] uniformly on compact subsets we have
      \[
      \lim_{\eps \to 0^{\pm}} \partial_r[M_{\wrho^M}(e^{(-r + \ii \eps)M})] = \frac{e^{-r}}{S_{\wmu}'(S_{\wmu}^{-1}(e^{-r}))} + O(M^{-1}).
      \]
  \end{itemize}
\end{lemma}
\begin{proof}
First, by definition we have
\begin{align*}
  f(u) &:= \frac{1}{M} \log M_{\wrho^M}^{-1}(u) = \log S_{\wmu}(u) - \frac{1}{M} \log(1 + u^{-1})\\
  f'(u) &= \frac{S_{\wmu}'(u)}{S_{\wmu}(u)} + \frac{1}{M} \frac{1}{u^2 + u}.
\end{align*}
For (a), expanding in series we conclude that for $\eps \neq 0$, we have
\[
M_{\wrho^M}(e^{(-r + \ii \eps)M}) = S_{\wmu}^{-1}(e^{-r + \ii \eps}) + \frac{1}{M} \frac{e^{-r + \ii \eps}}{S_{\wmu}'(S^{-1}_{\wmu}(e^{-r + \ii \eps}))} \log\Big(1 + \frac{1}{S_{\wmu}^{-1}(e^{-r + \ii \eps})}\Big) + O(M^{-2}).
\]
Taking the limit as $\eps \to 0^{\pm}$ yields the result.  For (b), we find that for $u = M_{\wrho^M}(e^{(-r + \ii \eps)M})$ that
\begin{align*}
\lim_{\eps \to 0^{\pm}}  \partial_r [ M_{\wrho^M}(e^{(-r + \ii \eps)M})] &= \lim_{\eps \to 0^{\pm}} \left[\frac{S'(u)}{S(u)} + \frac{1}{M} \frac{1}{u^2 + u}\right]^{-1}= \frac{e^{-r}}{S_{\wmu}'(S_{\wmu}^{-1}(e^{-r}))} + O(M^{-1}). \qedhere
\end{align*}
\end{proof}

For $t, s$ in a compact subset of $(-\log S_{\wmu}(-1), -\log S_{\wmu}(0))$, we notice that
\[
M K^M(tM, sM) = K_1^M(t, s) + K_2^M(t, s)
\]
for
\begin{align*}
  K_1^M(t, s) &:= \lim_{\eps \to 0^+} -\frac{M}{2\pi^2}\log \left| \frac{M_{\wrho^M}(e^{(-t + \ii \eps)M}) - M_{\wrho^M}(e^{(-s + \ii \eps)M})}{M_{\wrho^M}(e^{(-t + \ii \eps)M}) - M_{\wrho^M}(e^{(-s - \ii \eps)M})}\right|\\
  K_2^M(t, s) &:= \lim_{\eps \to 0^+} -\frac{M^2}{2\pi^2}  \left[\log \left| \frac{M_{\wmu}^{-1}(M_{\wrho^M}(e^{(-t + \ii \eps) M})) - M_{\wmu}^{-1}(M_{\wrho^M}(e^{(-s + \ii \eps) M}))}{M_{\wrho^M}(e^{(-t + \ii \eps)M}) - M_{\wrho^M}(e^{(-s + \ii \eps)M})}\right|\right.\\&\phantom{====} - \left.\log \left| \frac{M_{\wmu}^{-1}(M_{\wrho^M}(e^{(-t + \ii \eps) M})) - M_{\wmu}^{-1}(M_{\wrho^M}(e^{(-s - \ii \eps) M}))}{M_{\wrho^M}(e^{(-t + \ii \eps)M}) - M_{\wrho^M}(e^{(-s - \ii \eps)M})}\right|\right].
\end{align*}
If $|t - s| > M^{-1/2}$, applying Lemma \ref{lem:stieltjes-m} we see that 
\begin{align*}
&M K^M(tM, sM) = \lim_{\eps \to 0^+} \left(- \frac{M^2}{2\pi^2} \log \left|\frac{M_{\wmu}^{-1}(M_{\wrho^M}(e^{(-t + \ii \eps) M})) - M_{\wmu}^{-1}(M_{\wrho^M}(e^{(-s + \ii \eps) M}))}{M_{\wmu}^{-1}(M_{\wrho^M}(e^{(-t + \ii \eps) M})) - M_{\wmu}^{-1}(M_{\wrho^M}(e^{(-s - \ii \eps) M}))}\right|\right.\\
&\phantom{===========}+\left. \frac{M^2 - M}{2\pi^2} \log \left| \frac{M_{\wrho^M}(e^{(-t + \ii \eps)M}) - M_{\wrho^M}(e^{(-s + \ii \eps)M})}{M_{\wrho^M}(e^{(-t + \ii \eps)M}) - M_{\wrho^M}(e^{(-s - \ii \eps)M})}\right|\right)\\
&= -\frac{M^2}{2\pi^2} \log\left| \frac{M_{\wmu}^{-1}(S_{\wmu}^{-1}(e^{-t})) - M_{\wmu}^{-1}(S_{\wmu}^{-1}(e^{-s})) + B_1 + \ii \pi M^{-1}(\frac{q(t)}{M_{\wmu}'(M_{\wmu}^{-1}(S_{\wmu}^{-1}(e^{-t})))} - \frac{q(s)}{M_{\wmu}'(M_{\wmu}^{-1}(S_{\wmu}^{-1}(e^{-s})))} + B_2)}{M_{\wmu}^{-1}(S_{\wmu}^{-1}(e^{-t})) - M_{\wmu}^{-1}(S_{\wmu}^{-1}(e^{-s})) + B_1 + \ii \pi M^{-1}(\frac{q(t)}{M_{\wmu}'(M_{\wmu}^{-1}(S_{\wmu}^{-1}(e^{-t})))} + \frac{q(s)}{M_{\wmu}'(M_{\wmu}^{-1}(S_{\wmu}^{-1}(e^{-s})))} + B_3)} \right|\\
&\phantom{===}+ \frac{M^2 - M}{2\pi^2} \log \left| \frac{S_{\wmu}^{-1}(e^{-t}) - S_{\wmu}^{-1}(e^{-s}) + C_1 + \ii \pi M^{-1} (q(t) - q(s) + C_2)}{S_{\wmu}^{-1}(e^{-t}) - S_{\wmu}^{-1}(e^{-s}) + C_1 + \ii \pi M^{-1} (q(t) + q(s) + C_3)}\right|\\
&= - \frac{M^2}{4\pi^2} \log\left(1 - \frac{4 \pi^2 M^{-2} \frac{q(t)}{M_{\wmu}'(M_{\wmu}^{-1}(S_{\wmu}^{-1}(e^{-t})))}\frac{q(s)}{M_{\wmu}'(M_{\wmu}^{-1}(S_{\wmu}^{-1}(e^{-s})))}}{(M_{\wmu}^{-1}(S_{\wmu}^{-1}(e^{-t})) - M_{\wmu}^{-1}(S_{\wmu}^{-1}(e^{-s})))^2} + O(M^{-3})\right)\\
&\phantom{===}+ \frac{M^2 - M}{4\pi^2} \log\left(1 - \frac{4\pi^2 M^{-2} q(t) q(s)}{(S_{\wmu}^{-1}(e^{-t}) - S_{\wmu}^{-1}(e^{-s}))^2} + O(M^{-3})\right)\\
  &= \frac{q(t)}{M_{\wmu}'(M_{\wmu}^{-1}(S_{\wmu}^{-1}(e^{-t})))}\frac{q(s)}{M_{\wmu}'(M_{\wmu}^{-1}(S_{\wmu}^{-1}(e^{-s})))} \frac{1}{(M_{\wmu}^{-1}(S_{\wmu}^{-1}(e^{-t})) - M_{\wmu}^{-1}(S_{\wmu}^{-1}(e^{-s})))^2}\\
  &\phantom{===}- \frac{q(t) q(s)}{(S_{\wmu}^{-1}(e^{-t}) - S_{\wmu}^{-1}(e^{-s}))^2} + O(M^{-1})\\
&= K(t, s) + O(M^{-1}),
\end{align*}
where the real constants $B_1, B_2, B_3, C_1, C_2, C_3$ are all of size $O(M^{-1})$, and $K(t, s)$ is the kernel of Theorem \ref{thm:lya}.  On the other hand, if $|t - s| \leq M^{-1/2}$, by Lemma \ref{lem:stieltjes-m} we have
\begin{align*}
  K_1^M(t, s) &= \lim_{\eps \to 0^+} \frac{M}{2\pi^2} \log \left|1 + \frac{M_{\wrho^M}(e^{(-s + \ii \eps)M}) - M_{\wrho^M}(e^{(s - \ii \eps)M})}{M_{\wrho^M}(e^{(-t + \ii \eps)M}) - M_{\wrho^M}(e^{(-s + \ii \eps)M})}\right| \\
  &= \lim_{\eps \to 0^+} \frac{M}{2\pi^2} \log \left|1 + \frac{M_{\wrho^M}(e^{(-s + \ii \eps)M}) - M_{\wrho^M}(e^{(s - \ii \eps)M})}{(t - s)\partial_r[M_{\wrho^M}(e^{(-r + \ii \eps)M})]}\right| \\
  &= \frac{M}{2\pi^2} \log\left|1 + \frac{2\ii \pi M^{-1} \frac{e^{-s}}{S_{\wmu}'(S_{\wmu}^{-1}(e^{-s}))} + \ii D_1}{(t - s)\frac{e^{-s}}{S_{\wmu}'(S^{-1}_{\wmu}(e^{-s}))} + O((t - s) M^{-1} + (t - s)^2)} \right|\\
  &= \frac{M}{4\pi^2} \log\left(1 + \frac{4\pi^2}{(t - s)^2 M^2} + O(M^{-1} + M^{-2} (t - s)^{-1}) \right),
\end{align*}
where in the second line $r$ lies in the interval between $s$ and $t$, and in the third line $D_1 = O(M^{-2})$ is real.  Similarly, applying a second order Taylor expansion for $M_{\wmu}^{-1}$, for
\[
D_1 = \frac{1}{M_{\wmu}'(M_{\wmu}^{-1}(M_{\wrho^M}(e^{(-t + \ii \eps)M})))}\qquad \text{ and } \qquad D_2 = \frac{M_{\wmu}''(M_{\wmu}^{-1}(M_{\wrho^M}(e^{(-t + \ii \eps)M})))}{M_{\wmu}'(M_{\wmu}^{-1}(M_{\wrho^M}(e^{(-t + \ii \eps)M})))^3}
\]
and
\[
\Delta_1 = M_{\wrho^M}(e^{(-t + \ii \eps)M}) - M_{\wrho^M}(e^{(-s + \ii \eps)M}) \qquad \text{ and } \qquad \Delta_2 = M_{\wrho^M}(e^{(-t + \ii \eps)M}) - M_{\wrho^M}(e^{(-s - \ii \eps)M}),
\]
we have by application of Lemma \ref{lem:stieltjes-m} that
\begin{align*}
  K_2^M(t, s) &= - \frac{M^2}{2\pi^2} \lim_{\eps \to 0^+} \log\left|\frac{D_1 - \frac{1}{2} D_2 \Delta_1 + O\Big(|\Delta_1|^2)}{D_1 - \frac{1}{2} D_2 \Delta_2 + O\Big(|\Delta_2|^2\Big)}\right| \\
  &= \frac{M^2}{2\pi^2} \log \left|1 + \frac{- \ii \pi M^{-1} \frac{e^t}{S_{\wmu}'(S^{-1}_{\wmu}(e^{-t}))} \frac{M_{\wmu}''(M_{\wmu}^{-1}(S_{\wmu}^{-1}(e^{-t})))}{M_{\wmu}'(M_{\wmu}^{-1}(S_{\wmu}^{-1}(e^{-t})))^3} + \ii O(M^{-2})}{\frac{1}{M_{\wmu}'(M_{\wmu}^{-1}(S_{\wmu}^{-1}(e^{-t})))} + O(M^{-1})} \right|\\
  &= \frac{M^2}{4\pi^2} \log\left(1 + \pi^2 M^{-2} \frac{e^{-2t}}{S_{\wmu}'(S^{-1}_{\wmu}(e^{-t}))^2} \frac{M_{\wmu}''(M_{\wmu}^{-1}(S_{\wmu}^{-1}(e^{-t})))^2}{M_{\wmu}'(M_{\wmu}^{-1}(S_{\wmu}^{-1}(e^{-t})))^4} + O(M^{-3}) \right)\\
  &= \frac{1}{4} \frac{e^{-2t}}{S_{\wmu}'(S^{-1}_{\wmu}(e^{-t}))^2} \frac{M_{\wmu}''(M_{\wmu}^{-1}(S_{\wmu}^{-1}(e^{-t})))^2}{M_{\wmu}'(M_{\wmu}^{-1}(S_{\wmu}^{-1}(e^{-t})))^4} + O(M^{-1}). 
\end{align*}
We conclude that for any compactly supported continuous function $f$, we have
\begin{align*}
  \lim_{M \to \infty}& \int_{-\log S_{\wmu}(-1)}^{-\log S_{\wmu}(0)} \int_{-\log S_{\wmu}(-1)}^{-\log S_{\wmu}(0)} M K^M(t M, s M) f(t, s) dt ds \\ &= \lim_{M \to \infty} \left[ \int_{-\log S_{\wmu}(-1)}^{-\log S_{\wmu}(0)}\int_{-\log S_{\wmu}(-1)}^{-\log S_{\wmu}(0)}\bI_{|t - s| > M^{-1/2}} K(t, s) f(t, s) dt ds\right. \\
    &\phantom{====}+ \left. \int_{-\log S_{\wmu}(-1)}^{-\log S_{\wmu}(0)}\int_{-\log S_{\wmu}(-1)}^{-\log S_{\wmu}(0)}\bI_{|t - s| \leq M^{-1/2}}\Big( K^M_1(t, s) + K^M_2(t, s)\Big) f(t, s) dt ds \right]\\
  &= \int_{-\log S_{\wmu}(-1)}^{-\log S_{\wmu}(0)}\int_{-\log S_{\wmu}(-1)}^{-\log S_{\wmu}(0)}\bI_{|t - s| > M^{-1/2}} \Big(K(t, s) - \delta(t - s)\Big) f(t, s) dt ds\\
  &\phantom{====} + \lim_{M \to \infty} \int_{-\log S_{\wmu}(-1)}^{-\log S_{\wmu}(0)} \int_{- M^{-1/2}}^{M^{-1/2}} \bI_{s + x \in [-\log S_{\wmu}(-1), -\log S_{\wmu}(0)]} K^M_1(s + x, s) f(s + x, s) dx ds.
\end{align*}
Using the integral identity
\[
\int_{-\infty}^\infty \log\Big(1 + \frac{4\pi^2}{y^2}\Big) dy = 4\pi^2,
\]
we notice that
\begin{align*}
  \lim_{M \to \infty} &\int_{- M^{-1/2}}^{M^{-1/2}} \bI_{s + x \in [-\log S_{\wmu}(-1), -\log S_{\wmu}(0)]} K^M_1(s + x, s) f(s + x, s) dx\\
  &= \frac{M}{4\pi^2} \int_{- M^{-1/2}}^{M^{-1/2}} \bI_{s + x \in [-\log S_{\wmu}(-1), -\log S_{\wmu}(0)]} \log\Big(1 + \frac{4\pi^2}{x^2 M^2} + O(M^{-1} + M^{-2} x^{-1})\Big)  f(s + x, s) dx\\
  &= \lim_{M \to \infty} \frac{1}{4\pi^2} \int_{- M^{1/2}}^{M^{1/2}} \bI_{s + y M^{-1} \in [-\log S_{\wmu}(-1), -\log S_{\wmu}(0)]} \log\Big(1 + \frac{4\pi^2}{y^2} + O(M^{-1} + M^{-1} y^{-1})\Big)  f(s + y M^{-1}, s) dy\\
  &= f(s, s) + o(1).
\end{align*}
Substituting this into the previous expression implies that
\[
\lim_{M \to \infty} \int_{-\log S_{\wmu}(-1)}^{-\log S_{\wmu}(0)} \int_{-\log S_{\wmu}(-1)}^{-\log S_{\wmu}(0)} M K^M(t M, s M) f(t, s) dt ds = \int_{-\log S_{\wmu}(-1)}^{-\log S_{\wmu}(0)} \int_{-\log S_{\wmu}(-1)}^{-\log S_{\wmu}(0)} K(t, s) f(t, s) dt ds.
\]
This gives a limit transition between log-correlated Gaussian fields with covariance $K^M(t, s)$ and the Gaussian field with covariance $K(t, s)$, which has a white noise component.
\end{remark}

\subsection{Lyapunov exponents for Jacobi and Ginibre} \label{sec:lya-jg}

We now consider Lyapunov exponents and their fluctuations for Jacobi and Ginibre matrices.

\begin{theorem} \label{thm:jac-lya}
If the $X_N^i$ are drawn from the Jacobi ensemble with parameters $\alpha$ and $R$ so that $\alpha/N = \halpha + O(N^{-1})$ and $R/N = \hR + O(N^{-1})$, then as $M, N \to \infty$, the empirical measure $d\lambda^N$ of the Lyapunov exponents converges in probability in the sense of moments to the measure
\[
d\lambda^\infty(z) = \frac{e^z(\halpha + \hR - 1)}{(1 - e^z)^2} \bI_{[\log \frac{\halpha}{\halpha + \hR}, \log\frac{\halpha + 1}{\halpha + \hR + 1}]} dz.
\]
For the height function
\[
\cH_N(t) := \#\{\lambda^N_i \leq t\},
\]
the rescaled recentered height function $M^{1/2}(\cH_N(t) - \EE[\cH_N(t)])$ converges in the sense of moments to the Gaussian white noise on $[\log \frac{\halpha}{\halpha + \hR}, \log\frac{\halpha + 1}{\halpha + \hR + 1}]$.
\end{theorem}
\begin{proof}
By Theorem \ref{thm:lln-many}, Lemma \ref{lem:chi-comp}, and Theorem \ref{thm:jacobi}, $d\lambda^N$ converges in probability to the measure with moments
\begin{align*}
\lim_{N \to \infty} \frac{1}{N} \EE[p_k(\lambda^N)] &= \oint \log(u/(u - 1)) \left[\log \frac{\halpha + u}{\halpha + \hR + u}\right]^k \frac{du}{2\pi\ii}\\
&= \int_0^1 \left[\log \frac{\halpha + u}{\halpha + \hR + u}\right]^k du \\
&= \int_{\log \frac{\halpha}{\halpha + \hR}}^{\log\frac{\halpha + 1}{\halpha + \hR + 1}} z^k \frac{e^z(\halpha + \hR - 1)}{(1 - e^z)^2} dz,
\end{align*}
so the law of large numbers follows because $d\lambda^\infty$ is compactly supported.

For the central limit theorem, we see as in the proof of Theorem \ref{thm:lya} that
\[
H_k := M^{1/2} \int_{\log \frac{\halpha}{\halpha + \hR}}^{\log\frac{\halpha + 1}{\halpha + \hR + 1}} (\cH_N(t) - \EE[\cH_N(t)]) t^k dt
\]
are asymptotically Gaussian with covariance
\begin{align*}
\lim_{N \to \infty} \Cov(H_k, H_l) &= \oint \log(u/(u - 1)) \left[\log \frac{\halpha + u}{\halpha + \hR + u}\right]^{k + l} \frac{\hR}{(\halpha + u)(\halpha + \hR + u)} \frac{du}{2\pi\ii}\\
&= \int_0^1 \left[\log \frac{\halpha + u}{\halpha + \hR + u}\right]^{k + l} \frac{\hR}{(\halpha + u)(\halpha + \hR + u)} du\\
&= \int_{\log \frac{\halpha}{\halpha + \hR}}^{\log\frac{\halpha + 1}{\halpha + \hR + 1}} t^{k + l} dt
\end{align*}
by Theorem \ref{thm:clt-many}, Lemma \ref{lem:chi-comp}, and Theorem \ref{thm:jacobi}.
\end{proof}

\begin{theorem} \label{thm:gin-lya}
If $X_N^i = (G^i_N)^* G^i_N$ with $G^i_N$ from the Ginibre ensemble with parameter $L$ so that $L/N = \gamma + O(N^{-1})$ for $\gamma > 1$, then as $M, N \to \infty$, the empirical measure $d\lambda^N$ of the Lyapunov exponents converges in probability in the sense of moments to the measure
\[
d\lambda^\infty(z) = e^z \bI_{[\log(\gamma - 1), \log \gamma]} dz.
\]
For the height function $\cH_N(t) := \#\{\lambda_i^N \leq t\}$, the rescaled recentered height function $M^{1/2}(\cH_N(t) - \EE[\cH_N(t)])$ converges in the sense of moments to the Gaussian white noise on $[\log(\gamma - 1), \log \gamma]$.
\end{theorem}
\begin{proof}
By Theorem \ref{thm:lln-many}, Lemma \ref{lem:chi-comp}, and Theorem \ref{thm:ginibre}, $d\lambda^N$ converges in probability to the measure with moments
\begin{align*}
\lim_{N \to \infty} \frac{1}{N} \EE[p_k(\lambda^N)] &= \oint \log(u/(u - 1)) [\log (u + \gamma - 1)]^k \frac{du}{2\pi\ii}\\
&= \int_0^1 [\log(u + \gamma - 1)]^k du \\
&= \int_{\log(\gamma - 1)}^{\log \gamma} z^k e^z dz,
\end{align*}
which implies the desired because $d\lambda^\infty$ is compactly supported.

For the central limit theorem, as in the proof of Theorem \ref{thm:lya}, we see that
\[
H_k := M^{1/2} \int_{\log(\gamma - 1)}^{\log \gamma} \Big(\cH_N(t) - \EE[\cH_N(t)]\Big) t^k dt
\]
are asymptotically Gaussian with covariance
\begin{align*}
\lim_{N \to \infty} \Cov(H_k, H_l) &= \oint \log(u/(u - 1)) [\log (u + \gamma - 1)]^{k + l} \frac{1}{u + \gamma - 1} \frac{du}{2\pi\ii}\\
&= \int_0^1 [\log (u + \gamma - 1)]^{k + l} \frac{1}{u + \gamma - 1} du \\
&= \int_{\log(\gamma - 1)}^{\log \gamma} t^{k + l} dt
\end{align*}
by Theorem \ref{thm:clt-many}, Lemma \ref{lem:chi-comp}, and Theorem \ref{thm:ginibre}.
\end{proof}

\begin{remark}
While Theorem \ref{thm:gin-lya} does not apply to the case $\gamma = 1$, we extrapolate the results to this case to compare to the literature.  We see that the law of large numbers in Theorem \ref{thm:gin-lya} agrees with the triangular law on the exponentials of the Lyapunov exponents shown in \cite{IN}.  For the central limit theorem, in \cite{ABK} and \cite{For15}, it was shown for the $M$-fold product of $N \times N$ Ginibre matrices with finite $N$ and large $M$, it was shown that the $i^\text{th}$ Lyapunov exponent satisfies
\[
\lambda_i^N = \Big(\Psi(N - i + 1) - \log N\Big) + M^{-1/2} X_i + O(M^{-1}),
\]
where $X_i \sim \cN(0, \Psi'(N - i + 1))$ and $\Psi(x) = \frac{\Gamma'(x)}{\Gamma(x)}$ is the digamma function.  We will take the formal limit as $M, N \to \infty$.  For compactly supported smooth functions $f_1, f_2$ on $(-\infty, 0]$, we apply the expansion
\[
f_j(\lambda_i^N) = f_j\Big(\Psi(N - i + 1) - \log N\Big) + M^{-1/2} X_i f'_j\Big(\Psi(N - i + 1) - \log N\Big) + O(M^{-1})
\]
to find that
\begin{align*}
M  \Cov\Big(\sum_{i = 1}^N& f_1(\lambda_i^N), \sum_{j = 1}^N f_2(\lambda_j^N)\Big)\\ &= \sum_{i, j = 1}^N f_1'\Big(\Psi(N - i + 1) - \log N\Big) f_2'\Big(\Psi(N - j + 1) - \log N\Big) \Cov(X_i, X_j) + O(M^{-1/2})\\
&= \sum_{i = 1}^N f_1'\Big(\Psi(N - i + 1) - \log N\Big) f_2'\Big(\Psi(N - i + 1) - \log N\Big) \Psi'(N - i + 1) + O(M^{-1/2}).
\end{align*}
 Recalling the expansion $\Psi(x) = \log(x) - \frac{1}{2x} + O(x^{-2})$, as $N \to \infty$ we have the expansions
\begin{align*}
  \Psi(N - tN + 1) - \log N &= \log(1 - t) + O(N^{-1})\\
  \Psi'(N - t N + 1) &= N^{-1} \frac{1}{1 - t} + O(N^{-1}).
\end{align*}
This yields
\begin{align*}
  M \Cov\Big(\sum_{i = 1}^N f_1(\lambda_i^N), \sum_{j = 1}^N f_2(\lambda_j^N)\Big) &= \int_0^1 f_1'(\log(1 - t)) f_2'(\log(1 - t)) \frac{1}{1 - t} dt + O(N^{-1} + M^{-1})\\
  &= \int_{-\infty}^0 f_1'(x) f_2'(x) dx + O(N^{-1} + M^{-1}),
\end{align*}
which suggests formally that the height function at $\gamma = 1$ converges to white noise on $(-\infty, 0]$.  This heuristic computation coincides with the extrapolation of Theorem \ref{thm:gin-lya} to $\gamma = 1$.
\end{remark}

\section{A 2-D Gaussian field from products of random matrices} \label{sec:2d-prod}

\subsection{Statement of the result}

Suppose that $Y_N^1, \ldots, Y^{M}_N$ are $N \times N$ random matrices which are right unitarily invariant, and let $X_N^i := (Y^i_N)^*(Y_N^i)$.  In the setting where $M \to \infty$ and $N \to \infty$, we study the joint distribution of the Lyapunov exponents of
\[
X_{N, \alpha} := (Y_N^1 \cdots Y_N^{\lfloor \alpha M \rfloor})^* (Y_N^1 \cdots Y_N^{\lfloor \alpha M \rfloor}) \qquad \text{ for $\alpha \in (0, 1)$}. 
\]
Define the Lyapunov exponents and empirical measure for $X_{N, \alpha}$ by
\[
\lambda_i^{N, \alpha} : = \frac{1}{M} \log \mu_i^\alpha \qquad \text{ and } \qquad d\lambda^{N, \alpha} := \frac{1}{N} \sum_{i = 1}^N \delta_{\lambda_i^{N, \alpha}},
\]
where $\mu_1^\alpha \geq \cdots \geq \mu_N^\alpha$ are the eigenvalues of $X_{N, \alpha}$.  Define the height function by
\[
\cH_N(t, \alpha) := \#\{\lambda_i^{N, \alpha} \leq t\}.
\]
In Theorem \ref{thm:lya-2d}, whose proof is given in Section \ref{sec:2d-lya-proof}, we show that $\cH_N(t, \alpha)$ has a limit shape with fluctuations forming a two-dimensional Gaussian field.

\begin{theorem} \label{thm:lya-2d}
  If each $X_N^i$ has deterministic spectrum satisfying Assumption \ref{ass:measure} for a compactly supported non-atomic measure $d\rho$, then as $M, N \to \infty$, the log-spectral measure $d\lambda^{N, \alpha}$ converges in probability in the sense of moments to
  \[
  d\lambda^{\infty, \alpha} := \frac{-\alpha^{-1} e^{-\alpha^{-1} z}}{S_{\wrho}'(S_{\wrho}^{-1}(e^{-\alpha^{-1}z}))} \bI_{[-\alpha \log S_{\wrho}(-1), -\alpha \log S_{\wrho}(0)]} dz.
  \]
  The rescaled centered height function $M^{1/2} (\cH_N(t, \alpha) - \EE[\cH_N(t, \alpha)])$ converges in the sense of moments to the Gaussian random field on
  \[
  D := \{(t, \alpha) \mid t \in [-\alpha \log S_{\wrho}(-1), -\alpha \log S_{\wrho}(0)]\}
  \]
  with covariance
  \begin{align*}
  K(t, \alpha; s, \beta) &= \alpha^{-1} H\Big(S_{\wrho}^{-1}(e^{-\alpha^{-1} t}) + 1, S_{\wrho}^{-1}(e^{-\beta^{-1} s}) + 1\Big) \frac{e^{-\alpha^{-1} t} e^{-\beta^{-1} s}}{S_{\wrho}'(S^{-1}_{\wrho}(e^{-\alpha^{-1} t})) S_{\wrho}'(S^{-1}_{\wrho}(e^{-\beta^{-1} s}))}\\
  &\phantom{=} + \alpha^{-1} \delta(\alpha^{-1}t - \beta^{-1} s)\\
  H(u, w) &= \frac{1}{M_{\wrho}'(M_{\wrho}^{-1}(u - 1)) M_{\wrho}'(M_{\wrho}^{-1}(w - 1)) (M_{\wrho}^{-1}(u - 1) - M_{\wrho}^{-1}(w - 1))^2} - \frac{1}{(u - w)^2}.
\end{align*}
for $0 < \beta \leq \alpha < 1$. 
\end{theorem}

\begin{remark}
  We recall that the distribution $\delta(\alpha^{-1} t - \beta^{-1} s)$ is defined so that for any continuous function $f(t, s)$ on $\RR^2$ we have
  \[
  \int_{-\infty}^\infty \int_{-\infty}^\infty f(t, s) \delta(\alpha^{-1} t - \beta^{-1}s) dt ds = \alpha \beta \int_{-\infty}^\infty f(\alpha r, \beta r) dr.
  \]
\end{remark}

\begin{remark}
Making the change of variables $\wt = \alpha^{-1}t$ and $\ws = \beta^{-1}s$, we find that the $\alpha^{-1} \delta(\alpha^{-1}t - \beta^{-1} s)$ piece of $K(t, \alpha; s, \beta)$ becomes $\beta \delta(\wt - \ws)$, meaning that the slice of the Gaussian field in Theorem \ref{thm:lya-2d} along $D_t := \{(t\alpha, \alpha) \in D\}$ is a standard Brownian motion indexed by $\alpha$.
\end{remark}

\subsection{Multilevel LLN and CLT via multivariate Bessel generating functions}

Our proof of Theorem \ref{thm:lya-2d} is based on a generalization of Theorems \ref{thm:lln-many} and \ref{thm:clt-many} to measures whose multivariate Bessel generating functions are related.  Let $\chi_N$ be $N$-tuples such that
\[
\frac{1}{N} \sum_{i = 1}^N \delta_{\chi_{N, i} / N} \to d\chi
\]
for some compactly supported measure $d\chi$.  Suppose that we have $\chi_N$-smooth measures $d\wlambda_N^1, \ldots, d\wlambda_N^M$, with corresponding multivariate Bessel generating functions $\psi_{\chi, N}^1(s), \ldots, \psi_{\chi, N}^M(s)$ with respect to $\chi_N$.  Let $d\lambda_N^i$ be the pushfoward of $d\wlambda_N^i$ under the map $\lambda \mapsto \frac{1}{M} \lambda$.  We assume the following condition on these measures.

\begin{assump} \label{ass:branching}
There exists a $\chi_N$-smooth measure $d\sigma_N$ with corresponding multivariate Bessel generating function $\phi_{\chi, N}(s)$ so that
\[
\psi_{\chi, N}^i(s) = \phi_{\chi, N}(s)^i \text{ for $i = 1, \ldots, M$}. 
\]
\end{assump}

Under this assumption, we have the following two-dimensional LLN and CLT. 

\begin{theorem} \label{thm:lln-2d}
  If Assumption \ref{ass:branching} holds with $d\sigma_N$ being LLN-appropriate for $\chi_N$, then for any $0 < \alpha < 1$, if $x^\alpha$ is distributed according to $d\lambda_N^{\lfloor \alpha M \rfloor}$, in probability we have
  \[
  \lim_{N \to \infty} \frac{1}{N} p_k(x^\alpha) = \lim_{N \to \infty} \frac{1}{N} \EE[p_k(x^{\lfloor \alpha M \rfloor})] = \fp_k^\alpha := \alpha^k \oint \Xi(u) \Psi(u)^k \frac{du}{2\pi\ii},
  \]
  where the $u$-contour encloses $V_\chi$ and lies within $U$.  In addition, the random measures $\frac{1}{N} \sum_{i = 1}^N \delta_{x_i^\alpha}$ converge in probability to a deterministic compactly supported measure $d\mu^\alpha$ with $\int x^k d\mu^\alpha(x) = \fp_k^\alpha$. 
\end{theorem}
\begin{proof}
This follows from Theorem \ref{thm:lln-many} after correcting for a difference in scaling.
\end{proof}

\begin{theorem} \label{thm:clt-2d}
  If Assumption \ref{ass:branching} holds with $d\sigma_N$ being CLT-appropriate for $\chi_N$, then if $\{x^\alpha \in \RR^N \mid x^\alpha_1 \geq \cdots \geq x^\alpha_N\}$ is distributed according to $d\lambda_N^{\lfloor \alpha M \rfloor}$, the collection of random variables
  \[
  \{M^{1/2} (p_k(x^\alpha) - \EE[p_k(x^\alpha)])\}_{k \in \NN, \alpha \in (0, 1)}
  \]
  converges in probability to a Gaussian variable with covariance
  \begin{multline*}
    \lim_{N \to \infty} \Cov\Big(M^{1/2} p_k(x^\alpha), M^{1/2} p_l(x^\beta)\Big) = kl\, \alpha^{k - 1} \beta^{l} \oint \oint \Xi(u) \Xi(w) \Psi(u)^{k - 1} \Psi(w)^{l - 1} \Lambda(u, w) \frac{du}{2\pi \ii} \frac{dw}{2\pi\ii} \\
  + kl \alpha^{k - 1} \beta^{l} \oint \Xi(u) \Psi(u)^{k + l - 2} \Psi'(u) \frac{du}{2\pi\ii}
  \end{multline*}
  for $1 > \alpha \geq \beta > 0$, where the $u$ and $w$-contours enclose $V_\chi$ and lie within $U$, and the $u$-contour is contained inside the $w$-contour.
\end{theorem}
\begin{proof}
The proof is parallel to that of Theorem \ref{thm:clt-many}.  We therefore describe the necessary modifications.  Fix $1 > \alpha_l > \cdots > \alpha_1 > 0$, $r_t \geq 1$, and $\{k^t_s\}_{1 \leq t \leq l, 1 \leq s \leq r_t}$.  Let $\beta_i := \alpha_i - \alpha_{i - 1}$, where we use the convention $\alpha_0 \equiv 0$.  Notice that
\[
\EE\left[\prod_{t = 1}^l \prod_{s = 1}^{r_t} p_{k_s^t}\Big(x^{\alpha_t}\Big) \right]
\]
is given by the evaluation at $s = \chi_N$ of the quantity
\begin{multline*}
  \frac{\phi_{\chi, N}(s)^{M - \lfloor \alpha_l M \rfloor} D_{k_1^l} \cdots D_{k_{r_l}^l} \phi_{\chi, N}(s)^{\lfloor \alpha_l M \rfloor - \lfloor \alpha_{l - 1} M \rfloor} \cdots D_{k_1^1} \cdots D_{k_{r_1}^1} \phi_{\chi, N}(s)^{\lfloor \alpha_1 M \rfloor}}{\phi_{\chi, N}(s)^M}\\
 = \frac{D_{k_1^l} \cdots D_{k_{r_l}^l}}{\phi_{\chi, N}(s)^{\lfloor \alpha_l M \rfloor}} \phi_{\chi, N}(s)^{\lfloor \alpha_l M \rfloor} \frac{D_{k_1^{l - 1}} \cdots D_{k_{r_{l - 1}}^{l - 1}}}{\phi_{\chi, N}(s)^{\lfloor \alpha_{l - 1} M \rfloor}} \phi_{\chi, N}(s)^{\lfloor \alpha_{l - 1} M \rfloor} \cdots \frac{D_{k_1^1} \cdots D_{k_{r_1}^1}}{\phi_{\chi, N}(s)^{\lfloor \alpha_1 M \rfloor}} \phi_{\chi, N}(s)^{\lfloor \alpha_1 M \rfloor}. 
\end{multline*}
As in Section \ref{sec:graph}, we may expand this quantity into the sum of terms associated with forests on the vertex set $\{(t, s)\}$.  The coefficients of this linear combination are independent of $N$ and are monomials in $M$ and $\beta_i$ whose exponent in $M$ is at most $-W_L$, where $W_L$ is the LLN weight of the term.  Using (\ref{eq:lln-wt-calc}) as in the proof of Theorem \ref{thm:clt-many}, the exponent of $M$ in each term with non-zero coefficient to an $n^\text{th}$ mixed cumulant is at most $1 - n$.  For $n \geq 3$, this shows that any $n^\text{th}$ cumulant is the sum of finitely many terms, which of which are of order at most $O(N^{1 - n} M^{1 - n/2}) = o(1)$, as needed.

It remains to compute the covariance; it is given by modifying (\ref{eq:covar-lya-calc}) in the proof of Theorem \ref{thm:clt-many} to
\begin{multline*}
  \Cov\Big(M^{1/2} p_k(x^\alpha), M^{1/2} p_l(x^\beta)\Big) = M^{-k - l + 1} \oint \oint \Big(\Xi(u) + M \alpha \Psi(u)\Big)^k\\ \Big(\Xi(w) + M \beta \Psi(w)\Big)^l \Big(\frac{1}{(u - w)^2} + M \beta \Lambda(w, u)\Big) \frac{dw}{2\pi \ii} \frac{du}{2\pi \ii} + O(N^{-1}),
\end{multline*}
where the $u$ and $w$-contours enclose $V_\chi$ and lie within $U$, the $u$-contour is contained inside the $w$-contour, and we add a prefactor to $\Psi(u)$, $\Psi(w)$, and $\Lambda(w, u)$ to reflect which differential operator in
\[
\frac{D_k}{\phi_{\chi, N}(s)^{\lfloor \alpha M \rfloor}} \phi_{\chi, N}(s)^{\lfloor \alpha M \rfloor} \frac{D_l}{\phi_{\chi, N}(s)^{\lfloor \beta M \rfloor}} \phi_{\chi, N}(s)^{\lfloor \beta M \rfloor}
\]
each term originates from. Extracting the leading order term in $M$ as in the proof of Theorem \ref{thm:clt-many} yields the desired.
\end{proof}

\subsection{Proof of Theorem \ref{thm:lya-2d}} \label{sec:2d-lya-proof}

First, for the law of large numbers, we see by Theorem \ref{thm:unitary-inv}, Lemma \ref{lem:chi-comp}, and Theorem \ref{thm:lln-2d} that for $0 < \alpha < 1$, the rescaled log-spectral  measure of $X_{N, \alpha}$ converges in probability in the sense of moments to the measure with moments
\begin{align*}
  \lim_{N \to \infty} \frac{1}{N} \EE[p_k(\lambda_N^{\lfloor \alpha M \rfloor})] &= \alpha^k \oint \log(u/(u - 1)) [-\log S_{\wrho}(u - 1)]^k \frac{du}{2\pi \ii} \\
  &= \int_0^1 [-\alpha \log S_{\wrho}(u - 1)]^k du \\
  &= \int_{-\alpha \log S_{\wrho}(-1)}^{-\alpha \log S_{\wrho}(0)} z^k \frac{-\alpha^{-1} e^{-\alpha^{-1} z}}{S_{\wrho}'(S_{\wrho}^{-1}(e^{-\alpha^{-1}z}))} dz,
\end{align*}
where the final change of variables is legal because $d\rho$ is not a single atom, hence $S_{\wrho}(z)$ is decreasing on $[-1, 0]$ by \cite[Theorem 4.4]{HL00}.  Because this measure is compactly supported, convergence of moments implies convergence of random measures, and we see that
\[
\lim_{N \to \infty} d\lambda^{N, \alpha} = \frac{-\alpha^{-1} e^{-\alpha^{-1} z}}{S_{\wrho}'(S_{\wrho}^{-1}(e^{-\alpha^{-1}z}))} \bI_{[-\alpha \log S_{\wrho}(-1), -\alpha \log S_{\wrho}(0)]} dz.
\]
For the central limit theorem, we see that
\begin{align*}
\int_{-\alpha \log S_{\wrho}(-1)}^{-\alpha \log S_{\wrho}(0)} \cH_N(t, \alpha) t^k dt &= \frac{N [-\alpha \log S_{\wrho}(0)]^{k + 1}}{k + 1} - \int_{-\alpha \log S_{\wrho}(-1)}^{-\alpha \log S_{\wrho}(0)} \cH_N'(t, \alpha) \frac{t^{k + 1}}{k + 1} dt\\
  &= \frac{N [-\alpha \log S_{\wrho}(0)]^{k + 1}}{k + 1} - \frac{p_{k + 1}(\lambda_N^{\lfloor \alpha M \rfloor})}{k + 1}.
\end{align*}
By Theorem \ref{thm:unitary-inv}, Lemma \ref{lem:chi-comp}, and Theorem \ref{thm:clt-2d}, this implies that the quantities
\[
H_{k, \alpha} := M^{1/2} \Big(p_k(\lambda_N^{\lfloor \alpha M \rfloor}) - \EE[p_k(\lambda_N^{\lfloor \alpha M \rfloor})]\Big)
\]
are jointly Gaussian with covariance 
\begin{align*}
  \lim_{N\to \infty}& \Cov(H_{k, \alpha}, H_{l, \beta})\\ &= \alpha^k \beta^{l + 1} \oint \oint \log(u/(u - 1)) \log(w/(w - 1)) [-\log S_{\wrho}(u - 1)]^k [-\log S_{\wrho}(w - 1)]^l F^{(1, 1)}(u, w) \frac{du}{2\pi \ii} \frac{dw}{2\pi \ii} \\
  &\phantom{==} + \alpha^k \beta^{l + 1} \oint \log(u/(u - 1)) [-\log S_{\wrho}(u - 1)]^{k + l} \left[- \frac{S_{\wrho}'(u - 1)}{S_{\wrho}(u - 1)}\right] \frac{du}{2\pi \ii} \\
  &= \beta \int_0^1 \int_0^1 [-\alpha \log S_{\wrho}(u - 1)]^k [-\beta \log S_{\wrho}(w - 1)]^l F^{(1, 1)}(u, w) du dw\\
  &\phantom{==} - \beta \int_0^1  [-\alpha \log S_{\wrho}(u - 1)]^k [-\beta \log S_{\wrho}(u - 1)]^l \frac{S_{\wrho}'(u - 1)}{S_{\wrho}(u - 1)} du \\
  &= \int_{-\alpha \log S_{\wrho}(-1)}^{-\alpha \log S_{\wrho}(0)}\int_{-\beta \log S_{\wrho}(-1)}^{-\beta \log S_{\wrho}(0)} t^k s^l K(t, \alpha; s, \beta) dt ds
\end{align*}
for $0 < \beta \leq \alpha < 1$ and
\begin{align*}
K(t, \alpha; s, \beta) &= \alpha^{-1} H\Big(S_{\wrho}^{-1}(e^{-\alpha^{-1} t}) + 1, S_{\wrho}^{-1}(e^{-\beta^{-1} s}) + 1\Big) \frac{e^{-\alpha^{-1} t} e^{-\beta^{-1} s}}{S_{\wrho}'(S^{-1}_{\wrho}(e^{-\alpha^{-1} t})) S_{\wrho}'(S^{-1}_{\wrho}(e^{-\beta^{-1} s}))}\\
&\phantom{=}+ \alpha^{-1} \delta(\alpha^{-1}t - \beta^{-1} s)\\
  H(u, w) &= \frac{1}{M_{\wrho}'(M_{\wrho}^{-1}(u - 1)) M_{\wrho}'(M_{\wrho}^{-1}(w - 1)) (M_{\wrho}^{-1}(u - 1) - M_{\wrho}^{-1}(w - 1))^2} - \frac{1}{(u - w)^2}.
\end{align*}

\appendix

\section{Fluctuations of sums of random matrices} \label{sec:sum-fluc}

In this appendix, we compute the distribution of fluctuations of sums of unitarily invariant random matrices.  Let $A^i_N = U_i A U_i^*$ with $U_i$ a $N \times N$ i.i.d.~Haar unitary matrix and $A$ diagonal with real spectrum $\mu^N$, and define
\[
X_{N, M} := A_N^1 + \cdots + A_N^M.
\]

\begin{prop} \label{prop:clt-sums}
As $M \to \infty$, we have the convergence in distribution
\[
\frac{1}{\sqrt{M}}\Big(X_{N, M} - \EE[X_{N, M}]\Big) \overset{d} \implies Y^N
\]
where $Y^N$ is a centered Hermitian Gaussian random matrix with covariance
\[
  \Cov(Y^N_{ij}, Y^N_{i'j'}) := \EE[Y^N_{ij} Y^N_{i'j'}] = \begin{cases} \frac{\Delta}{N + 1} & i = j = i' = j' \\ -\frac{\Delta}{(N - 1)(N + 1)} & i = j, i' = j', i \neq i' \\ \frac{N}{(N - 1)(N + 1)} \Delta & i \neq j, i = j', j = i' \\ 0 & \text{otherwise} \end{cases},
  \]
  where $\Delta = \frac{1}{N} \sum_{i = 1}^N (\mu_i^N)^2 - \Big(\frac{1}{N} \sum_{i = 1}^N \mu^N_i\Big)^2$. 
\end{prop}
\begin{proof}
By the ordinary central limit theorem, it suffices to check that the matrix elements of $M = UAU^*$ with $U$ Haar unitary have the desired covariance.  We have that
\[
M_{ij} = \sum_{k = 1}^N u_{ik} \mu^N_k \overline{u}_{jk},
\]
which implies that
\begin{align*}
  \Cov\Big(M_{ij}, M_{i'j'}\Big) &= \sum_{k, k' = 1}^N \Cov(u_{ik} \overline{u}_{jk}, u_{i'k'} \overline{u}_{j'k'}) \mu_k^N \mu_{k'}^N\\
  \Cov\Big(M_{ij}, \overline{M}_{i'j'}\Big) &= \sum_{k, k' = 1}^N \Cov(u_{ik} \overline{u}_{jk}, u_{j'k'} \overline{u}_{i'k'}) \mu_k^N \mu_{k'}^N.
\end{align*}
By \cite[Lemma 14]{CM08}, we have that
\begin{align*}
  \EE[u_{ik} \overline{u}_{jk}] &= \delta_{ij} \frac{1}{N}\\
\EE[u_{ik}u_{i'k'} \overline{u}_{jk} \overline{u}_{j'k'}] &= \frac{1}{(N - 1)(N + 1)} [\delta_{ij} \delta_{i'j'} + \delta_{ij'} \delta_{i'j} \delta_{kk'}] - \frac{1}{(N - 1)N(N + 1)} [\delta_{ij} \delta_{i'j'} \delta_{kk'} + \delta_{ij'} \delta_{i'j}].
\end{align*}
Computing using these formulas yields the desired form of the covariance.
\end{proof}

\begin{remark}
The matrix $Y^N$ has the law of 
\[
\sqrt{\frac{N}{(N - 1)(N + 1)} \Delta} \cdot \left(X_N - \frac{1}{N}\Tr(X_N) \cdot \Id_N\right),
\]
where $X_N$ is distributed as $\text{GUE}_N$. 
\end{remark}

\section{Multivariate Bessel generating functions and the Cholesky decomposition}

Let $Y$ be a $N \times N$ random matrix which is invariant under right multiplication by $U_N$, and let $X = Y^* Y$ be the resulting Hermitian positive-definite random matrix.  If $Y = UR$ with $U$ unitary and $R$ upper triangular is the QR decomposition of $Y$, then $X = R^* R$ is the Cholesky decomposition of $X$.  The following result gives a geometric interpretation of the diagonal entries of $R$.

\begin{lemma} \label{lem:cholesky-decomp}
  If $B_k \subset \CC^N$ is a random complex $k$-dimensional unit ball and $B_{k - 1} \subset B_k$ is a random $(k - 1)$-dimensional unit ball, then
  \[
  R_{kk} \overset{d} = \frac{\vol(Y(B_k))}{\vol(Y(B_{k - 1}))} \cdot \frac{\vol(B_{k - 1})}{\vol(B_k)}.
  \]
\end{lemma}
\begin{proof}
Choose a nested sequence of random complex unit balls $B_1 \subset \cdots \subset B_k \subset \CC^N$ of increasing dimension. Notice that
\[
\left(\frac{\vol(Y(B_1))}{\vol(B_1)}, \ldots, \frac{\vol(Y(B_k))}{\vol(B_k)}\right) \overset{d} = \Big(|Yu_1|, \ldots, |Yu_1 \wedge \cdots \wedge Yu_k|\Big),
\]
where $u_1, \ldots, u_k$ are the first $k$ columns of a Haar unitary matrix.  Because $Y$ is right unitarily invariant, we notice that $(Yu_1, \ldots, Yu_k) \overset{d} = (y_1, \ldots, y_k)$, where $y_1, \ldots, y_k$ are the first $k$ columns of $Y$.  By the interpretation of QR-decomposition as Gram-Schmidt orthogonalization, we obtain that
\[
\Big(|Yu_1|, \ldots, |Yu_1 \wedge \cdots \wedge Yu_k|\Big) \overset{d} = \Big(R_{11}, \ldots, R_{11} \cdots R_{kk}\Big),
\]
which implies the desired.
\end{proof}

From Lemma \ref{lem:cholesky-decomp}, we may see that the law of the diagonal entries of the Cholesky decomposition is multiplicative over products of random matrices.

\begin{prop} \label{prop:r-mult}
  Let $Y^1 = U^1 R^1$ and $Y^2 = U^2 R^2$ be the QR decompositions of independent right unitarily-invariant matrices, and let $Y^3 = Y^1 Y^2 = U^3 R^3$ be the QR decomposition of their product. We have the equality in law
  \[
  R^3_{kk} \overset{d} = R^1_{kk} \cdot R^2_{kk}.
  \]
\end{prop}
\begin{proof}
  Let $B_k$ be a random $k$-dimensional complex ball and $B_{k - 1} \subset B_k$ a random $(k - 1)$-dimensional ball.  Because $Y^1$ is right invariant, we have that 
  \[
  R_{kk}^3 \overset{d} = \frac{\vol(Y^1 Y^2(B_k))}{\vol(Y^1 Y^2(B_{k-1}))} \cdot \frac{\vol(B_{k - 1})}{\vol(B_k)} = \left[\frac{\vol(Y^2(B_k))}{\vol(Y^2(B_{k-1}))} \cdot \frac{\vol(B_{k - 1})}{\vol(B_k)}\right] \cdot \left[\frac{\vol(Y^1(B_k'))}{\vol(Y^1(B_{k-1}'))} \cdot \frac{\vol(B_{k - 1})}{\vol(B_k)}\right],
  \]
  where $B_k'$ and $B_{k - 1}'$ are independent copies of $B_k$ and $B_{k - 1}$.  Applying Lemma \ref{lem:cholesky-decomp} completes the proof.
\end{proof}

We now give an interpretation of the multivariate Bessel generating function $\phi_X(s)$ in terms of this Cholesky decomposition which is similar to \cite[Lemma 5.3]{KK16}.  Together with Proposition \ref{prop:r-mult}, this gives an independent geometric proof of Lemma \ref{lem:mult}.

\begin{prop} \label{prop:mvb-cholesky}
  If $X$ is a Hermitian positive-definite unitarily-invariant random matrix with smooth log-spectral measure $d\mu$, then
  \[
  \phi_X(s) = \EE\left[\prod_{k = 1}^N R_{kk}^{2(s_k - \rho_k)}\right],
  \]
  where $X = R^*R$ is the Cholesky decomposition of $X$.
\end{prop}
\begin{proof}
By the HCIZ integral of \cite{H1, H2, IZ}, we have
\[
\frac{\cB(s, x)}{\cB(\rho, x)} = \frac{\Delta(\rho) \det(e^{s_i x_j})}{\Delta(s) \Delta(e^x)} = \frac{\Delta(x)}{\Delta(e^x)} \int e^{\Tr[SU\diag(x) U^*]} d\Haar_U,
\]
where $S$ is the diagonal matrix with diagonal entries $(s_1, \ldots, s_N)$.  Denote now by $\OO_x$ the (coadjoint) orbit of the conjugation action of $U_N$ on $\diag(x_1, \ldots, x_N)$ and by $d\omega_x$ the pushforward of $d\Haar_U$ to $\OO_x$ under the map $U \mapsto U\diag(x) U^*$.  Define the Gelfand-Tsetlin polytope $\GT_x$ by
\[  
\GT_x := \{y^k_i, 1 \leq i \leq k, 1 \leq k \leq N \mid y^1 \prec \cdots \prec y^N = x\},
\]
where $y^k \prec y^{k + 1}$ means that $y^{k + 1}_{i + 1} \leq y^k_i \leq y^{k + 1}_i$ and define the map $\GT: \OO_x \to \GT_x$ by
\[
\GT(Y) := (\eig(Y_k))_{1 \leq k \leq N - 1},
\]
where $Y_k$ denotes the principal $k \times k$ submatrix of $Y$. By \cite[Lemma 1.12 and Proposition 4.7]{Bar01} and \cite{GN50}, we have that
\[
\GT_*(d\omega_x) = \bI_{y^1 \prec \cdots \prec y^N = x} \frac{(N - 1)! \cdots 1!}{\Delta(x)} \prod_{k = 1}^{N - 1} \prod_{i = 1}^k dy^k_i.
\]
Applying this twice and applying the change of variables $z^k_i = e^{y^k_i}$ implies that
\begin{align*}
\frac{\cB(s, x)}{\cB(\rho, x)} &= \frac{(N - 1)! \cdots 1!}{\Delta(e^x)} \int_{y^1 \prec \cdots \prec y^N = x} e^{\sum_{k = 1}^N (s_k - s_{k + 1}) \sum_{i = 1}^k y^k_i} \prod_{k = 1}^N \prod_{i = 1}^k dy^k_i\\
&= \frac{(N - 1)! \cdots 1!}{\Delta(e^x)} e^{\sum_{i = 1}^N x_i} \int_{z^1 \prec \cdots \prec z^N = e^x} \prod_{k = 1}^N \Big(\prod_{i = 1}^k z^k_i\Big)^{s_k - s_{k + 1} - 1} \prod_{k = 1}^N \prod_{i = 1}^k dz^k_i\\
&= e^{\sum_{i = 1}^N x_i} \int \prod_{k = 1}^N \det((Ue^xU^*)_k)^{s_k - s_{k + 1} - 1} d\Haar_U,
\end{align*}
where we adopt the convention that $s_{N + 1} \equiv 0$.  Now, for $U e^x U^* = R^*R \in \OO_{e^x}$, we see that $\det((Ue^xU^*)_k) = \det(R_k)^2 = \prod_{i = 1}^k R_{ii}^2$, so rearranging the previous equation we conclude that
  \[
  \frac{\cB(s, x)}{\cB(\rho, x)} = \int \prod_{k = 1}^N R_{kk}^{2(s_k - \rho_k)} d\Haar_U.
  \]
 Therefore, if $d\mu_N$ is the log-spectral measure of $X$, we see that
\[
\phi_X(s) = \int \frac{\cB(s, x)}{\cB(\rho, x)} d\mu(x) = \int \int \prod_{k = 1}^N R_{kk}^{2 (s_k - \rho_k)} d\Haar_U d\mu(x) = \EE\left[\prod_{k = 1}^N R_{kk}^{2(s_k - \rho_k)} \right],
\]
where $R^*R = X$ is the Cholesky decomposition of $X$. 
\end{proof}

We now combine Proposition \ref{prop:mvb-cholesky} with Theorem \ref{thm:unitary-inv} to obtain a geometric interpretation of the $S$-transform.  Let $d\lambda$ be a compactly supported measure on $(0, \infty)$, and for a fixed compact set $I \supset \supp d\lambda$, let $\lambda^N \in \{x \in I^N \mid x_1 \geq \cdots \geq x_N > 0\}$ be a sequence such that we have the weak convergence of measures
\[
\frac{1}{N} \sum_{i = 1}^N \delta_{\lambda^N} \to d\lambda.
\]

\begin{corr} \label{corr:s-interp}
  Let $X_N$ be the $N \times N$ unitarily invariant Hermitian random matrix with spectrum $\lambda^N$.  For $t \in [0, 1]$,  the log-$S$-transform of the measure $d\lambda$ is given by
  \[
 -\log S_{d\lambda}(t - 1) = \lim_{N \to \infty} \EE\left[ 2 \log R_{\lfloor t N \rfloor, \lfloor t N \rfloor}\right],
 \]
 where $X_N = R^* R$ is the Cholesky decomposition of $X_N$. 
\end{corr}
\begin{proof}
  First, by Proposition \ref{prop:mvb-cholesky}, we find that
  \[
  \partial_{s_k} \log \phi_{X_N}(\rho) = \partial_{s_k} \phi_{X_N}(\rho) = \int \partial_{s_k} \left.\left[\prod_{k = 1}^N R_{kk}^{2(s_k - \rho_k)}\right]\right|_{s = \rho} d\Haar_U = \int 2 \log R_{kk}\, d\Haar_U,
  \]
  where $X_N = R^*R$ is the Cholesky decomposition.  Applying Theorem \ref{thm:unitary-inv}, we have uniformly in $k \in \{1, \ldots, N\}$ that
  \[
  \lim_{N \to \infty} \left|\int 2 \log R_{kk} d\Haar_U + \log S_{d\lambda}(k/N - 1)\right| = 0,
  \]
  from which the conclusion follows by continuity of the $S$-transform on $[-1, 0]$.
\end{proof}

\bibliographystyle{alpha}
\bibliography{rp-bib}

\newcommand{\etalchar}[1]{$^{#1}$}
\begin{thebibliography}{XBSD{\etalchar{+}}18}

\bibitem[ABK14]{ABK}
G.~Akemann, Z.~Burda, and M.~Kieburg.
\newblock Universal distribution of {L}yapunov exponents for products of
  {G}inibre matrices.
\newblock {\em J. Phys. A}, 47(39):395202, 35, 2014.

\bibitem[ABK18]{ABK18}
G.~Akemann, Z.~Burda, and M.~Kieburg.
\newblock From integrable to chaotic systems: Universal local statistics of
  {L}yapunov exponents.
\newblock {\em ArXiv e-prints}, 2018.

\bibitem[AKW13]{AKW13}
G.~Akemann, M.~Kieburg, and L.~Wei.
\newblock Singular value correlation functions for products of {W}ishart random
  matrices.
\newblock {\em J. Phys. A}, 46(27):275205, 22, 2013.

\bibitem[AM18]{AM18}
O.~Arizmendi and J.~A. Mingo.
\newblock Second order cumulants: second order even elements and {$R$}-diagonal
  elements.
\newblock {\em In preparation}, 2018.

\bibitem[Bar01]{Bar01}
Y.~Baryshnikov.
\newblock G{UE}s and queues.
\newblock {\em Probab. Theory Related Fields}, 119(2):256--274, 2001.

\bibitem[BB07]{BB07}
S.~T. Belinschi and H.~Bercovici.
\newblock A new approach to subordination results in free probability.
\newblock {\em J. Anal. Math.}, 101:357--365, 2007.

\bibitem[BBCF17]{BBCF17}
S.~T. Belinschi, H.~Bercovici, M.~Capitaine, and M.~F\'{e}vrier.
\newblock Outliers in the spectrum of large deformed unitarily invariant
  models.
\newblock {\em Ann. Probab.}, 45(6A):3571--3625, 2017.

\bibitem[BD17]{BD17}
J.~Breuer and M.~Duits.
\newblock Central limit theorems for biorthogonal ensembles and asymptotics of
  recurrence coefficients.
\newblock {\em J. Amer. Math. Soc.}, 30(1):27--66, 2017.

\bibitem[BG56]{BG56}
Felix~Aleksandrovich Berezin and Izrail~Moiseevich Gel'fand.
\newblock Some remarks on the theory of spherical functions on symmetric
  riemannian manifolds.
\newblock {\em Trudy Moskovskogo Matematicheskogo Obshchestva}, 5:311--351,
  1956.

\bibitem[BG15a]{BG15}
A.~Borodin and V.~Gorin.
\newblock General {$\beta$}-{J}acobi corners process and the {G}aussian free
  field.
\newblock {\em Comm. Pure Appl. Math.}, 68(10):1774--1844, 2015.

\bibitem[BG15b]{BG15L}
A.~Bufetov and V.~Gorin.
\newblock Representations of classical {L}ie groups and quantized free
  convolution.
\newblock {\em Geometric and Functional Analysis}, 25(3):763--814, 2015.

\bibitem[BG18a]{BG16}
A.~Bufetov and V.~Gorin.
\newblock Fluctuations of particle systems determined by {S}chur generating
  functions.
\newblock {\em Adv. Math.}, 338:702--781, 2018.

\bibitem[BG18b]{BG17}
A.~Bufetov and V.~Gorin.
\newblock Fourier transform on high-dimensional unitary groups with
  applications to random tilings.
\newblock 2018+.
\newblock \url{arXiv:1712.09925}.

\bibitem[BGS18]{BGS18}
A.~{Borodin}, V.~{Gorin}, and E.~{Strahov}.
\newblock {Product matrix processes as limits of random plane partitions}.
\newblock {\em Int. Math. Res. Not.}, to appear, 2018.
\newblock \url{arXiv:1806.10855}.

\bibitem[BH96]{BH96}
E.~Br\'{e}zin and S.~Hikami.
\newblock Correlations of nearby levels induced by a random potential.
\newblock {\em Nuclear Phys. B}, 479(3):697--706, 1996.

\bibitem[BH97]{BH97}
E.~Br\'{e}zin and S.~Hikami.
\newblock Spectral form factor in a random matrix theory.
\newblock {\em Phys. Rev. E (3)}, 55(4):4067--4083, 1997.

\bibitem[Bia98]{Bia98}
P.~Biane.
\newblock Processes with free increments.
\newblock {\em Math. Z.}, 227(1):143--174, 1998.

\bibitem[BL17]{BL17}
C.~{Boutillier} and Z.~{Li}.
\newblock {Limit shape and height fluctuations of random perfect matchings on
  square-hexagon lattices}.
\newblock {\em ArXiv e-prints}, September 2017.

\bibitem[BLP16]{BLP16}
Z.~Bai, H.~Li, and G.~Pan.
\newblock {Central limit theorem for linear spectral statistics of large
  dimensional separable sample covariance matrices}.
\newblock {\em ArXiv e-prints}, November 2016.

\bibitem[Bor14a]{Bor14}
A.~Borodin.
\newblock C{LT} for spectra of submatrices of {W}igner random matrices.
\newblock {\em Mosc. Math. J.}, 14(1):29--38, 170, 2014.

\bibitem[Bor14b]{Bor14b}
A.~Borodin.
\newblock C{LT} for spectra of submatrices of {W}igner random matrices, {II}:
  {S}tochastic evolution.
\newblock In {\em Random matrix theory, interacting particle systems, and
  integrable systems}, volume~65 of {\em Math. Sci. Res. Inst. Publ.}, pages
  57--69. Cambridge Univ. Press, New York, 2014.

\bibitem[BV92]{BV92}
H.~Bercovici and D.~Voiculescu.
\newblock L\'evy-{H}in\v cin type theorems for multiplicative and additive free
  convolution.
\newblock {\em Pacific J. Math.}, 153(2):217--248, 1992.

\bibitem[Car14]{Car14}
F.~Carlson.
\newblock {\em Sur une class des s\'eries de {T}aylor}.
\newblock PhD thesis, Upsala, 1914.

\bibitem[CG18]{CG18}
C.~Cuenca and V.~Gorin.
\newblock {$q$-deformed character theory for infinite-dimensional symplectic
  and orthogonal groups}.
\newblock {\em ArXiv e-prints}, December 2018.

\bibitem[CM08]{CM08}
S.~Chatterjee and E.~Meckes.
\newblock Multivariate normal approximation using exchangeable pairs.
\newblock {\em ALEA Lat. Am. J. Probab. Math. Stat.}, 4:257--283, 2008.

\bibitem[CMR17]{CMR17}
F.~{Comets}, G.~R. {Moreno Flores}, and A.~F. {Ramirez}.
\newblock {Random polymers on the complete graph}.
\newblock {\em ArXiv e-prints}, page arXiv:1707.01588, July 2017.

\bibitem[CMSS07]{CMSS}
B.~Collins, J.~A. Mingo, P.~\'{S}niady, and R.~Speicher.
\newblock Second order freeness and fluctuations of random matrices. {III}.
  {H}igher order freeness and free cumulants.
\newblock {\em Doc. Math.}, 12:1--70, 2007.

\bibitem[CO18]{COR18}
N.~{Coston} and S.~{O'Rourke}.
\newblock {Gaussian fluctuations for linear eigenvalue statistics of products
  of independent iid random matrices}.
\newblock {\em ArXiv e-prints}, September 2018.

\bibitem[Cop65]{Cop65}
E.~T. Copson.
\newblock {\em Asymptotic expansions}.
\newblock Cambridge Tracts in Mathematics and Mathematical Physics, No. 55.
  Cambridge University Press, New York, 1965.

\bibitem[CPS18]{CPS18}
M.~Chen, J.~Pennington, and S.~Schoenholz.
\newblock Dynamical isometry and a mean field theory of {RNN}s: Gating enables
  signal propagation in recurrent neural networks.
\newblock In Jennifer Dy and Andreas Krause, editors, {\em Proceedings of the
  35th International Conference on Machine Learning}, volume~80 of {\em
  Proceedings of Machine Learning Research}, pages 873--882, Stockholmsmässan,
  Stockholm Sweden, 10--15 Jul 2018. PMLR.

\bibitem[CPV93]{CPV93}
A.~Crisanti, G.~Paladin, and A.~Vulpiani.
\newblock {\em Products of random matrices in statistical physics}, volume 104
  of {\em Springer Series in Solid-State Sciences}.
\newblock Springer-Verlag, Berlin, 1993.
\newblock With a foreword by Giorgio Parisi.

\bibitem[Cue18a]{Cue18}
C.~Cuenca.
\newblock Asymptotic formulas for {M}acdonald polynomials and the boundary of
  the {$(q,t)$}-{G}elfand-{T}setlin graph.
\newblock {\em SIGMA Symmetry Integrability Geom. Methods Appl.}, 14:Paper No.
  001, 66, 2018.

\bibitem[Cue18b]{Cue18b}
C.~Cuenca.
\newblock Pieri integral formula and asymptotics of {J}ack unitary characters.
\newblock {\em Selecta Math. (N.S.)}, 24(3):2737--2789, 2018.

\bibitem[{Cue}18c]{Cue18c}
C.~{Cuenca}.
\newblock Universal behavior of the corners of orbital beta processes.
\newblock {\em ArXiv e-prints}, July 2018.

\bibitem[Dei17]{Dei17}
P.~Deift.
\newblock Some open problems in random matrix theory and the theory of
  integrable systems. {II}.
\newblock {\em SIGMA Symmetry Integrability Geom. Methods Appl.}, 13:Paper No.
  016, 23, 2017.

\bibitem[DJ18]{DJ18}
M.~Duits and K.~Johansson.
\newblock On mesoscopic equilibrium for linear statistics in {D}yson's
  {B}rownian motion.
\newblock {\em Mem. Amer. Math. Soc.}, 255(1222):v+118, 2018.

\bibitem[DP12]{DP12}
I.~Dumitriu and E.~Paquette.
\newblock Global fluctuations for linear statistics of {$\beta$}-{J}acobi
  ensembles.
\newblock {\em Random Matrices Theory Appl.}, 1(4):1250013, 60, 2012.

\bibitem[DP18]{DP18}
I.~Dumitriu and E.~Paquette.
\newblock Spectra of overlapping {W}ishart matrices and the {G}aussian free
  field.
\newblock {\em Random Matrices Theory Appl.}, 7(2):1850003, 21, 2018.

\bibitem[Erd56]{Erd56}
A.~Erd\'{e}lyi.
\newblock {\em Asymptotic expansions}.
\newblock Dover Publications, Inc., New York, 1956.

\bibitem[FK60]{FK60}
H.~Furstenberg and H.~Kesten.
\newblock Products of random matrices.
\newblock {\em Ann. Math. Statist.}, 31:457--469, 1960.

\bibitem[For10]{For}
P.~Forrester.
\newblock {\em Log-gases and random matrices}, volume~34 of {\em London
  Mathematical Society Monographs Series}.
\newblock Princeton University Press, Princeton, NJ, 2010.

\bibitem[For13]{For13}
P.~J. Forrester.
\newblock Lyapunov exponents for products of complex {G}aussian random
  matrices.
\newblock {\em J. Stat. Phys.}, 151(5):796--808, 2013.

\bibitem[For15]{For15}
P.~J. Forrester.
\newblock Asymptotics of finite system {L}yapunov exponents for some random
  matrix ensembles.
\newblock {\em J. Phys. A}, 48(21):215205, 17, 2015.

\bibitem[GM17]{GM17}
V.~{Gorin} and A.~W. {Marcus}.
\newblock {Crystallization of random matrix orbits}.
\newblock {\em Int. Math. Res. Not.}, to appear, 2017.

\bibitem[GN50]{GN50}
I.~M. Gelfand and M.~A. Na\u{\i}mark.
\newblock {\em Unitarnye predstavleniya klassi\v{c}eskih grupp}.
\newblock Trudy Mat. Inst. Steklov., vol. 36. Izdat. Nauk SSSR,
  Moscow-Leningrad, 1950.

\bibitem[GN15]{GN15}
A.~Guionnet and J.~Novak.
\newblock Asymptotics of unitary multimatrix models: the {S}chwinger-{D}yson
  lattice and topological recursion.
\newblock {\em J. Funct. Anal.}, 268(10):2851--2905, 2015.

\bibitem[Gor14]{Gor14}
V.~Gorin.
\newblock From alternating sign matrices to the {G}aussian unitary ensemble.
\newblock {\em Comm. Math. Phys.}, 332(1):437--447, 2014.

\bibitem[GP15]{GP}
V.~Gorin and G.~Panova.
\newblock Asymptotics of symmetric polynomials with applications to statistical
  mechanics and representation theory.
\newblock {\em Ann. Probab.}, 43(6):3052--3132, 2015.

\bibitem[GR84]{GR84}
H.~Grauert and R.~Remmert.
\newblock {\em Coherent analytic sheaves}, volume 265 of {\em Grundlehren der
  Mathematischen Wissenschaften [Fundamental Principles of Mathematical
  Sciences]}.
\newblock Springer-Verlag, Berlin, 1984.

\bibitem[HC57a]{H1}
Harish-Chandra.
\newblock Differential operators on a semisimple {L}ie algebra.
\newblock {\em Amer. {J}. {M}ath.}, 79:87--120, 1957.

\bibitem[HC57b]{H2}
Harish-Chandra.
\newblock Fourier transforms on a semisimple {L}ie algebra.
\newblock {\em Amer. {J}. {M}ath.}, 79:193--257, 1957.

\bibitem[Hel84]{Hel84}
S.~Helgason.
\newblock {\em Groups and geometric analysis}, volume 113 of {\em Pure and
  Applied Mathematics}.
\newblock Academic Press, Inc., Orlando, FL, 1984.
\newblock Integral geometry, invariant differential operators, and spherical
  functions.

\bibitem[HL00]{HL00}
U.~Haagerup and F.~Larsen.
\newblock Brown's spectral distribution measure for {$R$}-diagonal elements in
  finite von {N}eumann algebras.
\newblock {\em J. Funct. Anal.}, 176(2):331--367, 2000.

\bibitem[HL16]{HL16}
J.~{Huang} and B.~{Landon}.
\newblock {Local law and mesoscopic fluctuations of {D}yson {B}rownian {M}otion
  for general $\beta$ and potential}.
\newblock {\em ArXiv e-prints}, page arXiv:1612.06306, December 2016.

\bibitem[HN18]{HN18}
B.~{Hanin} and M.~{Nica}.
\newblock {Products of Many Large Random Matrices and Gradients in Deep Neural
  Networks}.
\newblock {\em arXiv e-prints}, page arXiv:1812.05994, December 2018.

\bibitem[{Hua}18]{Hua18}
J.~{Huang}.
\newblock Law of large numbers and central limit theorems by {J}ack generating
  functions.
\newblock {\em ArXiv e-prints}, July 2018.

\bibitem[IN92]{IN}
M.~Isopi and C.~M. Newman.
\newblock The triangle law for {L}yapunov exponents of large random matrices.
\newblock {\em Comm. Math. Phys.}, 143(3):591--598, 1992.

\bibitem[IZ80]{IZ}
C.~Itzykson and J.~Zuber.
\newblock The planar approximation. ii.
\newblock {\em J. Math. Phys.}, 21(3):411–421, 1980.

\bibitem[Jia13]{Jia13}
T.~Jiang.
\newblock Limit theorems for beta-{J}acobi ensembles.
\newblock {\em Bernoulli}, 19(3):1028--1046, 2013.

\bibitem[JL18]{JL18}
K.~Johansson and G.~Lambert.
\newblock Gaussian and non-{G}aussian fluctuations for mesoscopic linear
  statistics in determinantal processes.
\newblock {\em Ann. Probab.}, 46(3):1201--1278, 2018.

\bibitem[Joh98]{Joh98}
K.~Johansson.
\newblock On fluctuations of eigenvalues of random {H}ermitian matrices.
\newblock {\em Duke Math. J.}, 91(1):151--204, 1998.

\bibitem[Joh04]{Joh04}
K.~Johansson.
\newblock Determinantal processes with number variance saturation.
\newblock {\em Comm. Math. Phys.}, 252(1-3):111--148, 2004.

\bibitem[Kar08]{Kar08}
V.~Kargin.
\newblock Lyapunov exponents of free operators.
\newblock {\em J. Funct. Anal.}, 255(8):1874--1888, 2008.

\bibitem[KK16a]{KK16a}
M.~Kieburg and H.~K\"{o}sters.
\newblock Exact relation between singular value and eigenvalue statistics.
\newblock {\em Random Matrices Theory Appl.}, 5(4):1650015, 57, 2016.

\bibitem[KK16b]{KK16}
M.~{Kieburg} and H.~{K{\"o}sters}.
\newblock Products of random matrices from polynomial ensembles.
\newblock {\em ArXiv e-prints}, January 2016.

\bibitem[KZ14]{KZ14}
A.~B.~J. Kuijlaars and L.~Zhang.
\newblock Singular values of products of {G}inibre random matrices, multiple
  orthogonal polynomials and hard edge scaling limits.
\newblock {\em Comm. Math. Phys.}, 332(2):759--781, 2014.

\bibitem[{Li}18a]{Li18b}
Z.~{Li}.
\newblock {Fluctuations of dimer heights on contracting square-hexagon
  lattices}.
\newblock {\em ArXiv e-prints}, September 2018.

\bibitem[{Li}18b]{Li18a}
Z.~{Li}.
\newblock {Schur function at generic points and limit shape of perfect
  matchings on contracting square hexagon lattices with piecewise boundary
  conditions}.
\newblock {\em ArXiv e-prints}, July 2018.

\bibitem[LQ18]{LQ18}
Z.~{Ling} and R.~C. {Qiu}.
\newblock {Spectrum concentration in deep residual learning: a free probability
  appproach}.
\newblock {\em arXiv e-prints}, page arXiv:1807.11694, July 2018.

\bibitem[LWW18]{LWW18}
D.-Z. {Liu}, D.~{Wang}, and Y.~{Wang}.
\newblock {Lyapunov exponent, universality and phase transition for products of
  random matrices}.
\newblock {\em ArXiv e-prints}, September 2018.

\bibitem[Mac95]{Mac}
I.~Macdonald.
\newblock {\em Symmetric functions and {H}all polynomials}.
\newblock Oxford Mathematical Monographs. The Clarendon Press, Oxford
  University Press, New York, second edition, 1995.
\newblock With contributions by A. Zelevinsky, Oxford Science Publications.

\bibitem[MN18]{MN18}
S.~Matsumoto and J.~Novak.
\newblock A moment method for invariant ensembles.
\newblock {\em Electron. Res. Announc. Math. Sci.}, 2018.

\bibitem[MS06]{MS06}
J.~A. Mingo and R.~Speicher.
\newblock Second order freeness and fluctuations of random matrices. {I}.
  {G}aussian and {W}ishart matrices and cyclic {F}ock spaces.
\newblock {\em J. Funct. Anal.}, 235(1):226--270, 2006.

\bibitem[MS17]{MS17}
J.~A. Mingo and R.~Speicher.
\newblock {\em Free probability and random matrices}, volume~35 of {\em Fields
  Institute Monographs}.
\newblock Springer, New York; Fields Institute for Research in Mathematical
  Sciences, Toronto, ON, 2017.

\bibitem[MSS07]{MSS07}
J.~A. Mingo, P.~\'{S}niady, and R.~Speicher.
\newblock Second order freeness and fluctuations of random matrices. {II}.
  {U}nitary random matrices.
\newblock {\em Adv. Math.}, 209(1):212--240, 2007.

\bibitem[New86a]{New86}
C.~M. Newman.
\newblock The distribution of {L}yapunov exponents: exact results for random
  matrices.
\newblock {\em Comm. Math. Phys.}, 103(1):121--126, 1986.

\bibitem[New86b]{New84}
C.~M. Newman.
\newblock Lyapunov exponents for some products of random matrices: exact
  expressions and asymptotic distributions.
\newblock In {\em Random matrices and their applications ({B}runswick, {M}aine,
  1984)}, volume~50 of {\em Contemp. Math.}, pages 121--141. Amer. Math. Soc.,
  Providence, RI, 1986.

\bibitem[NS97]{NS97}
A.~Nica and R.~Speicher.
\newblock A ``{F}ourier transform'' for multiplicative functions on
  non-crossing partitions.
\newblock {\em J. Algebraic Combin.}, 6(2):141--160, 1997.

\bibitem[Oko01]{Oko01}
A.~Okounkov.
\newblock Infinite wedge and random partitions.
\newblock {\em Selecta Math. (N.S.)}, 7(1):57--81, 2001.

\bibitem[Ose68]{Ose68}
V.~I. Oseledec.
\newblock A multiplicative ergodic theorem. {C}haracteristic {L}japunov,
  exponents of dynamical systems.
\newblock {\em Trudy Moskov. Mat. Ob\v{s}\v{c}.}, 19:179--210, 1968.

\bibitem[OV96]{OV96}
G.~Olshanski and A.~Vershik.
\newblock Ergodic unitarily invariant measures on the space of infinite
  {H}ermitian matrices.
\newblock In {\em Contemporary mathematical physics}, volume 175 of {\em Amer.
  Math. Soc. Transl. Ser. 2}, pages 137--175. Amer. Math. Soc., Providence, RI,
  1996.

\bibitem[Pil05]{Pil05}
J.~Pila.
\newblock Note on {C}arlson's theorem.
\newblock {\em Rocky Mountain J. Math.}, 35(6):2107--2112, 2005.

\bibitem[PR04]{PR04}
D.~Petz and J.~R\'{e}ffy.
\newblock On asymptotics of large {H}aar distributed unitary matrices.
\newblock {\em Period. Math. Hungar.}, 49(1):103--117, 2004.

\bibitem[PS11]{PS11}
L.~Pastur and M.~Shcherbina.
\newblock {\em Eigenvalue distribution of large random matrices}, volume 171 of
  {\em Mathematical Surveys and Monographs}.
\newblock American Mathematical Society, Providence, RI, 2011.

\bibitem[PSG17]{PSG17}
J.~Pennington, S.~Schoenholz, and S.~Ganguli.
\newblock Resurrecting the sigmoid in deep learning through dynamical isometry:
  theory and practice.
\newblock In I.~Guyon, U.~V. Luxburg, S.~Bengio, H.~Wallach, R.~Fergus,
  S.~Vishwanathan, and R.~Garnett, editors, {\em Advances in Neural Information
  Processing Systems 30}, pages 4785--4795. Curran Associates, Inc., 2017.

\bibitem[PSG18]{PSG18}
J.~Pennington, S.~Schoenholz, and S.~Ganguli.
\newblock The emergence of spectral universality in deep networks.
\newblock In {\em International Conference on Artificial Intelligence and
  Statistics, {AISTATS} 2018, 9-11 April 2018, Playa Blanca, Lanzarote, Canary
  Islands, Spain}, pages 1924--1932, 2018.

\bibitem[PV07]{PV07}
L.~Pastur and V.~Vasilchuk.
\newblock On the law of addition of random matrices: covariance and the central
  limit theorem for traces of resolvent.
\newblock In {\em Probability and mathematical physics}, volume~42 of {\em CRM
  Proc. Lecture Notes}, pages 399--416. Amer. Math. Soc., Providence, RI, 2007.

\bibitem[Rag79]{Rag79}
M.~S. Raghunathan.
\newblock A proof of {O}seledec's multiplicative ergodic theorem.
\newblock {\em Israel J. Math.}, 32(4):356--362, 1979.

\bibitem[{Red}16]{Red16}
N.~K. {Reddy}.
\newblock {Lyapunov exponents and eigenvalues of products of random matrices}.
\newblock {\em ArXiv e-prints}, June 2016.

\bibitem[SMG14]{SMG14}
A.~M. Saxe, J.~L. McClelland, and S.~Ganguli.
\newblock Exact solutions to the nonlinear dynamics of learning in deep linear
  neural networks.
\newblock In {\em International Conference on Learning Representations (ICLR)},
  2014.

\bibitem[Tuc10]{Tuc}
G.~Tucci.
\newblock Limits laws for geometric means of free random variables.
\newblock {\em Indiana Univ. Math. J.}, 59(1):1--13, 2010.

\bibitem[TWJ{\etalchar{+}}18]{TWJTN}
W.~{Tarnowski}, P.~{Warchol}, S.~{Jastrzebski}, J.~{Tabor}, and M.~A. {Nowak}.
\newblock {Dynamical Isometry is Achieved in Residual Networks in a Universal
  Way for any Activation Function}.
\newblock {\em arXiv e-prints}, page arXiv:1809.08848, September 2018.

\bibitem[Vas16]{Vas16}
V.~Vasilchuk.
\newblock On the fluctuations of eigenvalues of multiplicative deformed unitary
  invariant ensembles.
\newblock {\em Random Matrices Theory Appl.}, 5(2):1650007, 28, 2016.

\bibitem[Voi85]{Voi85}
D.~Voiculescu.
\newblock Symmetries of some reduced free product {$C^\ast$}-algebras.
\newblock In {\em Operator algebras and their connections with topology and
  ergodic theory ({B}u\c{s}teni, 1983)}, volume 1132 of {\em Lecture Notes in
  Math.}, pages 556--588. Springer, Berlin, 1985.

\bibitem[Voi87]{Voi87}
D.~Voiculescu.
\newblock Multiplication of certain noncommuting random variables.
\newblock {\em J. Operator Theory}, 18(2):223--235, 1987.

\bibitem[Voi93]{Voi93}
D.~Voiculescu.
\newblock The analogues of entropy and of {F}isher's information measure in
  free probability theory. {I}.
\newblock {\em Comm. Math. Phys.}, 155(1):71--92, 1993.

\bibitem[XBSD{\etalchar{+}}18]{XBSSP18}
L.~Xiao, Y.~Bahri, J.~Sohl-Dickstein, S.~Schoenholz, and J.~Pennington.
\newblock Dynamical isometry and a mean field theory of {CNN}s: How to train
  10,000-layer vanilla convolutional neural networks.
\newblock In Jennifer Dy and Andreas Krause, editors, {\em Proceedings of the
  35th International Conference on Machine Learning}, volume~80 of {\em
  Proceedings of Machine Learning Research}, pages 5393--5402,
  Stockholmsmässan, Stockholm Sweden, 10--15 Jul 2018. PMLR.

\bibitem[ZJ99]{Zin99}
P.~Zinn-Justin.
\newblock Adding and multiplying random matrices: a generalization of
  {V}oiculescu's formulas.
\newblock {\em Phys. Rev. E (3)}, 59(5, part A):4884--4888, 1999.

\end{thebibliography}
\end{document}